\documentclass[review,onefignum,onetabnum]{siamart190516}

\usepackage{braket,amsfonts}
\usepackage[mathlines]{lineno}
\linenumbers

\usepackage{array}

\usepackage[caption=false]{subfig}
\usepackage{pgfplots}

\newsiamthm{claim}{Claim}
\newsiamremark{remark}{Remark}
\newsiamremark{hypothesis}{Hypothesis}
\crefname{hypothesis}{Hypothesis}{Hypotheses}

\usepackage{algorithmic}

\usepackage{graphicx,epstopdf}

\Crefname{ALC@unique}{Line}{Lines}

\usepackage{amsopn}

\usepackage{xspace}
\usepackage{bold-extra}
\usepackage[most]{tcolorbox}

\colorlet{texcscolor}{blue!50!black}
\colorlet{texemcolor}{red!70!black}
\colorlet{texpreamble}{red!70!black}
\colorlet{codebackground}{black!25!white!25}

\lstdefinestyle{siamlatex}{%
  style=tcblatex,
  texcsstyle=*\color{texcscolor},
  texcsstyle=[2]\color{texemcolor},
  keywordstyle=[2]\color{texemcolor},
  moretexcs={cref,Cref,maketitle,mathcal,text,headers,email,url},
}

\tcbset{%
  colframe=black!75!white!75,
  coltitle=white,
  colback=codebackground, % bottom/left side
  colbacklower=white, % top/right side
  fonttitle=\bfseries,
  arc=0pt,outer arc=0pt,
  top=1pt,bottom=1pt,left=1mm,right=1mm,middle=1mm,boxsep=1mm,
  leftrule=0.3mm,rightrule=0.3mm,toprule=0.3mm,bottomrule=0.3mm,
  listing options={style=siamlatex}
}

\newtcblisting[use counter=example]{example}[2][]{%
  title={Example~\thetcbcounter: #2},#1}
\newcommand{\ceil}[1]{\left\lceil #1 \right\rceil}

\newtcbinputlisting[use counter=example]{\examplefile}[3][]{%
  title={Example~\thetcbcounter: #2},listing file={#3},#1}

\DeclareTotalTCBox{\code}{ v O{} }
{ %fontupper=\ttfamily\color{texemcolor},
  fontupper=\ttfamily\color{black},
  nobeforeafter,
  tcbox raise base,
  colback=codebackground,colframe=white,
  top=0pt,bottom=0pt,left=0mm,right=0mm,
  leftrule=0pt,rightrule=0pt,toprule=0mm,bottomrule=0mm,
  boxsep=0.5mm,
  #2}{#1}

\patchcmd\newpage{\vfil}{}{}{}
\usepackage{braket,amsfonts}
\usepackage{array}

\allowdisplaybreaks[1]

\usepackage[caption=false]{subfig}
\crefname{section}{Section}{Sections} % force cref for section to have capital S 
\crefname{subsection}{Section}{Sections}

\Crefname{ALC@unique}{Step}{Steps}

\usepackage{algorithmic}

\usepackage{graphicx,epstopdf}
\usepackage{amsopn}
\usepackage{amssymb}
\usepackage{multirow}

\usepackage{bbm}
\usepackage{hyperref}
\usepackage{booktabs}
\usepackage{multirow}
\usepackage{graphicx}           
\usepackage{caption}
\usepackage{amsfonts}
\usepackage{mathdots}
\usepackage{amsmath}

\usepackage{amssymb}
\usepackage{glossaries} 
\usepackage{mathtools}
\usepackage{enumerate} 
\usepackage{stmaryrd} 
\usepackage{xcolor} 

\newcommand{\argmin}{\operatornamewithlimits{argmin}}

\DeclarePairedDelimiter\abs{\lvert}{\rvert}%
\DeclarePairedDelimiter\norm{\lVert}{\rVert}%

\newcommand{\nys}{Nystr{\"o}m }
\newcommand{\alphahat}{\hat{\alpha}}
\newcommand{\betahat}{\hat{\beta}}

\newcommand{\DNOA}{D^*_{N,\Omega}(\annt)}
\newcommand{\annt}{\mathcal{A}}
\newcommand{\anntXp}{\mathcal{A}_{X}}
\newcommand{\anntpi}{\mathcal{A}^{(i)}}

\newcommand{\dist}[1]{\text{dist}( #1 )}
\newcommand{\distinf}[1]{\text{dist}_{\infty}( #1 )}

\makeatletter
\let\oldabs\abs
\def\abs{\@ifstar{\oldabs}{\oldabs*}}
\let\oldnorm\norm
\def\norm{\@ifstar{\oldnorm}{\oldnorm*}}
\makeatother

\nolinenumbers

\title{Fast Deterministic Approximation of Symmetric Indefinite Kernel Matrices with High Dimensional Datasets}

\author{Difeng Cai\thanks{Department of Mathematics, Emory University, Atlanta, GA 30322 (\email{dcai7@emory.edu}, \email{jnagy@emory.edu},\email{yxi26@emory.edu}).The research of Difeng Cai and Yuanzhe Xi are supported by NSF award  OAC 2003720 and the research of James Nagy is supported by NSF award DMS 1819042.}
\and James Nagy$^\dagger$
\and Yuanzhe Xi$^\dagger$}

\headers{D. Cai, J. Nagy, and Y. Xi}{Anchor Net \nys}

\begin{document}
\maketitle

\begin{abstract}
Kernel methods are used frequently in various applications of machine learning. For large-scale high dimensional applications, the success of kernel methods hinges on the ability to operate certain large dense kernel matrix $K$. An enormous amount of literature has been devoted to the study of symmetric positive semi-definite (SPSD) kernels, where \nys methods compute a low-rank approximation to the kernel matrix via choosing landmark points. In this paper, we study the \nys method for approximating both symmetric \emph{indefinite} kernel matrices as well SPSD ones.
We first develop a theoretical framework for general symmetric kernel matrices, which provides a theoretical guidance for the selection of landmark points. We then leverage discrepancy theory to propose the \emph{anchor net method} for computing accurate \nys approximations with optimal complexity.
%balance the efficiency and robustness for the selection of landmark points. 
%The analysis reveals that optimal landmark points should capture the geometry of the data, which serves as the theoretical foundation of the proposed method. The anchor net method selects the landmark points via the construction of an anchor net that reflects the data geometry. 
The anchor net method operates entirely on the dataset without requiring the access to $K$ or its matrix-vector product.
Results on various types of kernels (both indefinite and SPSD ones) and machine learning datasets demonstrate that the new method achieves better accuracy and stability with lower computational cost compared to the state-of-the-art \nys methods.
\end{abstract}



\section{Introduction} 
\label{sec:intro}
Kernel methods provide a powerful tool for solving nonlinear problems in data science and are used in various machine learning tools such as support vector machine (SVM), kernel ridge regression, spectral clustering, Gaussian processes (GPs) (cf. \cite{bishopbook}). Given $n$ data points $x_1,\ldots, x_n\in\mathbb{R}^d$ and a kernel function $\kappa(\cdot,\cdot)$, kernel methods form an $n\times n$ kernel matrix $K_{i,j}=\kappa(x_i, x_j)$ to implicitly map data to a kernel feature space, where the originally nonlinear relationship between categories can be transformed into a linear one. The kernel function $\kappa$ is often taken to be symmetric positive semi-definite (SPSD) in the literature \cite{vapnikbook,eigCMAM}. Recently, methods based on indefinite kernel functions such as jittering kernel \cite{svm2002}, Kullback-Leibler divergence kernel \cite{svm2003}, tangent distance kernel \cite{svm2002tangent} and multiquadric kernel \cite{multiquadric2004} have also been developed. In addition, indefinite kernel matrices also occur as the derivatives of SPSD kernels in solving optimization problems (see, for example \cite{bach2002kernel}) {and in non-metric proximity learning (cf. \cite{nys2015ind,nys2015review})}. Theoretical justifications for the support vector machines (SVMs) associated with indefinite kernels can be found in \cite{indefinite2004,indefinite2005}. 
%Yet a universal treatment for the fast computation with general kernel matrices is lacking.

Due to the need to assemble and operate dense kernel matrices $K$, kernel methods are often quoted to scale at least $O(n^2)$. A popular approach to circumvent this computational bottleneck is to work with a low-rank approximation to $K$. \nys methods are widely used to derive low-rank approximations to SPSD kernel matrices that arise frequently in SVMs and other applications. The method first generates a small subset of points $S$, known as \emph{landmark points}, and then computes a low-rank approximation of the following form:
\begin{equation}
\label{eq:nys} 
    K\approx K_{XS} K_{SS}^{+} K_{SX},
\end{equation}
where we use $K_{IJ}$ to denote the matrix with entries given by $\kappa(x,y)$ for $x\in I\subset\mathbb{R}^d, y\in J\subset\mathbb{R}^d$ and $K_{SS}^{+}$ denotes the pseudoinverse of $K_{SS}$.
Different ways of choosing $S$ yield different variants of \nys method. 
The original \nys method in \cite{nys2001} selects $S$ via a uniform sampling over the dataset, and is often called the \emph{uniform} \nys method. 
Later developments for generating $S$ include
non-uniform sampling techniques such as ridge leverage score based sampling \cite{nys2012,alaoui2015,leverage2016,nys2017,rudi2018} and determinantal point processes \cite{kdpp2011,dpp2020,dpp2021diversity}, $k$-means clustering based method \cite{nys2008kmeans,nys2010kmeans}, randomized projection method \cite{yuji20}, etc.
Since the choice of $S$ dictates the approximation accuracy, a fundamental question for \nys method is the following.
\begin{itemize}
    \item \textbf{Question 1.} Given a dataset $X$, what kind of subset $S$ yields a good \nys approximation ?
\end{itemize}
{For SPSD kernels, a probabilistic interpretation  is given by ridge leverage score.}
For general possibly \emph{indefinite} kernels, from a computational point of view, it is more desirable to have a straightforward geometric understanding of good landmark points $S$, which will facilitate the fast computation of $S$ in $O(n)$ complexity.
% The underlying principle for choosing $S$ is that it should be computationally efficient (i.e. with $O(n)$ scaling) and be able to yield an accurate approximation in \eqref{eq:nys}.

Despite the lack of discussion, the query for applying \nys approximation in \eqref{eq:nys} to \emph{indefinite} kernel matrices is quite natural, because mathematically, \eqref{eq:nys} does \emph{not} require $K$ to be SPSD. 
% but only SPSD kernel matrices are studied in existing \nys methods based on \eqref{eq:nys}. 
Thus it is natural to ask the following questions:
\begin{itemize}
    \item \textbf{Question 2.} Does \nys approximation  \eqref{eq:nys} apply to \emph{indefinite} kernel matrices ?
    \item \textbf{Question 3.} For a symmetric (possibly indefinite) kernel matrix, how should one choose $S$ in $O(n)$ complexity to obtain an accurate \nys approximation ?
\end{itemize}
% We point out that due to the pseudoinverse $K_{SS}^+$ in the \nys method, the choice of $S$ will significantly impact the numerical stability in computing \eqref{eq:nys}.
Note that ridge leverage score is only defined for SPSD kernel matrix and a different low-rank approximation method called random kitchen sinks method (or random Fourier features method) \cite{kitchen2007,kitchen2008,fastfood} { not only requires the kernel to be SPSD but also shift-invariant}. Hence those methods can \emph{not} be directly applied to \emph{indefinite} kernels.
{ 
Obviously, the original \nys method \cite{nys2001} based on uniform sampling can be applied to indefinite kernels (cf. \cite{nys2015ind,nys2015review}) since it is essentially sampling over the index set.
However, it is unclear how much the indefinite case differs from the positive definite case in terms of \nys approximation and what kind of landmark points are considered good.
}
The questions above motivate the work in this paper and the main contributions of the paper are summarized below.

\begin{enumerate}
    \item \textbf{Theoretical guidance for landmark point selection}.
    To guide the choice of landmark points, we present a new framework to analyze the \nys approximation error in the general setting where the kernel matrix can be \emph{indefinite}.
    % Different from existing error estimates for \nys methods, 
    The new error estimate takes the following form and is independent of the underlying scheme to select $S$:
    %Note that existing error estimates for \nys approximation are limited to the specific landmark-point selection scheme and thus can not be used to the quality of approximation given the landmark points. 
    %Since they have different forms and hidden constants, these estimates are not directly applicable for comparing the \nys approximation accuracy among different schemes. 
        \begin{equation}
        \label{eq0}
            || K-K_{XS}K_{SS}^+K_{SX}||_{\max}\leq \epsilon_1 + 2\epsilon_2+C_S\epsilon_2^2,
        \end{equation}
       where $\epsilon_1$, $\epsilon_2$ measure certain deviation between $S$ and $X$, $C_S=||K_{SS}^+||_2$ and {$||\cdot||_{\max}$ denotes the max norm, e.g., $\norm{A}_{\max}:=\max\limits_{i,j} |A_{i,j}|$.}
        A geometric interpretation of $\epsilon_1$ and $\epsilon_2$ suggests that landmark points should spread evenly in the dataset in order to achieve small approximation error.
    \item \textbf{Optimal complexity for general symmetric kernels}. 
    %Compared to probabilistic methods, deterministic approaches have the natural benefit that they have \emph{zero} variance and the results are predictable given the error estimates. Existing deterministic algorithms for low-rank approximations, like analytic approximation, come with rigorous theoretical guarantee, valid for arbitrary continuous kernels, and do not require accessing the kernel matrix. A major drawback is that it is inefficient in high dimensions.
Based on the analysis, we leverage discrepancy theory and propose an efficient deterministic \nys method for arbitrary symmetric (possibly \emph{indefinite}) kernels. {The proposed method scales $O(dmn)$ for selecting $m$ landmark points in a dataset of $n$ points in $\mathbb{R}^d$ and forming the associated low-rank factors $K_{XS}$ and $K_{SS}$.} This process is highly parallelizable and does \emph{not} require any access to the kernel matrix or its matrix-vector products.

\item \textbf{Improved efficiency, approximation accuracy and stability}.
% The proposed method does \emph{not} require over-sampling to obtain a rank $m$ approximation
 %draws $m+O(1)$ points from the dataset if a rank-$m$ approximation is sought 
% while probabilistic \nys methods often need to sample $k\gg m$ points in order to satisfy the statistical guarantees and those guarantees are no longer valid for \emph{indefinite} matrices. For example, for a given approximation accuracy $\epsilon$, $O(m\epsilon^{-2}\ln m)$ points have to be sampled \cite{guCUR2015} via uniform sampling or leverage score sampling in order to obtain a rank-$m$ \nys approximation with error bounded by $\epsilon$ with certain probability. 
{Comprehensive experiments have been performed to show that the proposed method outperforms several state-of-the-art methods for various kinds of kernels on both synthetic and real datasets when the same rank is used. We also show that the choice of $S$ significantly affects the numerical stability of the resulting \nys approximation and numerical regularization or stabilization techniques can \emph{not} fully resolve the stability issue.}
\end{enumerate}

The rest of the paper is organized as follows.
Section \ref{sec:nys Methoed} reviews existing \nys methods and
Section \ref{sec:errorbound} presents a general error analysis for \nys approximations to guide the selection of landmark points, valid for both \emph{indefinite} kernels and SPSD ones.
%providing a direct link between the landmark points and the approximation accuracy. 
%In Section \ref{sec:theory}, we present theoretical results to relate the data geometry to the numerical rank of the kernel matrix. 
Section \ref{sec:anchornet} introduces the \emph{anchor net method} for computing \nys approximation in linear complexity. Extensive experiments are provided in Section \ref{sec:experiments} and concluding remarks are drawn in Section \ref{sec:conclusion}. In the remaining sections, the following notations will be used throughout the paper.
\begin{itemize}
    \item $|x-y|$ denotes the Euclidean distance between $x,y\in\mathbb{R}^d$;
    \item $\norm{\cdot}$ denotes the 2-norm of a vector or a matrix;
    \item $\norm{\cdot}_{\max}$ denotes the max norm of a matrix, i.e., 
    $\norm{A}_{\max}:=\max\limits_{i,j} |A_{i,j}|$;
    % \item Let $I$ and $J$ be two finite sets of points in $\mathbb{R}^d$. We use $K_{I J}$ to denote a kernel matrix with entries $\kappa(x,y)$ for $x\in I, y\in J$;
    %\item Let $\text{rank}_{\epsilon}(K)$ denote the $\epsilon$-rank of the matrix $K$, i.e., the number of singular values that are greater than or equal to $\epsilon$ times the largest singular value of $K$;
%\cdf{\item We use $\operatorname{meas(\cdot)}$ to denote the Lebesgue measure of a set;}
\item $\distinf{\cdot,\cdot}$ denotes the distance function in $l^{\infty}$ norm;
\item $\lambda(\Omega)$ denotes the Lebesgue measure of a bounded measurable set $\Omega$ in $\mathbb{R}^d$.
\end{itemize}

\section{General \nys method: SPSD and indefinite cases} 
\label{sec:nys Methoed}
Given an input dataset $X=\{x_{1},\dots,x_{n}\}\subset \mathbb{R}^d$ and a symmetric ({possibly} indefinite) kernel function $\kappa(x,y)$, the corresponding kernel matrix is defined by $K = [\kappa(x_{i},x_{j})]_{i,j=1}^n$. 
For kernel functions supported on the entire domain of definition, such as $\kappa(x,y)=e^{-\Vert x-y\Vert^2}, \tanh(x\cdot y + 1)$, the corresponding kernel matrix $K$ is dense and the corresponding cost for storing the matrix or applying it to a vector is $O(n^2)$. 
%Therefore, for large scale datasets, it is favorable in practice to seek a low-rank approximation to $K$ in order to reduce the computational costs.

The \nys method was proposed in \cite{nys2001} to reduce the quadratic cost by computing an approximate low-rank factorization in the form $K\approx K_{XS} K_{SS}^{+} K_{SX}$, where the size of $S$ is significantly smaller than $n$.
Different variants of \nys method use different methods to compute the landmark points $S$.
Some methods require $K$ to be SPSD, including non-uniform sampling based approaches like leverage score sampling, determinantal point processes, etc., while others like uniform sampling and $k$-means clustering can also be potentially applied to \emph{indefinite} kernel matrices.
% In almost all \nys schemes, the set $S$ is chosen as a subset of $X$. The only exception is the $k$-means based method \cite{nys2008kmeans,nys2010kmeans} where $S$ contains cluster centers which are in general outside $X$. 
We review below some popular \nys methods.

%For symmetric \emph{indefinite} kernel matrices, however, it is not clear whether the \nys approximation in \eqref{eq:nys} is still valid.
%We use the term \emph{generalized \nys method} to emphasize the use of \nys method for possibly \emph{indefinite} kernels.

The original \nys method in \cite{nys2001}, known as the \emph{uniform} \nys method, selects landmark points via a uniform sampling over $X$ (or equivalently, over the index set from $1$ to $n$). Since then, a variety of schemes have been developed to select landmark points. See, for example, \cite{nys2008kmeans,kumar2009sampling,nys2010kmeans,nys2012,alaoui2015,leverage2016,nys2017,rudi2018}. Computationally, the uniform \nys method is the most efficient one, since it does \emph{not} require any access to the kernel matrix or its matrix-vector product, and is not iterative. As a result, the uniform \nys method is easy to compute and can be applied to a broad class of kernel matrices.
% The biggest advantage of the uniform \nys method is that it is free from the curse of dimensionality because the sampling is performed on the index set of the data, independent of the dimension. 
%Second, the uniform \nys method can use an arbitrary number of landmark points, while the total number of nodes or terms in an analytic approach can not be specified arbitrarily. 
However, the uniform \nys method also suffers from several issues. Firstly, due to the stochastic nature, it suffers from possibly large variance \cite{rudi2018}. Secondly, the approximation accuracy usually fails to increase consistently with the increase of the number of landmark points. Thirdly, the random choice of landmark points may lead to numerically unstable approximations.
The accuracy slowdown and numerical instability will be illustrated via extensive numerical experiments.
% Experiments show that the accuracy may get stagnated after the number of landmark points goes beyond a certain number. 

Non-uniform sampling techniques have been developed to improve the approximation accuracy with strong theoretical guarantees \cite{nys2012,alaoui2015,leverage2016,nys2017,kdpp2011}. These methods measure the importance of each data point with some statistical scores. A notable example is the leverage score based sampling \cite{leverage2009,alaoui2015,leverage2016}, including determinantal point processes \cite{kdpp2011,dpp2021diversity}. Each point $x_i$ in the dataset is associated with a leverage score defined as $l_i^{\gamma}(K):=(K(K+\gamma I)^{-1})_{i,i}$ with $\gamma>0$ a user-specified parameter. To generate the landmark points, each point $x_i$ is sampled with a probability proportional to $l_i^{\gamma}(K)$. Since computing leverage scores involves the dense kernel matrix $K$ and computing the matrix inverse $(K+\gamma I)^{-1}$, these methods cost at least $O(n^2)$. Recently, several iterative schemes 
%based on divide-and-conquer principles 
have been proposed to accelerate its computations \cite{nys2017,rudi2018}. 
{Different from uniform sampling, those methods require $K$ to be SPSD in order to guarantee the non-negativeness of $l_i^{\gamma}$.}

%These new developments make the leverage score based sampling techniques scalable for solving large-scale problems. 
%Comparisons with these methods will be provided in Section \ref{sec:experiments}. 

Another variant of \nys methods is the $k$-means \nys method \cite{nys2008kmeans,nys2010kmeans}. This method performs the $k$-means clustering over the dataset and chooses the cluster centers as the landmark points. 
Similar to uniform sampling, the $k$-means method does not require any access to the kernel matrix. Experiments show that it tends to be more accurate than the uniform \nys method \cite{nys2008kmeans,nys2010kmeans} but still suffers from numerical instability.
%Note that the $k$-means clustering \nys method is not entirely deterministic since different initializations of cluster centers will yield different results. Additionally, one needs to specify an iteration number, which may greatly affect the quality of clustering and the computational costs.

{
In existing literature, discussion on the choice of good landmark points that work for \emph{indefinite} kernels has been scarce.
In \cite{nys2015ind,nys2015review}, uniform sampling \cite{nys2001} is used to select landmark points for indefinite kernels. 
\cite{nys2019ind} proposed to first use uniform sampling to obtain a \nys approximation and then apply leverage score method to the \nys approximation to select landmark points.
However, it is not clear how well the original \nys method \cite{nys2001} performs for indefinite kernels in general and how different choices of landmark points affect the approximation accuracy for indefinite kernels.
Moreover, a theoretical study of landmark selection for indefinite kernels is lacking.
% In fact, the \nys approximation $K\approx K_{XS}K_{SS}^+K_{SX}$ has not been investigated for \emph{indefinite} kernels, theoretically and numerically.
We will show that, though \nys approximation can be used for \emph{indefinite} kernels, there are a lot more numerical challenges such as numerical instability, as compared to SPSD kernels.
For indefinite kernels, landmark points must be chosen judiciously to prevent a poor \nys approximation.
}

\section{Error estimates for the general \nys method}
\label{sec:errorbound}
In this section, we derive error estimates for the \emph{general} \nys method, which are valid for all symmetric kernels, including indefinite ones. The only assumption is that the landmark points are chosen from the original dataset. 
The analysis reveals the inherent relation between landmark points and the quality of the corresponding \nys approximation. It serves as the theoretical foundation of the new linear complexity method proposed in Section \ref{sec:anchornet}.
%The analysis in this section reveals how the choice of landmark points is related to the error of the resulting \nys approximation and
%consequently answers the fundamental question:
%``what serves as a \emph{good} choice of landmark points".
%The theoretical result provides a guideline for choosing good landmark points, 
%which inspires the subsequent analysis and methodology in Section \ref{sec:theory} and Section \ref{sec:anchornet}, respectively.

%We first present a general result that states the positive definiteness of kernel matrices associated with exponential type kernels, including Gaussians.  A proof can be found in \cite{eigCMAM}. Proposition \ref{prop:SPD} implies that, the submatrix matrix $K_{SS}$ in \eqref{eq:nys} is always nonsingular for Gaussian kernels.
%\begin{proposition}
%\label{prop:SPD}
%    Let $\rho(x)$ be a function on $\mathbb{R}^d$ given in one of the following forms:
%    $(1)\, \rho(x)=\sum\limits_{i=1}^d |\xi_i|/\sigma_i$;
%    $(2)\, \rho(x)=\sum\limits_{i=1}^d \xi_i^2/\sigma_i^2$;
%    $(3)\, \rho(x)=\left(\sum\limits_{i=1}^d \xi_i^2/\sigma_i^2\right)^{1/2}$,
%    where $x=(\xi_1,\dots,\xi_d)\in\mathbb{R}^d$ and $\sigma_i>0$.
%    For any distinct points ${x}_1,\dots, {x}_N\in\mathbb{R}^d$,
%    the kernel matrix $K$ with $K_{i,j}:=e^{-\rho({x}_i-{x}_j)}$ is positive definite.
%\end{proposition}

%%%The three classes of $\rho$ in Proposition \ref{prop:SPD} correspond to the (weighted) $l^1$ norm, the squared (weighted) $l^2$ norm and the (weighted) $l^2$ norm, respectively. 

The lemma below will be used in proving the main result in Theorem \ref{thm:errorbound}.
\begin{lemma}
\label{lm:lm1}
    Assume $A$ is an $n$-by-$n$ matrix and $\alpha,\alphahat,\beta,\betahat$ are $n$-by-$1$ vectors.
    Define $\epsilon_1:=\norm{\alphahat-\alpha}$ and $\epsilon_2:=\norm{\betahat-\beta}.$
    Then 
    \begin{equation}
    \label{eq:lm1}
        \abs{\alphahat^T A \betahat - \alpha^T A \beta} 
        \leq \norm{\alpha^T A}\cdot \epsilon_2 + \norm{A\beta}\cdot \epsilon_1
        + \norm{A}\cdot \epsilon_1\epsilon_2.
    \end{equation}
    %where $\norm{\cdot}$ denotes the 2-norm.
\end{lemma}
\begin{proof}
    Define $e_1:=\alphahat-\alpha$ and $e_2:=\betahat-\beta$.
    Then \eqref{eq:lm1} follows from the fact that
    \[
    \alphahat^T A \betahat - \alpha^T A \beta = \alpha^T A e_2 + e_1^T A \beta + e_1^T A e_2.
    \]
\end{proof}

In the theorem below, we derive a universal error bound for the \nys method. The kernel function is assumed to be symmetric and continuous, not necessarily positive-definite. Unlike existing error estimates, the result below is independent of the specific \nys scheme. The only assumption is that the landmark points belong to the original dataset, which is indeed the case in all \nys schemes except the one based on $k$-means clustering \cite{nys2008kmeans,nys2010kmeans}.
\begin{theorem}
\label{thm:errorbound}
    Let $\kappa(x,y)$ {be a symmetric function e.g.,} $\kappa(x,y)=\kappa(y,x)$.
    Suppose $X=\{x_1,\dots,x_n\}\subset \mathbb{R}^d$ and $K=K_{XX}:=[\kappa(x_i,x_j)]_{i,j=1}^n$.
    If $S=\{z_1,\dots,z_r\}\subset X$, then 
    \begin{equation}
    \label{eq:errorbound}
        \norm{K-K_{XS}K_{SS}^{+}K_{XS}^T}_{\max} \leq E_r + 2\hat{E}_r + \norm{K_{SS}^{+}} \hat{E}_r^2,
    \end{equation}
    where
    \begin{equation}
    \label{eq:EE}
    \begin{aligned}
        E_r := \max_{x,y\in X} \min_{u,v\in S}|\kappa(x,y)-\kappa(u,v)|,\quad
        \hat{E}_r := \max_{x\in X} \min_{u\in S}\left( \sum_{i=1}^r |\kappa(x,z_i)-\kappa(u,z_i)|^2 \right)^{\frac{1}{2}}.
    \end{aligned}
    \end{equation}
%    \begin{equation}
%\label{eq:zx2}
%%    \sum_{i=1}^r |\kappa(x,z_i)-\kappa(z_x,z_i)|^2 = \min_{s\in S} \sum_{i=1}^r |\kappa(x,z_i)-\kappa(s,z_i)|^2.
%    \norm{x-z_x} = \min_{s\in S} \norm{x-s}.
%\end{equation}
%If there are multiple minimizers, then $z_x$ is chosen to be with the smallest index.
\end{theorem}
%\cdf{TODO in new paper: use \eqref{eq:zx} (i.e., latent kernel distance $||\kappa(x,z_i)-\kappa(s,z_i)||_{i=1}^r$) to determine landmark points and clusters instead of using spatial distance in $\mathbb{R}^d$ ? 
%or find $z_x$ by minimizing the upper bound estimate? 
%What about using Neural Network to Minimize the bound in \eqref{eq:errorbound} to find landmark points: $z_i$ ?}

\begin{proof}
Define $R=X\backslash S$. 
Since $S\subset X$, for some permutation matrix $P$, there holds
 \[
    K-K_{XS}K_{SS}^{+}K_{XS}^T = P \begin{bmatrix}
        O & O \\
        O & K_{RR}-K_{RS}K_{SS}^{+}K_{RS}^T
    \end{bmatrix} P^T.
 \]
Consequently,
 $\norm{K-K_{XS}K_{SS}^{+}K_{XS}^T}_{\max} = \norm{K_{RR}-K_{RS}K_{SS}^{+}K_{RS}^T}_{\max}.$
It suffices to estimate the difference below
\begin{equation}
\label{eq:diff}
%    \kappa(x,y) - \begin{bmatrix}
%        \kappa(x,z_1) & \cdots & \kappa(x,z_r)
%    \end{bmatrix}    K_{SS}^{+}
%    \begin{bmatrix}
%        \kappa(z_1,y) & \cdots & \kappa(z_r,y)
%    \end{bmatrix}^T =
    \kappa(x,y)-K_{xS} K_{SS}^{+} K_{yS}^T, 
\end{equation}
where 
\[
    K_{xS} := \begin{bmatrix}
        \kappa(x,z_1) & \cdots & \kappa(x,z_r)
    \end{bmatrix}\quad \text{for any }\; x\in\mathbb{R}^d.
\]
Note that for any $u,v\in S$,
\begin{equation}
\label{eq:kzrewrite}
   \kappa(u,v) =  
%    \begin{bmatrix}
%        \kappa(z_x,z_1) & \cdots & \kappa(z_x,z_r)
%    \end{bmatrix}
    K_{u S}
    K_{SS}^{+}
    K_{v S}^T
%    \begin{bmatrix}
%        \kappa(z_1,z_y) & \cdots & \kappa(z_r,z_y)
%    \end{bmatrix}^T.
\end{equation}
because $K_{SS}K_{SS}^{+}K_{SS}=K_{SS}$ and $K_{SS}=K_{SS}^T$.
Define the column vectors
\begin{equation}
\label{eq:abdef}
    \alpha := K_{u S}^T,\quad 
    \alphahat:= K_{x S}^T,\quad
    \beta :=K_{v S}^T,\quad
    \betahat:=K_{y S}^T
\end{equation}
and the scalars
\begin{equation}
\label{eq:ee12def}
\begin{aligned}
    \epsilon_1 &:= ||\alphahat-\alpha||=\left( \sum_{i=1}^r |\kappa(x,z_i)-\kappa(u,z_i)|^2 \right)^{\frac{1}{2}},\\
    \epsilon_2 &:= ||\betahat-\beta||=\left( \sum_{i=1}^r |\kappa(y,z_i)-\kappa(v,z_i)|^2 \right)^{\frac{1}{2}}.
\end{aligned}
\end{equation}
%For each $x\in X$, choose $z_x$ to be the point in $S$ such that
%\begin{equation}
%\label{eq:zx}
%    \sum_{i=1}^r |\kappa(x,z_i)-\kappa(z_x,z_i)|^2 = \min_{s\in S} \sum_{i=1}^r |\kappa(x,z_i)-\kappa(s,z_i)|^2.
%    \norm{x-z_x} = \min_{s\in S} \norm{x-s}.
%\end{equation}
%If there are multiple minimizers, then $z_x$ is chosen to be with the smallest index.
We can then use \eqref{eq:kzrewrite} and \eqref{eq:abdef} to rewrite \eqref{eq:diff} as
\begin{equation}
\label{eq:diffrewrite} 
    \kappa(x,y) - \alphahat^T K_{SS}^{+} \betahat 
    = (\kappa(x,y) - \kappa(u,v)) +( \alpha^T K_{SS}^{+} \beta - \alphahat^T K_{SS}^{+} \betahat),
\end{equation}
where $u,v\in S$ can be arbitrary.

{The second part on the right-hand side of \eqref{eq:diffrewrite} can be estimated as follows by using Lemma \ref{lm:lm1}:
\begin{equation}
\label{eq:term2}
\begin{aligned}
        \abs{\alphahat^T K_{SS}^{+} \betahat - \alpha^T K_{SS}^{+} \beta} 
        &\leq \norm{\alpha^T K_{SS}^{+}} \epsilon_2 + \norm{K_{SS}^{+}\beta} \epsilon_1
        + \norm{K_{SS}^{+}} \epsilon_1\epsilon_2\\
        &\leq \epsilon_2+\epsilon_1+\norm{K_{SS}^{+}} \epsilon_1\epsilon_2.
        %2\hat{E}_r+\norm{K_{SS}^{+}} \hat{E}_r^2,
\end{aligned}
\end{equation}
Here, the last inequality is due to the fact that both 
$\alpha^T K_{SS}^{+}=K_{uS}K_{SS}^{+}$ and $K_{SS}^{+}\beta=K_{SS}^{+}K_{Sv}$ are row or column of the matrix 
\[
    K_{SS}K_{SS}^{+} = K_{SS}^{+}K_{SS} = U\begin{bmatrix}
        e_1 & & \\
          & \ddots & \\
          &  &  e_r
    \end{bmatrix}U^T,
\]
where $U$ is an orthogonal matrix and $e_i\in \{0,1\}$.} {As a result, $\norm{\alpha^T K_{SS}^{+}}\leq 1$ and $\norm{K_{SS}^{+}\beta}\leq 1$. }

Finally, \eqref{eq:term2} and \eqref{eq:diffrewrite} imply the following estimate:
\begin{equation}
\label{eq:syuvBound}
    \abs{\kappa(x,y) - \alphahat^T K_{SS}^{+} \betahat} 
    \leq \abs{\kappa(x,y) - \kappa(u,v)} + \epsilon_1+\epsilon_2+\norm{K_{SS}^{+}} \epsilon_1\epsilon_2,
\end{equation}
which holds for any $u,v\in S$.
Minimizing the upper bound in \eqref{eq:syuvBound} over all $u,v\in S$ immediately yields
\begin{equation}
\label{eq:minuv}
    \abs{\kappa(x,y) - \alphahat^T K_{SS}^{+} \betahat} 
    \leq \min_{u,v\in S} \abs{\kappa(x,y) - \kappa(u,v)} + \min_{u\in S}\epsilon_1+\min_{v\in S}\epsilon_2+\norm{K_{SS}^{+}} \min_{u\in S}\epsilon_1\cdot \min_{v\in S}\epsilon_2,
\end{equation}
where $\epsilon_1, \epsilon_2$ are defined in \eqref{eq:ee12def}.
The proof is completed by taking a maximum of the upper bound in \eqref{eq:minuv} over $x,y\in X$.
That is,
\[
\begin{aligned}
    \abs{\kappa(x,y) - \alphahat^T K_{SS}^{+} \betahat} 
    &\leq \max_{x,y\in S}\min_{u,v\in S} \abs{\kappa(x,y) - \kappa(u,v)} + \max_{x\in X}\min_{u\in S}\epsilon_1+\max_{y\in X}\min_{v\in S}\epsilon_2\\
    &+\norm{K_{SS}^{+}} \max_{x\in X}\min_{u\in S}\epsilon_1\cdot \max_{y\in Y}\min_{v\in S}\epsilon_2\\
    &= E_r + 2\hat{E}_r + \norm{K_{SS}^{+}}\hat{E}_r^2.
\end{aligned}
\]
\end{proof}

\begin{remark}
\label{rm:errorbound}
Note that Theorem \ref{thm:errorbound} only requires the kernel function to symmetric. Thus the result applies to a broad class of kernels, including SPSD kernels like Gaussian, or more generally Mat{\'e}rn kernels, and indefinite kernels, such as  multiquadrics, thin plate spline, sigmoid kernel, etc. 
\end{remark}

We call $E_r$ and $\hat{E}_r$ in Theorem \ref{thm:errorbound} the bivariate and univariate \emph{kernelized marking errors},
respectively,
as both quantities are measured in terms of either bivariate or univariate kernel function evaluations and indicate the overall capacity of the landmark points $S$ to approximate the dataset $X$. {There are two variables that affect the approximation error of \nys method: the number of the landmark points $r$ and the set of landmark points $S$.
Here we focus on how to choose landmark points $S$ when $r$ is fixed.
In this case, both the quantities $\hat{E}_r$ and $E_r$ can be used to investigate how different choices of landmark points $S$ would impact the \nys approximation. 
If $r$ is viewed as a variable, then $\hat{E}_r$ may or may not grow as $r$ increases. 
Consider the one dimensional toy problem where $\kappa(x,y)=|x-y|^2$,
$X=\{x_i\}_{i=0}^{10} = \{\frac{i}{10}\}$, $S_1=\{x_{10}\}$, $S_2=\{x_0,x_{10}\}$.
Let $\hat{E}_k$ denote the quantity for $S=S_k$.
Then it can be computed that $\hat{E}_1 = 1$ (achieved at $x=x_0$)
and $\hat{E}_2 = \frac{1}{2\sqrt{2}}\approx 0.35$ (achieved at $x=x_5$). In this case, $\hat{E}_r$ does decay as $r$ increases.
% It is not hard to construct an example where $\hat{E}_r$. 
 If we further assume the kernel function $k(x,y)$ is Lipschitz continuous, then $|k(x,y)-k(u,v)|$ will be small if $(x,y)$ and $(u,v)$ are close. Under this assumption, the distance between points reflects the difference between the respective kernel evaluations.}
%In view of $E_r$, we see that the error will be small if for any point $x\in X$, there is a point $u\in S$ nearby. This applies to $\hat{E}_r$ as well. 
Hence the set of {a fixed number of} landmark points $S$ is considered good if it is able to minimize the deviation from $X$, namely, making $\dist{x,S}$ small for each point $x\in X$. {A similar result has recently been conducted in \cite{bauer2020kernelindependent}, which shows the exponential convergence of Adaptive Cross Approximation (ACA) \cite{DBLP:journals/nm/Bebendorf00} with respect to the fill-distance of pivoting points. We also want to emphasize that the estimate \eqref{eq:errorbound} is mainly used to motivate the selection of $S$ rather than to select the number of the points in $S$ in order to satisfy certain approximation accuracy.}

{In the next corollary, we further show that the approximation error can be bounded by $\max\limits_{x\in X}\dist{x,S}$ when the kernel function $\kappa$ is Lipschitz continuous.}
\begin{corollary}
\label{cor:errorbound}
Under the assumption of Theorem \ref{thm:errorbound}, if $\kappa(x,y)\in C(\mathbb{R}^d\times \mathbb{R}^d)$ is Lipschitz continuous, i.e, $|\kappa(x',y')-\kappa(x,y)|\leq L (|x-x'|^2 + |y-y'|^2)^{1/2}$ with Lipschitz constant $L$, then 
\begin{equation}
    \label{eq:errorboundLip}
        \norm{K-K_{XS}K_{SS}^{+}K_{XS}^T}_{\max} \leq \sqrt{2}L\delta_{X,S} + 2\sqrt{r}L\delta_{X,S} + \norm{K_{SS}^{+}} rL^2\delta_{X,S}^2,
    \end{equation}
    where $\delta_{X,S} = \max\limits_{x\in X}\dist{x,S}$.
\end{corollary}
\begin{proof}
    The proof relies on \eqref{eq:errorbound} and it suffices to relate $E_r$, $\hat{E}_r$ to $\max\limits_{x\in X}\dist{x,S}$.
    First we estimate $E_r$. For each $x\in X$, choose $z_x\in S$ to be the nearest point to $x$.
    Then it follows that 
\[
    E_r^2 \leq L^2 \max_{x,y\in X} (|x-z_x|^2+|y-z_y|^2) = 
    2L^2 \max_{x\in X} |x-z_x|^2 = 2 L^2 \max_{x\in X} \dist{x,S}^2.
\]
    Similarly, for $\hat{E}_r$, we deduce that
\[
    \hat{E}_r^2 \leq L^2 \max_{x\in X} \sum_{i=1}^r |x-z_x|^2 = 
    L^2 \max_{x\in X} r |x-z_x|^2 =
    r L^2 \max_{x\in X} \dist{x,S}^2.
\]
    The two estimates and \eqref{eq:errorbound} immediately imply \eqref{eq:errorboundLip} and the proof is complete.
\end{proof}

%\textcolor{blue}{Both Theorem \ref{thm:errorbound} and Corollary \ref{cor:errorbound} indicate that there should be \emph{no} essential difference between SPSD and indefinite kernels in applying \nys approximations to possibly high dimensional datasets.} 
It can also be seen from \eqref{eq:errorboundLip} that, to achieve a better approximation, landmark points are encouraged to spread over the entire dataset to capture its geometry, thus reducing $\delta_{X,S}$. Roughly speaking, this means that any point in $X$ is not ``too far" from a landmark point in $S$. In fact, this principle can also lead to a submatrix $K_{SS}$ with a relatively large numerical rank in general, an improved numerical stability and accuracy.
In two and three dimensions, one way to generate evenly spaced samples is to use farthest point sampling (FPS) \cite{FPS97}. {FPS constructs a subset $S$ of $X$ by first initializing $S$ with one point and then sequentially adding to $S$ a point in $X\backslash S$ that is farthest from $S$.} However, in high dimensions, the method tends to sample points on the boundary of the dataset and may ignore the interior of the dataset unless the number of samples is large enough. Computationally, the sequential procedure of FPS can be quite expensive in high dimensions since each step requires solving a minimization problem over $O(n)$ points and the overall complexity is $O(m^2 n)$ for generating $m$ landmark points, which is not optimal in $m$.
We present in the next section an efficient, fully parallelizable algorithm with linear complexity in $m$ and $n$ to generate the desired subset $S$.

Note that several existing work has analyzed low-rank approximations associated kernel matrices based on analytic approximation of the kernel function \cite{hack1989,ubt_eref6643,10.1016/j.jat.2010.04.012,Townsend2015ContinuousAO}. Although these results are independent of the positive definiteness of the kernel function, they are restricted to low dimensions because of the curse of dimensionality associated with analytic techniques. That is, the number of terms in an analytic approximation increases \emph{exponentially} with the dimension and the resulting matrix approximation is \emph{not} low-rank for high dimensional problems. {In the context of integral equations, a popular method called adaptive cross approximation (ACA) \cite{DBLP:journals/nm/Bebendorf00} serves as a column-pivoted LU factorization. Thus it is able to perform low-rank factorization for kernel matrix with high dimensional data in linear complexity.}

%Compared to the existing error analysis on the \nys method, the theory developed in this section has the following advantages.
%\begin{enumerate}
%    \item \textbf{General kernel functions.} Theorem \ref{thm:errorbound} only requires the kernel function to be continuous and symmetric, \emph{not} necessarily positive-definite.
%    \item \textbf{Independence of landmark selection scheme.} Unlike existing error analysis developed for the respective \nys method, Theorem \ref{thm:errorbound} and Corollary \ref{cor:errorbound} are independent of the specific \nys scheme. That is, the estimates in \eqref{eq:errorbound} and \eqref{eq:errorboundLip} can be applied regardless of the particular algorithm (stochastic or deterministic) to generate  landmark points.
%    \item \textbf{Deterministic and explicit bound.} The estimates in \eqref{eq:errorbound} and \eqref{eq:errorboundLip} hold true with probability 1 and there is no hidden constant in the error bounds.
%    \item \textbf{Computability and theoretical guidance.} The estimates in \eqref{eq:errorbound} and \eqref{eq:errorboundLip} are computable and can thus be used to evaluate the quality of landmark points.
%    Thanks to the data-dependence in the error bounds, they provide theoretical guidance for how to choose good landmark points.
%\end{enumerate}

%This can be verified by revealing the \cdf{relation between the geometry of the dataset and the numerical rank of the associated kernel matrix.} 

%The cost is merely dominated by the calculation of pairwise distance between points in $X$ and $S$. 
\begin{remark}
One may use a different norm to measure the approximation error. The set of optimal landmark points that minimize the error bound may differ, depending on the underlying norm. A detailed investigation on how the norm affects the choice of landmark points will be discussed in a forthcoming paper. The analysis in this section aims to provide an intuitive understanding of the desired qualities of landmark points, which will then serve as a theoretical guidance for choosing landmark points.  
\end{remark}

\begin{remark}
    It should be pointed out that directly minimizing the bound in Theorem \ref{thm:errorbound} is \emph{not} a practical way to generate $S$ due to the high computational cost, for example, $O(n^2)$ in computing $E_r$ in \eqref{eq:EE}.
    Our goal is to design a fast algorithm (with optimal complexity) for generating landmark points with good quality.    
    Thus the error bound is used as a theoretical guidance for designing more efficient algorithms on generating landmark points $S$.
\end{remark}

\section{Anchor net method} 
\label{sec:anchornet}
In this section, we introduce the anchor net method to facilitate the selection of landmark points.
From the analysis in Section \ref{sec:errorbound},
we see that landmark points that spread evenly in the dataset and contain no clumps are more favorable in reducing the Nystrom approximation error.
If the dataset is the unit cube,
then the uniform grid points satisfy the desired properties. 
In general, the study of uniformity is a central topic in discrepancy theory for solving high dimensional problems. 
The discrepancy of a given point set measures how far the distribution deviates from the uniform one.
Existing work on discrepancy theory all focuses on distribution in the unit cube,
while distribution in a general region has not been investigated yet, theoretically or computationally. 
In Section \ref{sub:Anchor Nets}, we review low discrepancy sets and give the definition of discrepancy for a \emph{general} region instead of the unit cube.
%later used as building blocks for the search of good landmark points.
Based on low discrepancy sets, the anchor net is introduced in Section \ref{sub:ANlandmark}, which is able to capture the geometry of the given dataset. 
In Section \ref{sub:ANN}, we use anchor net to design a linear complexity landmark-point selection algorithm.
Discussion on implementation details  is provided in  Section \ref{subsec:implementation}. 
%Specifically, we prove that, for highly unbalanced datasets, 
%with a high probability,
%the uniform sampling will generate landmark points $S$ that leads to a nearly singular submatrix $K_{SS}$ while the anchor net method doesn’t.

%The analysis in Section \ref{sec:theory} implies that a good choice of landmark points should span the entire dataset and contain no small clusters.
%To facilitate the implementations respecting those criteria, in the following subsection, we introduce several concepts that help quantify the observations and contribute to the method of anchor net.

\subsection{Low discrepancy sets} 
\label{sub:Anchor Nets}
%Opposing to the highly concentrated distribution or highly unbalanced distribution is the uniform distribution.
We start with the concept of low discrepancy sets.
%Discrepancy measures the uniformity of a dataset. 
Roughly speaking, a dataset with low discrepancy contains points that spread evenly in the space, with almost no local accumulations. There are several kinds of discrepancies \cite{UD2012book} and
the most widely used one is the star discrepancy, as defined below. 
\begin{definition}
\label{def:DN}
The star discrepancy $D^*_N(\annt)$ of $\annt=\{x_1,\dots,x_N\}$
$\subset [0,1]^d$ is defined by 
\[
    D^*_N(\annt):= \sup_{J\in\mathcal{J}_1} |\#(\annt\cap J)/N-\lambda(J)|,
\]
where $\mathcal{J}_1$ is the family of all boxes in $[0,1)^d$ of the form $\prod_{i=1}^d [0,a_i)$ {and $\lambda(J)$ denotes the Lebesgue measure of $J$.}
\end{definition}

Low discrepancy sets have been studied in a number of literature as a means of generating quasi-random sequences \cite{halton,sobol,quasirandombook,UD2012book}. The most widely used ones include Halton sequences \cite{halton}, digital nets and digital sequences \cite{sobol,digitalnet1987,digitalnet2010book}. 
They are known to have low discrepancies in the sense that 
\[
D_N^*(\annt_N)=O(N^{-1}(\log N)^d),
\]
where $\annt_N$ denotes the first $N$ terms of a Halton sequence or a digital sequence \cite{quasirandombook,UD2012book}. 
{Note that uniform tensor grids are also representative low discrepancy sets but are not popular in practice due to the curse of dimensionality.
We present adaptive tensor grids in Section \ref{subsec:implementation} to alleviate the issue,
which allows the practical use of tensor grids in high dimensions.}

The above low discrepancy sets themselves are only defined for the unit cube and are inefficient in tessellating a real dataset whose ``shape" may not be regular. Therefore, we introduce what we call the \emph{anchor net} in Section \ref{sub:ANlandmark}, which is built upon a collection of low discrepancy sets adjusted to the structure of the dataset. Loosely speaking, anchor nets can be viewed as generalized low discrepancy sets dictated by and specific to the given dataset. In order to measure the uniformity of a dataset in a general region,
we first generalize Definition \ref{def:DN} below.
\begin{definition}
\label{def:genDN}
Let $\annt=\{x_1,\dots,x_N\}\subset [0,\infty)^d$ and $\Omega$ be a bounded measurable set in $[0,\infty)^d$
such that $\lambda(\Omega)>0$ and $\annt\subset\Omega$. 
The generalized star discrepancy $\DNOA$ of $\annt$ in $\Omega$ is defined by
\[
    \DNOA = \sup_{J\in\mathcal{J}} |\#(\annt\cap J)/N-\lambda(\Omega\cap J)/\lambda(\Omega)|,
\]
where $\mathcal{J}$ is the family of all boxes in $[0,\infty)^d$ of the form $\prod_{i=1}^d [0,a_i)$.
\end{definition}
Note that the generalized star discrepancy $\DNOA$ coincides with the standard one $D^*_N(\annt)$
if $\annt\subset\Omega=[0,1)^d$.
%Definition \ref{def:genDN} does not make much sense 
{Given an arbitrary dataset, finding a region $\Omega$ that contains the dataset and {reflects} the geometry of the data will be beneficial in generating efficient  samples that can effectively minimize the approximation error.
However, the perfect region is extremely challenging to find in general since the discrete dataset can be \emph{arbitrary}.
In Section \ref{sub:ANlandmark}, we introduce \emph{anchor nets} as a computationally efficient way for constructing such a region $\Omega$.
Anchor nets will be used in Section \ref{sub:ANN}
to facilitate the selection of landmark points with linear complexity.
}

\subsection{Anchor nets} 
\label{sub:ANlandmark}
%In this section we first give the formal definition of \emph{anchor net} and then propose an efficient algorithm to construct an anchor net for a given dataset.  

In this section, we present an efficient algorithm to construct the so-called \emph{anchor net} for a given dataset. We then verify two major properties of the anchor net: it is able to capture the entire dataset and it has low discrepancy. 
As discussed in Section \ref{sec:errorbound}, good landmark points are expected to spread over the entire dataset without forming clumps. The anchor net is designed to achieve this goal by leveraging discrepancy theory \cite{digitalnet1987,digitalnet2010book}, where one tries to construct low discrepancy sequences (deterministically) in order to avoid clumps that are frequently found in pure random sequences. Low discrepancy sequences can achieve faster convergence than pure random sequences in Monte Carlo methods \cite{QMC1992book,quasirandombook,digitalnet2010book}. Intuitively, one can view the anchor net as the counterpart of low discrepancy sequence and uniform sampling as the counterpart of random sequence.

% \begin{definition}[Anchor net]
% \label{def:annt}
% For a dataset $X\subset [0,\infty)^d$, the sets $\anntXp\subset [0,\infty)^d$ of increasing size (in terms of cardinality) parametrized by $p=0,1,2,\dots$
% are called anchor nets if the following two conditions are satisfied.
% \begin{enumerate}
%     \item Define
%     $\Omega := \bigcap\limits_{\epsilon>0} \limsup\limits_{p\to\infty} \left\{ x\in\mathbb{R}^d: \distinf{x,\anntXp}\leq\epsilon \right\}$.
%     Then $\lambda(\Omega)>0$ and $X\subset \Omega$. 
%     \item $\lim\limits_{p\to\infty} D^*_{L_p,\Omega}(\anntXp) = 0, \text{ where } L_p=\#\anntXp$.
% \end{enumerate}
% \end{definition}

% We see that the first condition of Definition \ref{def:annt} implies that the anchor net $\anntXp$ becomes dense enough surrounding the given dataset $X$ as $p$ increases and the second condition pertains to the uniformity of the anchor net. 
%\cdf{To characterize the given set $X$, the set $\Omega$ should be ideally as close to $X$ as possible, so it does not make much sense to choose $\Omega$ to be an arbitrary set that contains $X$. Note that since the concepts in discrepancy theory are of an asymptotic sense, so is Definition \ref{def:annt}.}

{The anchor net can be considered as a two-level low discrepancy set. The first level is used to decompose the dataset into smaller subsets and the second level is used to generate ``anchors". The construction procedure is sketched in Algorithm \ref{alg:anchornet}. The inputs are the dataset and a net size $m$.  
Line 1 first generates a low discrepancy set $\mathcal{T}$ of a given size in the smallest box $B_0$ that contains $X$. 
Lines 3--6 decompose $X$ into smaller subsets $G_i$'s.
%Lines 3-6 check the distance between the data points and $s$ points in $\mathcal{T}$ and classify the data points into $Q$ nonempty subgroups $G_{i}$ ($i=1,\ldots,Q$) based on their nearness to each point in $\mathcal{T}$. 
Lines 8--10 then construct a low discrepancy set $\annt^{(i)}$ for each (non-empty) $G_i$. 
Here we choose the number of points in $\annt^{(i)}$ to be  equal to $\ceil{{m*\lambda(B_i)/\sum_{i}\lambda(B_{i})}}$
based on the following guideline that the size of $\annt^{(i)}$ is proportional to $\lambda(B_i)$
 \begin{equation}
    \label{eq:picond}
         \frac{M_i}{\sum_{i=1}^Q M_i} = \frac{\lambda(B_i)}{\sum_{i=1}^Q \lambda(B_i)}
    \end{equation}
    where $M_i=\# \annt^{(i)}$.
% In the end, the union of $\annt^{(i)}$'s is returned as the anchor net $\annt$. 
}
This guideline is necessary for proving the property of Anchor nets in Theorem \ref{thm:anchornet}.

\begin{algorithm}[tbhp]
	\caption{\it Anchor net construction}
	\emph{Input:} Given dataset $X=\{x_{1},\dots,x_{n}\}\subset \mathbb{R}^d$ with $n$ data points, {net size $m$}\\
	\emph{Output:} Anchor net $\anntXp$
		\label{alg:anchornet}
        \begin{algorithmic}[1]
		\STATE Create a low discrepancy set $\mathcal{T}=\{t_{1},t_{2},\ldots,t_{s}\}$ with $s=O(m)$ points in the smallest rectangular box $B_0$ that contains $X$
		\STATE Initialize $G_i=\{\}$ for $i=1,2,\dots,s$.
		\FOR{$j=1,2,\dots,n$}
		\STATE Find index $i$ such that  $i=\argmin\limits_{k=1,\dots,s}\Vert x_j-t_k\Vert_{\infty}$  
% 		\COMMENT{Assume $i$ is the smallest index}
		\STATE Update $G_i=G_i\cup\{x_j\}$
		\ENDFOR
		\STATE Check the number of nonempty $G$-sets: $G_1,\dots,G_Q$
		\FOR{$i=1,2,\dots,Q$}
		\STATE Find the smallest closed box $B_i$ that contains $G_i$
		and compute its Lebesgue measure $\lambda(B_i)$
		%\COMMENT{Assume $B_i$'s are mutually non-overlapping}
		\ENDFOR
		\STATE Choose $Q$ low discrepancy sets $\annt^{(1)}\subset B_1,\dots,\annt^{(Q)}\subset B_Q$ such that $\#\annt^{(i)}=\ceil{ {m*\lambda(B_i)/\sum_{i}\lambda(B_{i})}}$
		\RETURN $\anntXp = \bigcup\limits_{i=1}^Q \anntpi$  
	\end{algorithmic}
\end{algorithm}
In Steps 1 and 11, the choice of the particular low discrepancy set is determined by the user. Options include Halton sequences, digital nets, tensor grids, etc. More details on the  implementation of Algorithm \ref{alg:anchornet} are discussed in Section \ref{subsec:implementation}. See Figure \ref{fig:AN} for an illustration of anchor nets with increasing net size $m$ constructed for a 2D  highly non-uniform synthetic dataset.

\begin{figure}
    \centering
    \includegraphics[scale=0.17]{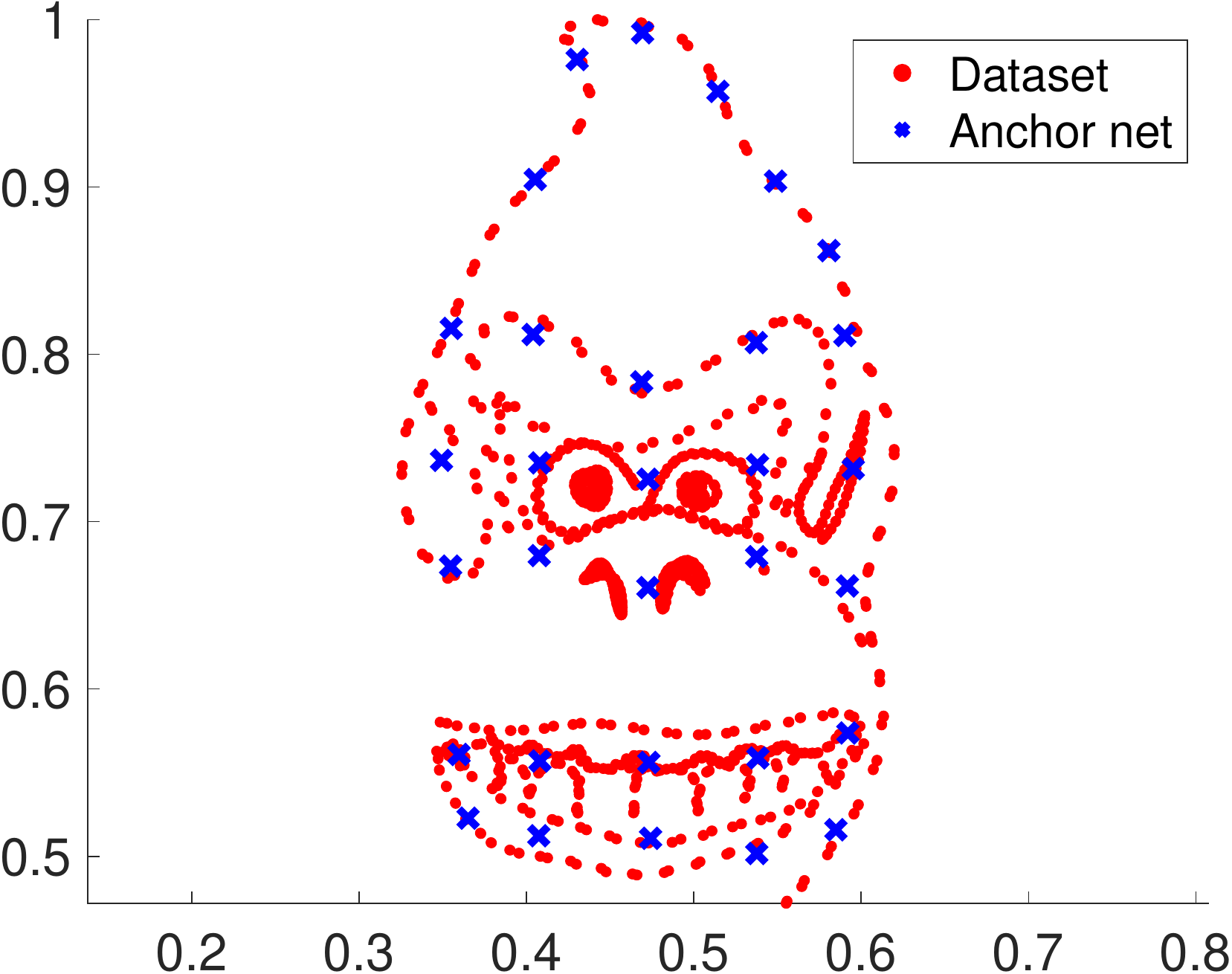}
    \includegraphics[scale=0.17]{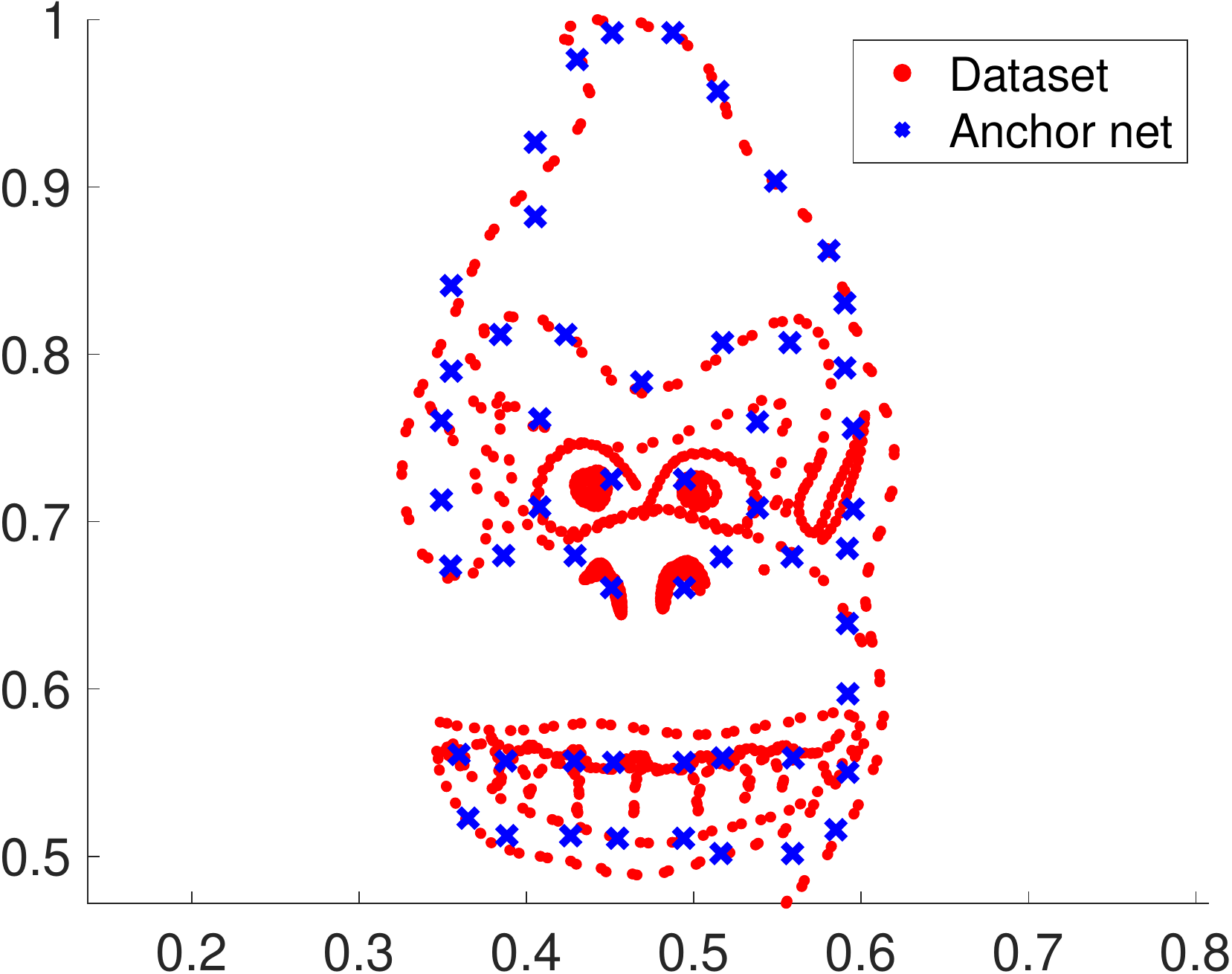}
    \includegraphics[scale=0.17]{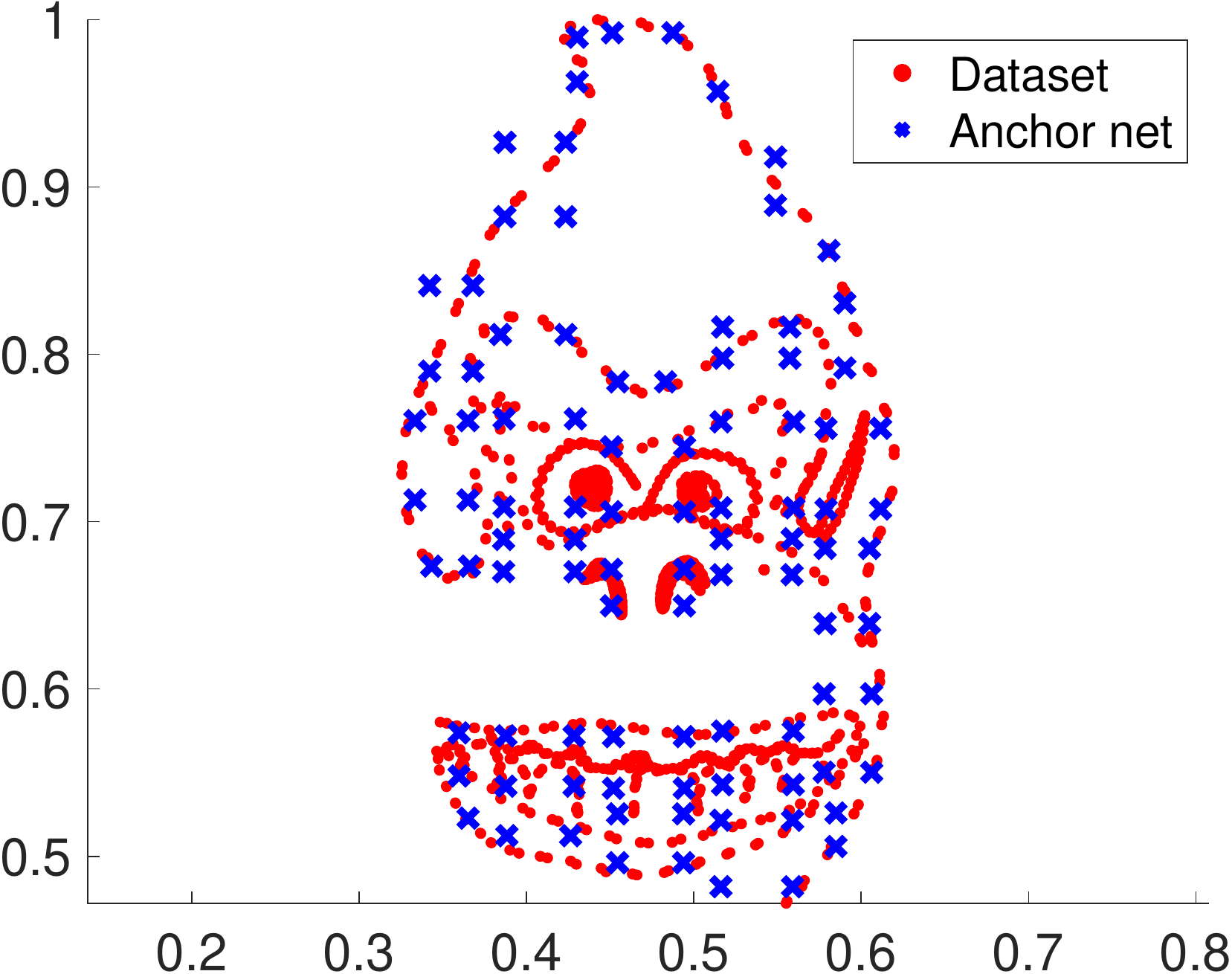}
    \includegraphics[scale=0.17]{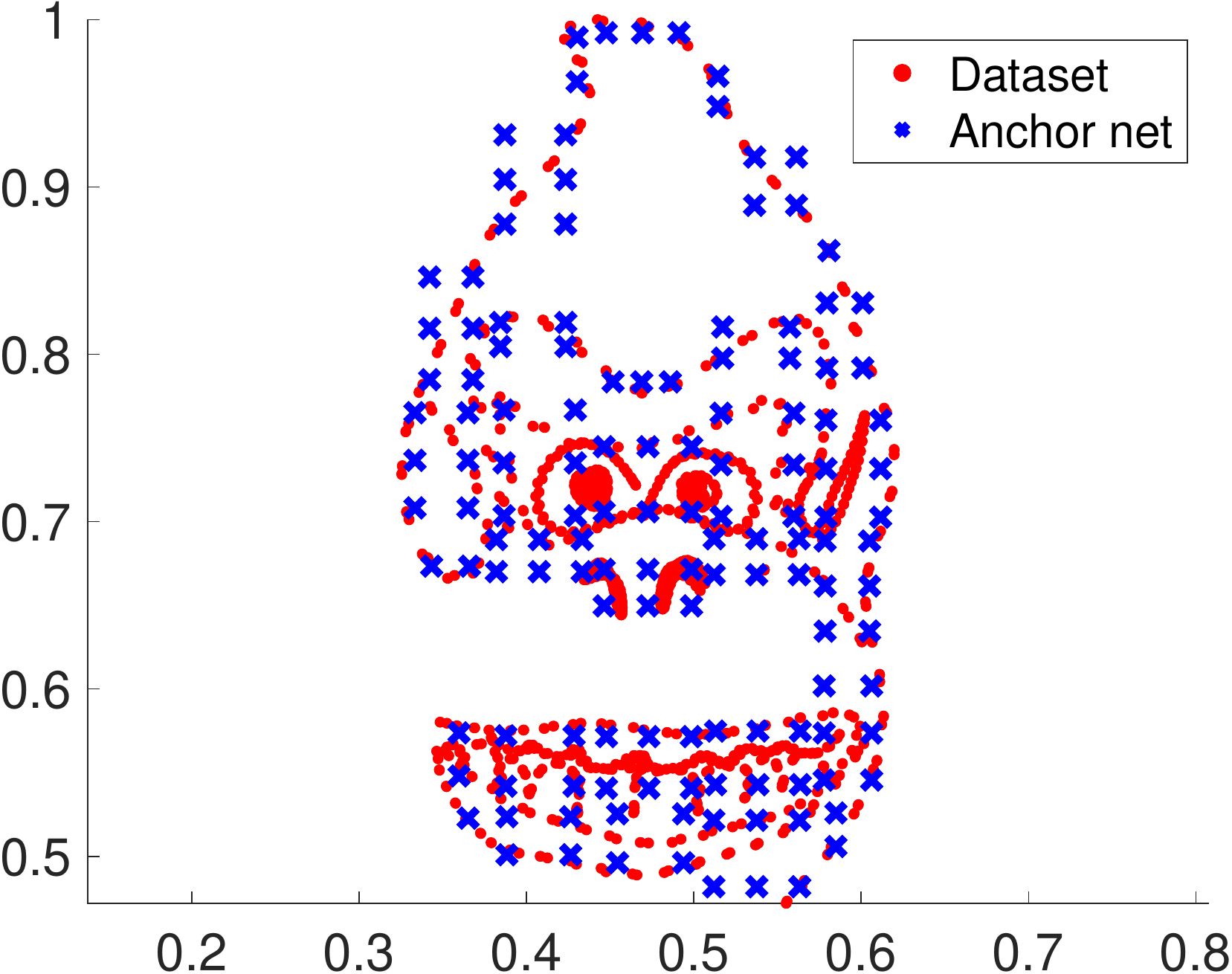}
    \includegraphics[scale=0.17]{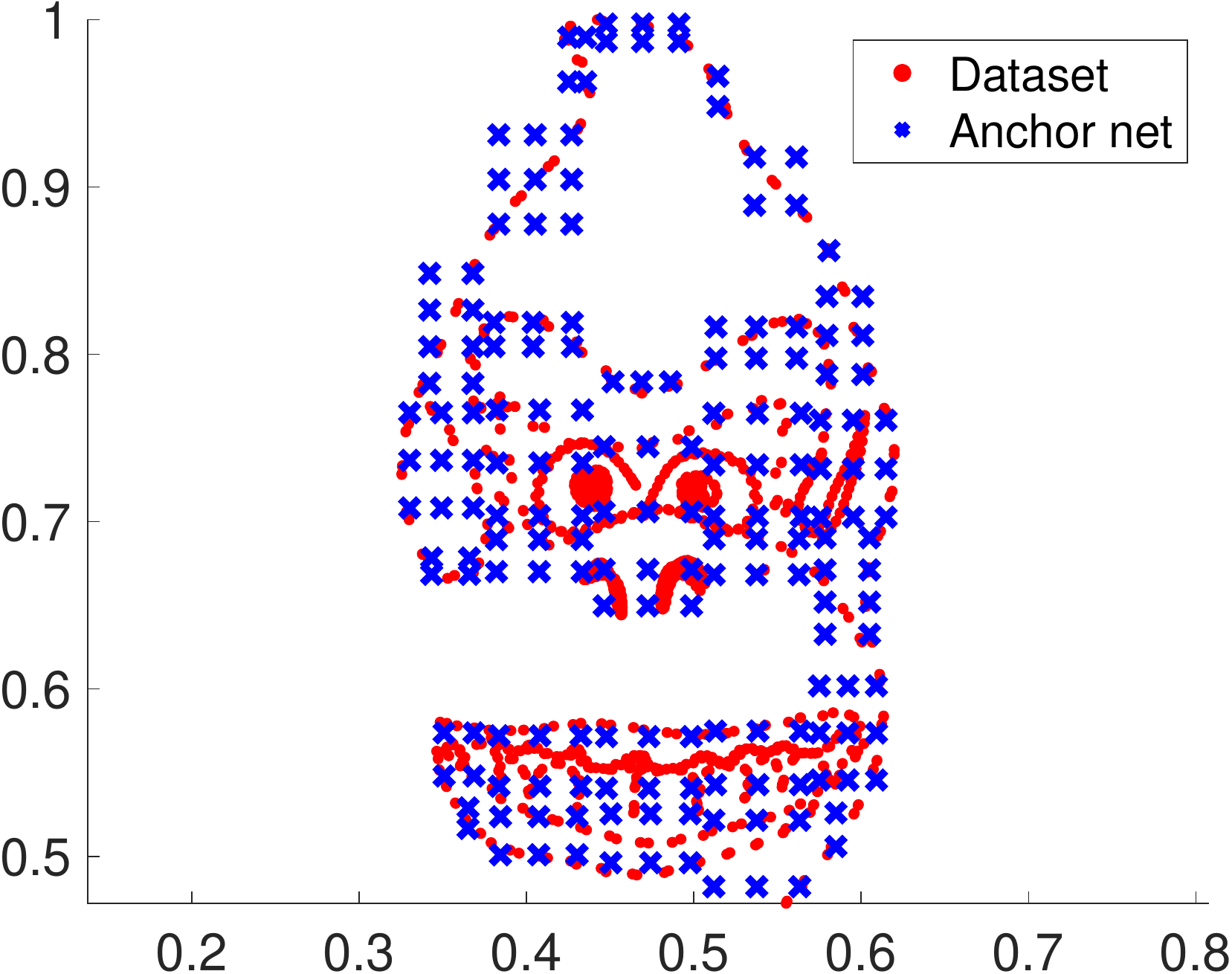}
    \includegraphics[scale=0.17]{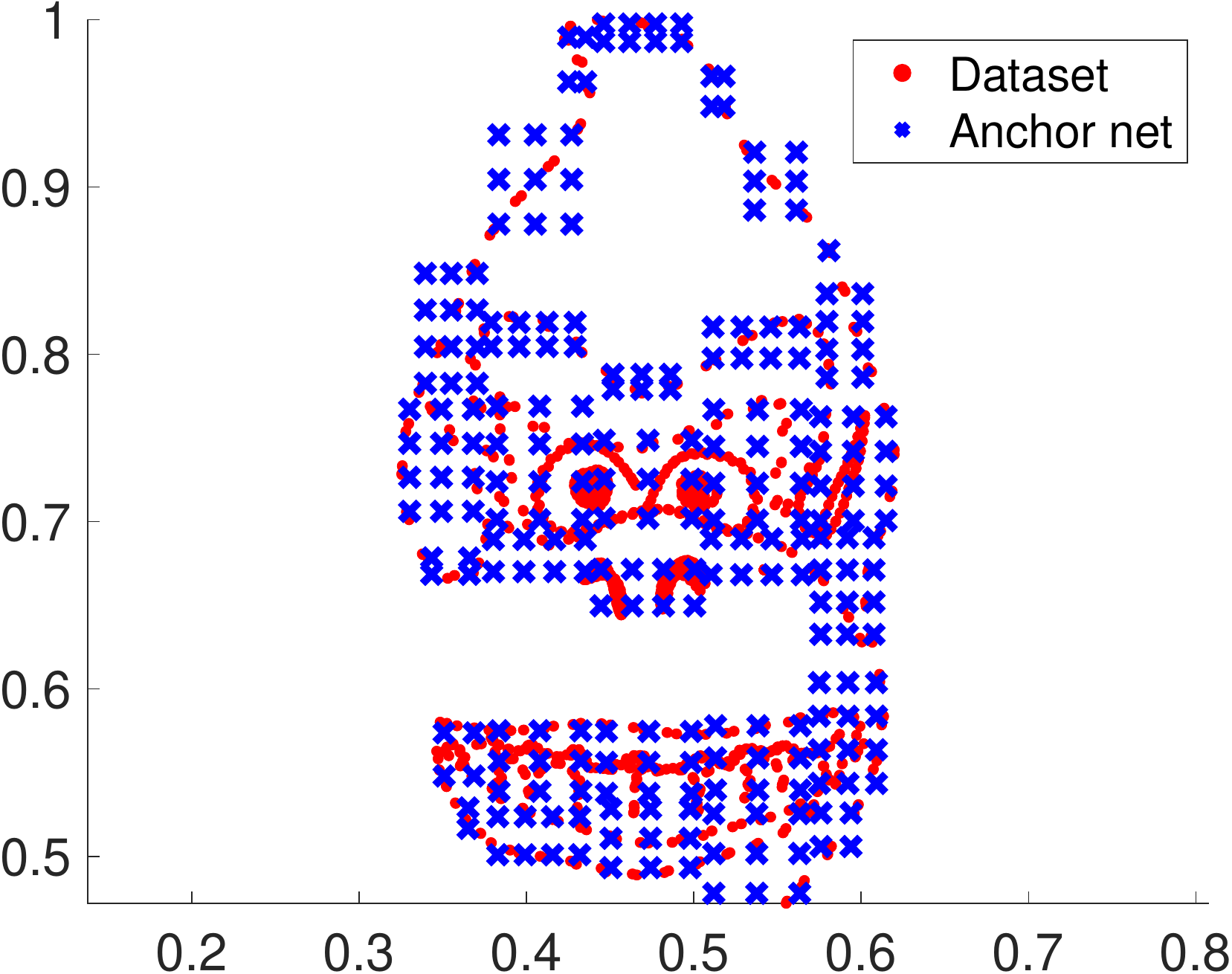}
    \includegraphics[scale=0.17]{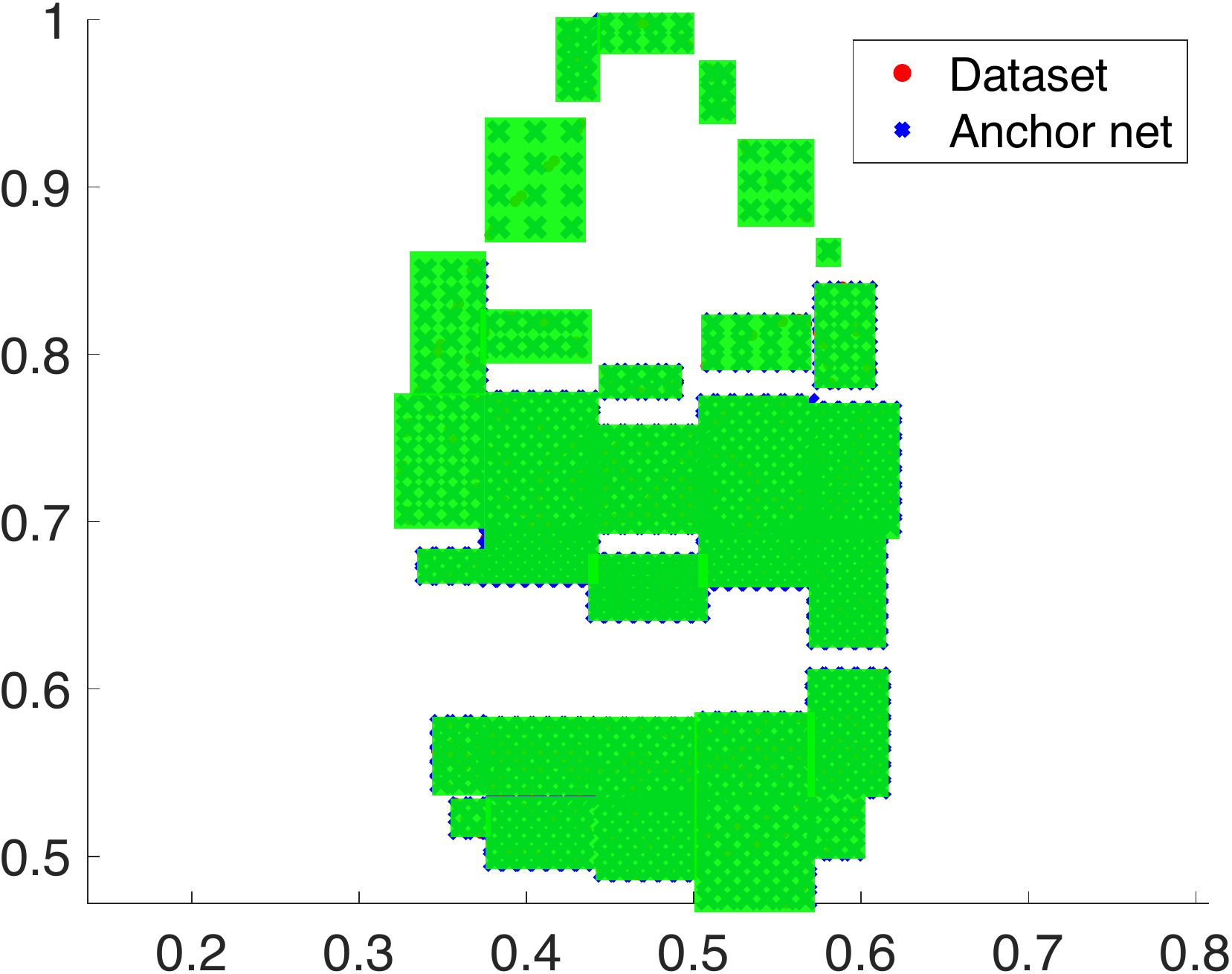}
    \caption{An illustration of six anchor nets (blue 'x') of increasing net sizes $m$ and $\Omega =\bigcup\limits_{i=1}^{Q} B_i$ (green) on a  highly non-uniform synthetic dataset (red dots).}
    \label{fig:AN}
\end{figure}

{Next we prove major properties of $\anntXp$ returned by Algorithm \ref{alg:anchornet}. 
The main result is stated in Theorem \ref{thm:anchornet}.

First we prove the following lemma.}
\begin{lemma}
\label{lm:Omega}
    For $i=1,\dots,Q$, define
    \[
        A_{\epsilon}^{(i)} := \limsup_{M_i\to\infty} \left\{ x\in\mathbb{R}^d: \distinf{x,\anntpi}\leq\epsilon  \right\},
    \]
    where $M_i=\#\anntpi$.    Then $\bigcap\limits_{\epsilon>0} A_{\epsilon}^{(i)}=B_i$.
\end{lemma}
\begin{proof}
    {Without loss of generality, assume $B_i=[0,1]^d$.}
    Our goal is to prove that $\bigcap\limits_{\epsilon>0} A_{\epsilon}^{(i)}=B_i$.
    It is easy to see that $A_{\epsilon_1}^{(i)}\subset A_{\epsilon_2}^{(i)}$ whenever $\epsilon_1 < \epsilon_2$,
    so $\bigcap\limits_{\epsilon>0} A_{\epsilon}^{(i)}$ can be viewed as the limit of sets as $\epsilon\to 0$. 
    We first show that $B_i\subset A_{\epsilon}^{(i)}$ for each $0<\epsilon<1$.
    Fix an $\epsilon\in (0,1)$.
    For an arbitrary $x\in B_i$, let $J_x$ be the box centered at $x$ with side $\epsilon$.
    Define $J=J_x\cap [0,1)^d$.
    Then $\lambda(J)\geq (\frac{1}{2})^d\lambda(J_x) = \frac{1}{2^d}\epsilon^d$.
    Since $\anntpi$ is a low discrepancy set,  $D^*_{M_i}(\anntpi)\to 0$ as $M_i\to\infty$.
    Therefore, for the tolerance $\frac{1}{5^d}\epsilon^d$, if $M_i$ is large enough, we have 
    \[
        \abs{\#(\anntpi\cap J)/M_i-\lambda(J)} \leq D^*_{M_i}(\anntpi) < \frac{\epsilon^d}{5^d} < \lambda(J),
    \]
    which implies that $\anntpi\cap J\neq \emptyset$.
    Hence there is a point in $\anntpi$ whose $l^{\infty}$ distance to $x$ is within $\epsilon$, 
    i.e., 
    \begin{equation}
    \label{eq:xanntpi}
        \distinf{x,\anntpi}\leq\epsilon.
    \end{equation}
    Note that \eqref{eq:xanntpi} is true as long as $M_i$ is large enough. 
    Consequently, there are infinitely many $M_i$ such that \eqref{eq:xanntpi} holds true.
    According to the definition of $\limsup$, it follows that
    \[
        x\in \limsup_{M_i\to\infty} \left\{ x\in\mathbb{R}^d: \distinf{x,\anntpi}\leq\epsilon  \right\} = A_{\epsilon}^{(i)}.
    \]
    This shows that $B_i\subset A_{\epsilon}^{(i)}$ since $x$ is arbitrary in $B_i$.
    Because $\epsilon>0$ is arbitrary, we see that 
    $B_i\subset \bigcap\limits_{\epsilon>0} A_{\epsilon}^{(i)}.$
    It remains to prove the other direction: $\bigcap\limits_{\epsilon>0} A_{\epsilon}^{(i)}\subset B_i$.
    This is equivalent to the fact that: if $x\notin B_i$, then $x\notin \bigcap\limits_{\epsilon>0} A_{\epsilon}^{(i)}$.
    Now suppose $x\notin B_i=[0,1]^d$. Then $\distinf{x,B_i}=\delta>0$ for some positive constant $\delta$. 
    We know that $\anntpi\subset B_i$, so $\distinf{x,\anntpi}\geq\delta>0$ for any $M_i$.
    Therefore, $x\notin A^{(i)}_{\delta}$, which yields that $x\notin \bigcap\limits_{\epsilon>0} A_{\epsilon}^{(i)}$.
    Now the second direction is proved and we conclude that 
    $\bigcap\limits_{\epsilon>0} A_{\epsilon}^{(i)}=B_i$.
\end{proof}

The next lemma is a property of the generalized discrepancy.
\begin{lemma}
\label{lm:D12}
    Let $S_1$ and $S_2$ be two finite subsets of $\Omega_1\subset [0,\infty)^d$ and $\Omega_2\subset [0,\infty)^d$, respectively.
    Suppose $\lambda(\Omega_1\cap\Omega_2)=0$ and $S_1\cap S_2=\emptyset$.
    If $D^*_{N_1,\Omega_1}(S_1) < \epsilon$ and $D^*_{N_2,\Omega_2}(S_2) < \epsilon$, where $N_i=\# S_i$,
    then 
    \begin{equation}
    \label{eq:D12} 
        D^*_{N_1+N_2,\Omega_1\cup\Omega_2}(S_1\cup S_2) < \abs{\frac{N_1}{N_1+N_2}-\frac{\lambda(\Omega_1)}{\lambda(\Omega_1)+\lambda(\Omega_2)}} + \epsilon.
    \end{equation}
\end{lemma}
\begin{proof}
    Denote $\lambda_i=\lambda(\Omega_i)$ with $i=1,2$.
    Let $\mathcal{J}$ be the family of boxes as in Definition \ref{def:genDN}.
    For any $J\in\mathcal{J}$, define 
    \[
        A_i:=\#(S_i\cap J),\quad b_i:=\lambda(\Omega_i\cap J),\quad i=1,2.
    \]
    According to Definition \ref{def:genDN} and the assumptions in the claim,
    it suffices to show that
     \begin{equation}
    \label{eq:toshow12}
    \begin{aligned}
        \abs{\dfrac{\#( (S_1\cup S_2)\cap J )}{N_1+N_2}-\dfrac{\lambda((\Omega_1\cup\Omega_2)\cap J)}{\lambda(\Omega_1\cup\Omega_2)}} 
        &= \abs{\frac{A_1+A_2}{N_1+N_2}-\frac{b_1+b_2}{\lambda_1+\lambda_2}}\\
        &< \abs{\frac{N_1}{N_1+N_2}-\frac{\lambda_1}{\lambda_1+\lambda_2}} +\epsilon.
    \end{aligned}
    \end{equation}
    Note first that the definition of $D^*_{N_i,\Omega_i}(S_i)$ yields
    \begin{equation}
    \label{eq:know12} 
        \abs{\frac{A_i}{N_i}-\frac{b_i}{\lambda_i}}
        \leq D^*_{N_i,\Omega_i}(S_i) 
        < \epsilon,\quad i=1,2.
    \end{equation}
    It is easy to see that
    \[
    \begin{aligned}
        \frac{A_1+A_2}{N_1+N_2}-\frac{b_1+b_2}{\lambda_1+\lambda_2} 
        &= 
        \frac{ N_1 }{ N_1+N_2 }\cdot \frac{ A_1 }{ N_1 } + \frac{ N_2 }{ N_1+N_2 }\cdot \frac{ A_2 }{ N_2 } \\
        &- \left( \frac{ \lambda_1 }{ \lambda_1+\lambda_2 }\cdot \frac{ b_1 }{ \lambda_1 } + \frac{ \lambda_2 }{ \lambda_1+\lambda_2 }\cdot \frac{ b_2 }{ \lambda_2 }  \right).
    \end{aligned}
    \]
    Together with \eqref{eq:know12}, we deduce that 
    \[
    \begin{aligned}
        &\abs{\frac{A_1+A_2}{N_1+N_2}-\frac{b_1+b_2}{\lambda_1+\lambda_2}} 
        < \\
        &\abs{ \left(\frac{ N_1 }{ N_1+N_2 } - \frac{ \lambda_1 }{ \lambda_1+\lambda_2 }\right) \cdot \frac{ A_1 }{ N_1 } 
        + \left(\frac{ N_2 }{ N_1+N_2 } - \frac{ \lambda_2 }{ \lambda_2+\lambda_2 }\right) \cdot \frac{ A_2 }{ N_2 } } + \epsilon\\
%        &= \abs{ \left(\frac{ N_1 }{ N_1+N_2 } - \frac{ \lambda_1 }{ \lambda_1+\lambda_2 }\right) \cdot \left( \frac{ A_1 }{ N_1 } - \frac{ A_2 }{ N_2 } \right) } + \epsilon \\
    &\leq  \abs{ \frac{ N_1 }{ N_1+N_2 } - \frac{ \lambda_1 }{ \lambda_1+\lambda_2 } } +\epsilon,
    \end{aligned}
    \]
where we have used the fact that 
%    \[
%        \frac{ N_2 }{ N_1+N_2 } - \frac{ \lambda_2 }{ \lambda_2+\lambda_2 }     = 1-\frac{ N_1 }{ N_1+N_2 } - (1-\frac{ \lambda_1 }{ \lambda_2+\lambda_2 }) = \frac{ \lambda_1 }{ \lambda_2+\lambda_2 }-\frac{ N_1 }{ N_1+N_2 }
%    \] and that
    $\abs{\frac{ A_1 }{ N_1 } - \frac{ A_2 }{ N_2 }}\leq 1$.
    Since \eqref{eq:toshow12} is proved for any $J\in\mathcal{J}$,
    by taking a $\sup$ of the left-hand side of \eqref{eq:toshow12} over $J\in\mathcal{J}$, we conclude that \eqref{eq:D12} holds true.
\end{proof}
Based on Lemmas \ref{lm:Omega} and \ref{lm:D12}, we show in Theorem \ref{thm:anchornet} the properties of the output of Algorithm \ref{alg:anchornet}. The first property says that the region $\Omega=\bigcup\limits_{i=1}^Q B_i$ associated with anchor nets is able to compactly capture $X$  and the second property indicates that the anchor nets have low discrepancy in $\Omega$.
\begin{theorem}
\label{thm:anchornet}
{    For a given dataset $X$, let $\anntXp$ be the output of Algorithm \ref{alg:anchornet} and $N:=\#\anntXp$, $M_i:=\#\anntpi$.
    Assume that $0<\tau_1 \leq M_i/N\leq \tau_2<1$ for some constants $\tau_1,\tau_2\in (0,1)$.
    Define $\Omega=\bigcup\limits_{i=1}^Q B_i$, then
\begin{enumerate}
        \item  $\Omega$ has the following equivalent expression  \begin{equation}
    \label{eq:Omega}
        \Omega = \bigcap\limits_{\epsilon>0} \limsup\limits_{N\to\infty} \left\{ x\in\mathbb{R}^d: \distinf{x,\anntXp}\leq\epsilon \right\}.
    \end{equation}
        with 
        $\lambda(\Omega)>0$ and $X\subset \Omega$;
        \item $\lim\limits_{N\to\infty} D^*_{N,\Omega}(\anntXp) = 0$.
    \end{enumerate}
        Furthermore, if \eqref{eq:picond} holds for every $i$, then 
        \begin{equation}
        \label{eq:lowDanchor}
            D^*_{N,\Omega}(\anntXp) = O(N^{-1}(\log N)^d).
        \end{equation}
    }
\end{theorem}
\begin{proof}
    We verify that the two conditions are satisfied by $\anntXp=\bigcup\limits_{i=1}^Q \anntpi$.
    %The first condition pertains to the set  $\Omega$ in \eqref{eq:Omega}.
    
    % \[
    %     \Omega := \bigcap\limits_{\epsilon>0} \limsup\limits_{M_i\to\infty} \left\{ x\in\mathbb{R}^d: \distinf{x,\anntpi}\leq\epsilon \right\}.
    % \]
 {Since $\anntXp=\bigcup\limits_{i=1}^Q \anntpi$ and $M_i/N\in (\tau_1,\tau_2)$, we see that
    \[
        \Omega =  \bigcup\limits_{i=1}^Q \bigcap\limits_{\epsilon>0} 
        \limsup\limits_{M_i\to\infty} \left\{ x\in\mathbb{R}^d: \distinf{x,\anntpi}\leq\epsilon \right\} 
        =  \bigcup\limits_{i=1}^Q \bigcap\limits_{\epsilon>0} A_{\epsilon}^{(i)},
    \]
    where $A_{\epsilon}^{(i)}$ is defined as in Lemma \ref{lm:Omega}.
    According to Lemma \ref{lm:Omega}, 
    it follows that $\Omega = \bigcup\limits_{i=1}^Q B_i.$ In addition, we have the estimation 
    \[
        \lambda(\Omega)\geq\lambda(B_1) > 0 \quad\text{and}\quad 
        X\subset\bigcup\limits_{i=1}^Q G_i \subset \bigcup\limits_{i=1}^Q B_i = \Omega,
    \]
    which justifies the first condition.}
    
    Next we prove the second property:
\begin{equation}
\label{eq:showcond2} 
    \lim_{N\to\infty} D^*_{N,\Omega}(\anntXp) = 0.
\end{equation}
    This is proved by using Lemma \ref{lm:D12}.
    %First we observe that $B_i$ and $B_j$ do not overlap if $i\neq j$.
    Assume at this moment $Q=2$. 
    Then $\anntXp=\annt^{(1)}\cup \annt^{(2)}$, $\Omega=B_1\cup B_2$.
    We deduce from Lemma \ref{lm:D12} that
\begin{equation}
\label{eq:lowDineq}
    \lim_{N\to\infty} D^*_{N,\Omega}(\anntXp) 
    \leq \lim_{N\to\infty} \abs{\frac{M_1}{M_1+M_2}-\frac{\lambda(B_1)}{\lambda(B_1)+\lambda(B_2)}} + \lim_{N\to\infty} \max_{i=1,2} D^*_{M_i,B_i}(\anntpi) ,
\end{equation}
   {where the first term in the upper bound goes to zero due to \eqref{eq:picond}}
    and the second term also vanishes because of the fact that
    $\lim\limits_{M_i\to \infty} D^*_{M_i,B_i}(\anntpi) = 0$ and $M_i/N\in (\tau_1,\tau_2)$.
    If $Q>2$, based on the result for $Q=2$,
    we can apply Lemma \ref{lm:D12} inductively to show that the condition holds true for $Q=3,4,\dots.$
    Therefore, \eqref{eq:showcond2} is justified.
    
    {Finally it remains to prove \eqref{eq:lowDanchor}.}
    This is actually an immediate result of \eqref{eq:lowDineq}.
    Consider $Q=2$. 
    Under the above assumption, it follows from \eqref{eq:lowDineq} that 
    \begin{equation}
    \label{eq:lowDineq2}
        \lim_{N\to\infty} D^*_{N,\Omega}(\anntXp) \leq \lim_{N\to\infty} \max_{i=1,2} D^*_{M_i,B_i}(\anntpi).
    \end{equation}
    Since $\anntpi$ is a low discrepancy set in $B_i$,
    $D^*_{M_i,B_i}(\anntpi) = O(M_i^{-1}(\log M_i)^d)$.
    According to the assumption in the theorem, i.e., there are constants $\tau_1,\tau_2\in(0,1)$ such that $\tau_1 N \leq M_i \leq \tau_2 N$,
    we see that $O(M_i^{-1}(\log M_i)^d) = O(N^{-1}(\log N)^d)$.
    Therefore, \eqref{eq:lowDineq2} implies
        $D^*_{N,\Omega}(\anntXp) = O(N^{-1}(\log N)^d)$,
    which completes the proof.
%    If $Q>2$, we can set 
%    $S_1=\annt^{(1)}_{p_1}$, $S_2=\annt^{(2)}_{p_2}\cup\dots\cup\annt^{(Q)}_{p_Q}$,
%    $\Omega_1=B_1$, $\Omega_2=B_2\cup\dots\cup B_Q$.
%    Then apply Lemma \ref{lm:D12} 
\end{proof}
% \cdf{Add a few discussion here. The first property xxxx and the second property xxxx. In particular, how to intepret $\Omega$, why this is better than a unit box. Both Edmond and I feel that anchor net is constructing a uniform distribution on a discrete dataset instead of a continuous domain. You can correct us. If this is true, we can visualize this with a 2D toy example to show that the limit will compactly cover the given dataset.}

{It should be pointed out that even though the first condition in Theorem \ref{thm:anchornet} says that $\Omega$ is large enough to capture $X$, it does \emph{not} indicate that $\Omega$ will be \emph{unnecessarily} large.
Note that $\Omega$ adapts to the geometry of $X$ and can be roughly viewed as a region spanned by $X$, as illustrated in Figure \ref{fig:AN}.
For highly non-uniform datasets, sampling in $\Omega$ will be more efficient than in one single box that contains $X$.
This is because $\Omega$ nicely reflects the geometry of the dataset $X$ and thus uniform distribution (guaranteed by the second property) in $\Omega$ is expected to yield uniform distribution in $X$, as can be seen from the last subfigure in Figure \ref{fig:AN}.
% Sampling in the bounding box, however,  as the data points in $X$ can only occupy a small region of the box.
}

% subsection Selection of Landmark Points via Anchor Net (end)

\subsection{Anchor net method}
\label{sub:ANN}
%The analysis in Section \ref{sec:theory} reveals 
%how the geometry of landmark points affects the numerical rank of the corresponding submatrix.
%Following the same line, 
%in this subsection, we show why the uniform sampling, which totally ignores the geometry of the dataset, is not able to capture the numerical rank of the original kernel matrix for geometrically unbalanced distributions.
%Therefore, according to Theorem \ref{thm:errorbound} or Corollary \ref{cor:errorbound}, the \nys method based on uniform sampling fails to produce good approximations in this case.
In this section we propose the anchor net method for selecting landmark points and prove its computational complexity.
The anchor net method starts with the construction of an anchor net for the given
dataset $X$ and then search for the landmark points in the vicinity of the anchor net. The algorithm is presented in Algorithm \ref{alg:anchornetmethod}. 
\begin{algorithm}[tbhp]
	\caption{\it Anchor net method}
	\emph{Input:} Dataset $X=\{x_{1},\dots,x_{n}\}\subset \mathbb{R}^d$, integer $m$\\
	\emph{Output:} The set of landmark points $S$
	\begin{algorithmic}[1]
		\STATE {Apply Algorithm \ref{alg:anchornet} with net size $m$ to construct the anchor net $\anntXp$ for $X$}
		\FOR{each point $y$ in $\anntXp$}
		\STATE Find $x_i$ s.t., $x_i=\argmin\limits_{x_k\in X}||x_k-y||_{\infty}$
% 		\COMMENT{Let $i$ be the smallest index }
		\STATE Update $S=S\cup\{x_i\}$
		\ENDFOR
		\RETURN $S$
	\end{algorithmic}
	\label{alg:anchornetmethod}
\end{algorithm}

Since the landmark points are selected in the vicinity of the Anchor net in Algorithm \ref{alg:anchornetmethod}, the selected landmark points are uniformly spread inside the dataset. In Proposition \ref{prop:complexity}, we show that the computational cost of Algorithm \ref{alg:anchornetmethod} scales linearly in $n$.
\begin{proposition}
\label{prop:complexity}
% Let $m$ in Algorithm \ref{alg:anchornet}, respectively.
The complexity of the anchor net method described in Algorithm \ref{alg:anchornetmethod} with net size $m$ is $O(mdn)$.
\end{proposition}
\begin{proof}
First we calculate the complexity of Algorithm \ref{alg:anchornet}. 
Since $s=O(m)$, it is easy to see that Step 1 costs $O(md)$ and the \textbf{for} loop in Steps 3--6 amounts to $O(mdn)$. %This is because in Algorithm \ref{alg:anchornet}, $s$ is a parameter independent of $d$, so it does not grow as $d\to\infty$. For example, one can always choose a low discrepancy set with 100 points no matter how large $d$ is.
Since $G_1,\dots,G_Q$ form a disjoint partition of $X$, we have $\# G_1+\dots+\# G_Q=n$.
The cost of the \textbf{for} loop in Steps 8--10 is then
$d\cdot\# G_1+\dots+d\cdot\# G_Q=dn$.
The cost of Step 11 is $d\cdot \#\annt^{(1)}+\dots+d\cdot \#\annt^{(Q)} = d O(m)$.
% Note that $Q\leq s$ because the indices of $G$-sets in Step 5 of Algorithm \ref{alg:anchornet} are associated with points in $\mathcal{T}$ and $\#\mathcal{T}=s$. 
Overall, we see that the complexity of Algorithm \ref{alg:anchornet} is $O(mdn)$.

Now we compute the overall complexity of Algorithm \ref{alg:anchornetmethod}.
Since $\#\anntXp=O(m)$, the \textbf{for} loop in Steps 2--5 of Algorithm \ref{alg:anchornetmethod} costs $O(mdn)$.
Thus we conclude that the overall complexity of Algorithm \ref{alg:anchornetmethod} is $O(mdn)$
%The cost is optimal because the storage for the dataset $X\subset \mathbb{R}^d$ requires $O(dn)$ space.
and the proof is complete.
\end{proof}

{It is known that both uniform sampling and $k$-means \nys methods tend to generate more sample points from regions with high density of points, which can \emph{not} effectively help reduce the \nys approximation error.
Different from those density-based approaches, anchor net $\anntXp$ is designed to efficiently tessellate the given data to avoid the formation of clumps.
Because of the geometric properties of anchor nets, the anchor net method can yield more accurate \nys approximation with same approximation rank, regardless of the positive-definiteness of the kernel function.
It should be emphasized that a good selection of landmark points also benefits the numerical stability of the \nys method, which significantly affects the quality of the approximation.
We discuss in Section \ref{subs:pinv} the stability issue associated with \nys method and
provide an numerical example in Section \ref{sub:DK} to demonstrate the impact of landmark points on approximation accuracy and numerical stability.
}

\subsection{Practical implementation}
In this section, we discuss several implementation details of the proposed method.
\label{subsec:implementation}
\subsubsection{Adaptive tensor grids}\label{sub:atensor}
% Though tensor interpolation provides a convenient way for computing low-rank approximations, the resulting rank may grow exponentially with dimension $d$, due to the curse of dimensionality.
% The use of sparse grid can reduce the exponential growth to a polynomial one, but the DOFs still increase rapidly in practice when the dimension is high. 
Though tensor grids display perfect uniformity, they are not used for high dimensional data due to the curse of dimensionality.
The naive construction of tensor grid by employing a parameter that specifies the number of points per direction is not practical in high dimension, since the degrees of freedom (DOFs) depend exponentially on dimension. 
In this section, we propose an adaptive tensor grid to significantly reduce the exponential growth of DOFs with dimension, which enables the practical use of tensor grids. 

% \cdf{Remember we only have one input net size $m$, need to relate m to the construction of the adaptive tensor grid.}
Instead of treating approximation in each dimension independently,
we control the \emph{total} number of nodes per direction over \emph{all} dimensions.
That is, for a nonnegative integer $p$, 
if $i_k$ is the number of nodes in the $k$th dimension, then we require $i_1+\dots+i_d = p+d$.
 This new strategy yields significantly fewer DOFs and
 results in a much slower growth of DOFs with respect to  $d$ or $p$, as illustrated in Figure \ref{fig:DOFs}. 
 An upper bound of the DOFs is given in Proposition \ref{prop:DOFs}.
%  In fact, as shown in Proposition \ref{prop:DOFs}, the DOFs grow at a rate of $O(\frac{p^d}{d^d})$, which is much slower than $O(p^d)$ in high dimensions. See Figure \ref{fig:DOFs} for an illustration. 
 \begin{proposition}
     \label{prop:DOFs}
     Let $p$ be a nonnegative integer. Consider a tensor grid in $\mathbb{R}^d$ with $i_k$ points $(i_k\geq 1)$ in the $k$th dimension
     such that $i_1+\dots+i_d=p+d$.
     Then the total number of nodes is bounded by $e^p$, i.e.,
 $i_1 i_2\dots i_d \leq \left( \frac{p+d}{d}  \right) ^d < e^p.$
 \end{proposition}
 \begin{proof}

%     The estimate follows immediately from the inequality of arithmetic and geometric means:
%     \[
%         \sqrt[d]{i_1 i_2\dots i_d} \leq \frac{1}{d}(i_1+i_2+\dots+i_d)
%     \]
% and the fact that $\left( 1+\frac{p}{d}  \right) ^{d/p} < e.$
 The second inequality in the estimate follows from the fact that
     \[ \left( 1+\frac{p}{d}  \right) ^{d/p} < e.\]
 We now prove the first inequality by induction on $d$.
 For $d=1$, the inequality automatically holds true.
 Assume that the inequality holds true for $\mathbb{R}^d$.
 For $\mathbb{R}^{d+1}$,
 there are $p+1$ possible values for $i_{d+1}$.
 That is, $i_{d+1}=k\quad\text{and}\quad i_1+\dots+i_d=p+d+1-k,
     \quad k=1,\dots,p+1.$
 Applying the induction assumption for $\mathbb{R}^d$ gives
 $i_1 i_2 \cdots i_d
         \leq \left(\frac{p+d+1-k}{d}\right)^d.$
 Hence
 \[
     \max_{ \substack{{|\mathbf{i}|=p+d+1}\\  {i_{d+1}}=k}}  i_1 i_2 \cdots i_d i_{d+1}
         \leq k\left(\frac{p+d+1-k}{d}\right)^d \leq 
   \max_{1\leq k\leq p+1} g(k),
 \]
 where 
   $g(x) := x\left(\frac{p+d+1-x}{d}\right)^d.$
 Next we show that $g(k)$ is bounded by $\left(\frac{p+d+1}{d+1}\right)^{d+1}$.
 By computing $g'(x)$,
\[    g'(x)=\left(\frac{p+d+1-x}{d}\right)^{d-1}
        \frac{p+d+1-(d+1)x}{d}
\]
 we see that $g$ has a unique maximizer at $x_*=(p+d+1)/(d+1)$ in $[0,p+1]$.
 Therefore,
 \[
     \max_{|\mathbf{i}|=p+d+1}i_1 i_2 \cdots i_d i_{d+1}\leq \max_{1\leq k\leq p+1} g(k)\leq g(x_*) =
      \left(\frac{p+d+1}{d+1}\right)^{d+1}.
 \]
 We conclude that the inequality holds for $\mathbb{R}^{d+1}$
 and proof is complete.
 \end{proof}

% %\cdf{Names? $l^{\infty}$ grid: $(1+p)^k$ grid; $p=0,1,\dots$ $l^1$ grid: new grid.}

% %The reduced growth rate of DOFs is attributed to a new measure of the multi-index when specifying a particular approximation parameter $p$, corresponding to the $l^1$ norm instead of the $l^{\infty}$ norm of the multi-index in traditional approaches.
% %which, similar to the widely used $L^1$ regularization in various applications, promotes ``sparsity".

% %TODO: write an algorithm on how to choose $p_k$ s.t. $p_1+\dots+p_d=p$?
% %Show estimates of number of grid points: both lower and upper bound!
% %lower bound is achieved for tube-like(1D) grid pts, upper bound is at uniform grid

 Figure~\ref{fig:DOFs} shows a comparison between DOFs of the uniform tensor grid  (dotted line) and the new one (solid line). In the uniform tensor grid, $p+1$ denotes the number of nodes in each dimension, while in adaptive tensor grid, $p+d$ controls the sum of numbers of nodes in each dimension.
 The left subfigure plots the DOFs with respect to dimension $d$ when $p=2,3,4,5$
 and the right subfigure plots the DOFs with respect to $p$ at different dimensions $d=2,3,4,5$. 
 It can be seen from the left plot in Figure~\ref{fig:DOFs} that the classical tensor grid (dotted line) yields exponentially increasing degrees of freedom with the dimension, while the new one (solid line) is immune to the increase of dimension.
 The right plot in Figure~\ref{fig:DOFs} shows that, compared to the old method, 
 the new method yields a much slower growth of DOFs as $p$ increases.
 We see from both figures that the new method is not sensitive to the increase of dimension $d$. {Adaptive tensor grids control the \emph{rate of increase} of DOFs across different levels of approximation by adding more intermediate levels. The numerical experiments in Section \ref{sec:experiments} demonstrate that the approximation error decreases as more DOFs are used in the adaptive tensor grid.}
 \begin{figure}[htbp]
     \centering
     \includegraphics[scale=.32]{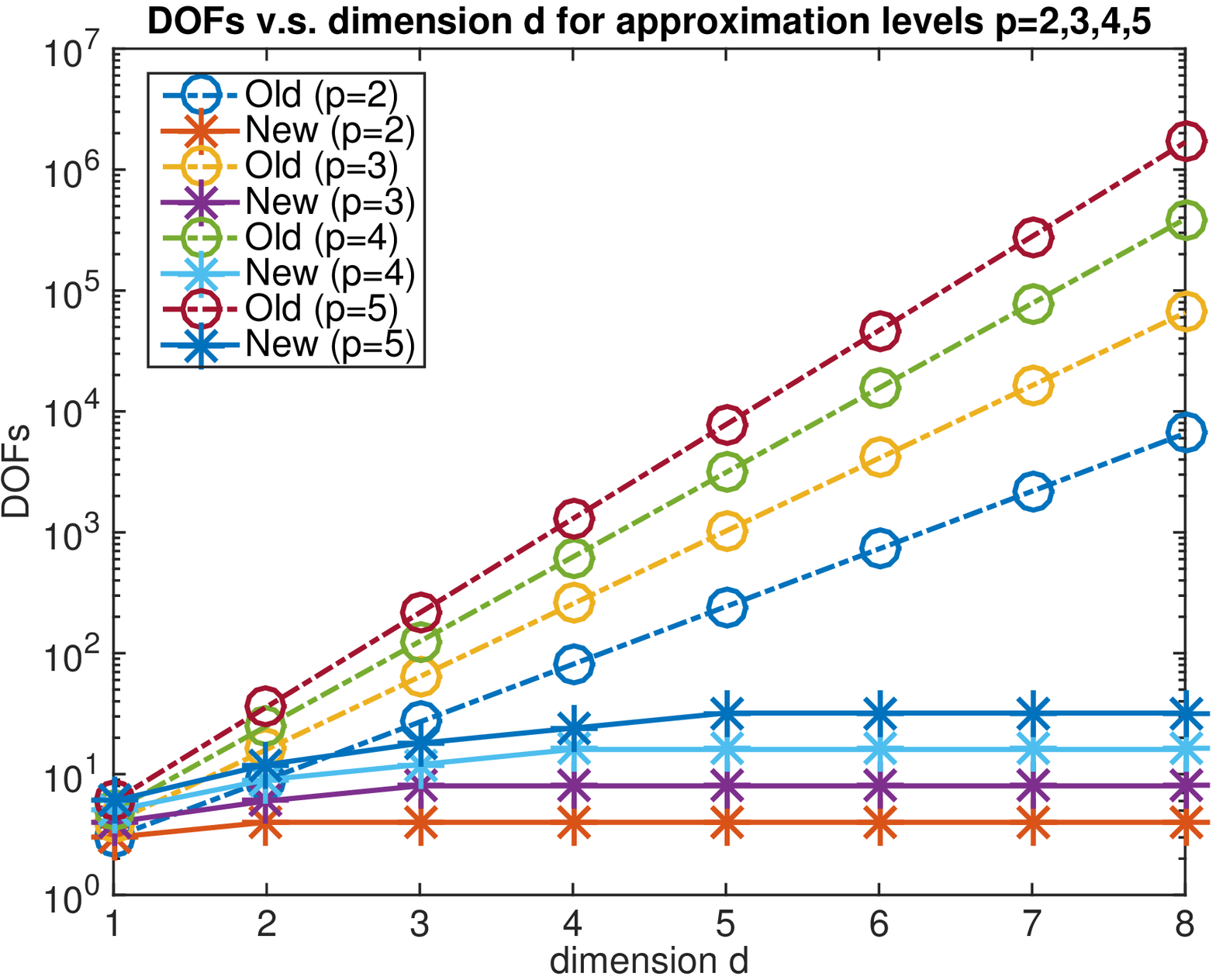}
     \hspace{0.5cm}
     \includegraphics[scale=.32]{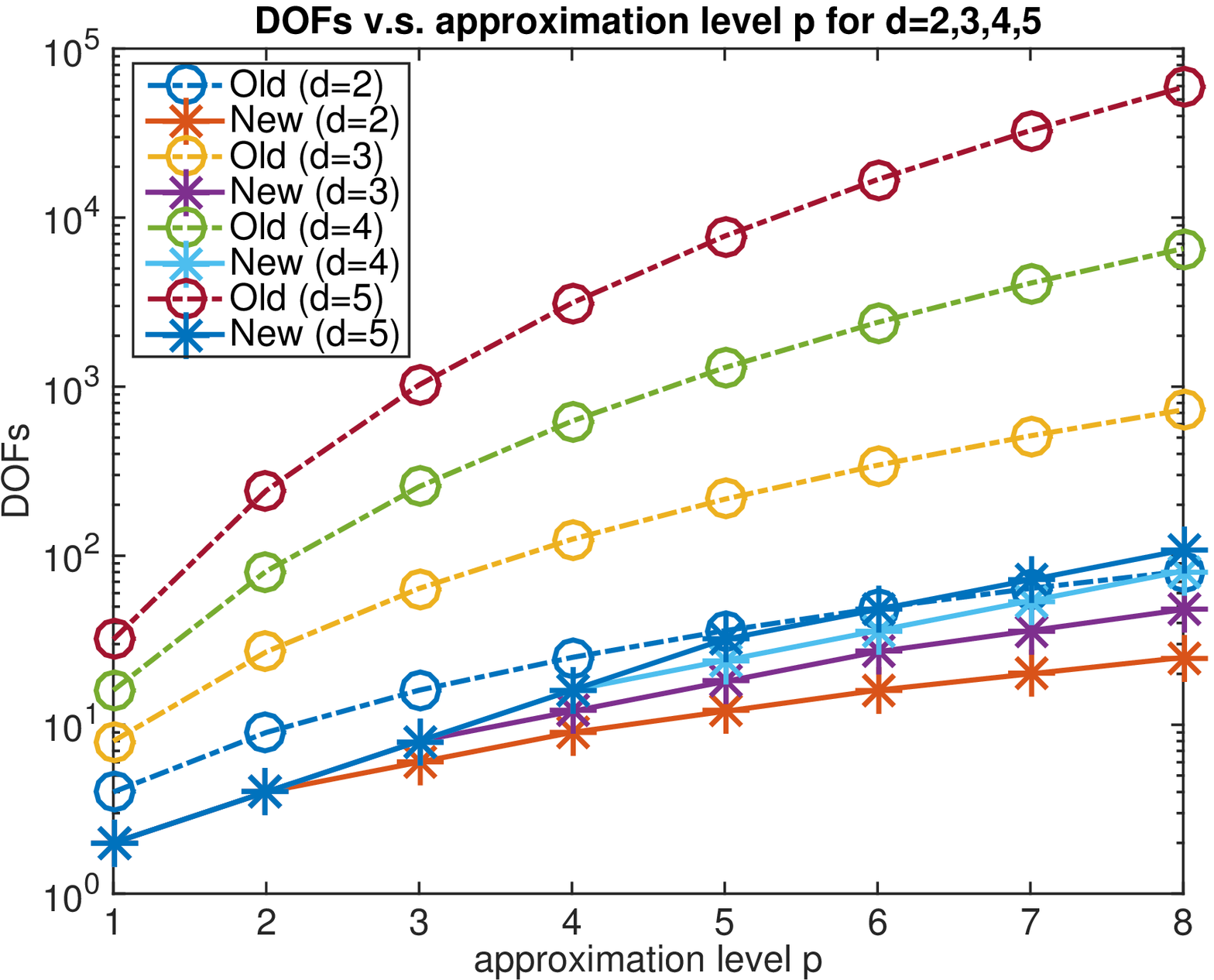}
     \caption{Left: DOFs vs dimension $d$; Right: DOFs vs $p$} 
     \label{fig:DOFs}
 \end{figure}

{Although adaptive tensor grids share the same goal as sparse grids \cite{spgrid2000} to control the number of generated nodes in high dimensions, there are several major differences between them: (1) Sparse grids use highly non-uniform nodes in the cubic domain. For example, along a specific dimension (for example, $y=0$ in the two dimensional case), the nodes are sparser in the interior and denser near the boundary. On the other hand, adaptive tensor grids tend to generate uniformly distributed nodes in the dataset. 
(2) Despite the fact that sparse grids reduce the exponential dependence on the dimension $d$ to a polynomial one, from $p^d$ to $d^p$, the actual number of degree of freedoms can still be very large even for a moderate $d$. For example, as shown in \cite{spgrid2000}, when $p$ (max number of nodes per dimension) increases from 1 to 7, the number of degrees of freedom increases from 21 to 652,065 for a dimension $d=10$ problem.
Therefore, it one wants higher accuracy by increasing $p$, significantly more DOFs will be generated. On the other hand, as shown in the right subfigure of \ref{fig:DOFs} the number of nodes increases at a much slower rate in adaptive tensor grids as the approximation level $p$ increases.
(3) Sparse grid is used for approximating functions and high dimensional integrals instead of matrix approximations, particularly \nys method for low-rank factorization. The motivation of sparse grid is to reduce the cost in approximating a continuous problem (e.g. a function, an integral) in high dimensions, while a matrix is a discrete object.}
% The new construction of tensor grids diversifies the options of low discrepancy sets. Compared to Halton sequences or digital nets, a tensor grid contains fewer parameters and is easier to compute in practice.

% {In practice, for example, in Algorithm \ref{alg:anchornet}, the number of grid points per dimension can be chosen to be proportional to the corresponding volume of the rectangular box $B_i$.}

% \subsubsection{Selection of $s$ or a practical implementation for high dimensional problems}

\subsubsection{Numerical techniques for improving stability}
\label{subs:pinv}
The \nys formula requires computing  $K_{SS}^+$, the pseudoinverse of the kernel matrix associated with the landmark points. In some cases, the resulting kernel matrix can be nearly singular, causing numerical instability when computing the exact pseudoinverse. The issue can be circumvented for SPSD kernels by regularization techniques, i.e., adding a scalar matrix $\alpha I$ with a small constant $\alpha>0$ to lift all eigenvalues to $(\alpha,\infty)$ and computing the inverse of the sum. For indefinite kernels, however, regularization is no longer effective since $K_{SS}$ may have both positive and negative eigenvalues around 0. 
A well-known method that can handle both cases is to use the $\epsilon$-pseudoinverse $K_{SS,\epsilon}^+$ in place of $K_{SS}^+$,
where $K_{SS,\epsilon}$ is derived from $K_{SS}$ by treating singular values smaller than $\epsilon$ as zeros.
The modified \nys approximation with truncated pseudoinverse then becomes
\begin{equation}
\label{eq:nyspinv}
    K_{XX}\approx K_{XS} K_{SS,\epsilon}^+ K_{SX}.
\end{equation}
Some other alternatives have also been proposed.
For example, \cite{yuji20} proposed the following QR-based approximation in place of $K_{SS}$:
\begin{equation}
\label{eq:nysQR}
    K_{XX}\approx (K_{XS}R_{\epsilon}^+)( Q^T K_{SX}),
\end{equation}
where  $K_{SS}=QR$ is the QR factorization of $K_{SS}$ and $R_{\epsilon}$ is derived from $R$ by truncating singular values smaller than $\epsilon$, similar to $K_{SS,\epsilon}$ with respect to $K_{SS}$.
{In Section \ref{sec:experiments}, we perform numerical tests to show that the truncation techniques do rectify the stability issue.
However, aside from improved stability, numerical results show that \eqref{eq:nyspinv}   impairs the accuracy of the original \nys approximation. Although \eqref{eq:nysQR} performs better than \eqref{eq:nyspinv}, the approximation  still becomes less accurate as the number of landmark points increases.
In general, numerical techniques require accurate computation of singular values close to zero for a numerically low-rank matrix and are \emph{not} able to fully resolve the structural issues on accuracy and stability.
%which are essentially caused by the poor choice of landmark points.
In this paper, we alleviate this issue by choosing a good selection of landmark points to improve the conditioning of the $K_{SS}$, as demonstrated in Section \ref{sec:experiments}.
}

\section{Numerical experiments}
\label{sec:experiments}

%\cdf{We show that for numerically low-rank kernel matrices, i.e., those whose singular values decay rapidly, existing \nys methods could easily break down (with \emph{no} regularization when computing $K_{SS}^{+}$), regardless of the positive-definiteness of the kernel matrix.}

In this section we present various experiments to demonstrate the performance of the anchor net method and the numerical instability of some \nys methods for kernel matrices with rapidly decaying singular values.
The datasets are shown in Table \ref{tab:data}.
All experiments were performed in MATLAB 2020b on a desktop with an Intel i9-9900K 3.60GHz CPU and 64 GB of RAM.
The 2-norm is used to measure the \nys approximation error in all experiments except the one in Figure \ref{fig:dk500} where 2-norm can not be computed accurately and Frobenius norm is used instead.
For probabilistic methods like uniform sampling, the error is averaged over 10 repeated runs, and in each error-rank plot, the solid line corresponds to the averaged error while the dotted line corresponds to the error in an individual run. See, for example, Figures \ref{fig:MQ} -- \ref{fig:Spline}.
For the anchor net construction, we choose the low discrepancy set to be the adaptive tensor grid discussed in Section \ref{sub:atensor} as it is straightforward to parametrize adaptive tensor grids using the sum of the number of nodes in each direction. 
%Since each $G_i$ is a subset of $X$, in Algorithm \ref{alg:anchornet}, the sizes of $\anntpi$ are in general smaller than $\mathcal{T}$. In practice, if the pair $(a,b)$ parametrizes the size of $\mathcal{T}$ and the maximum size of $\anntpi$ ($1\leq i\leq Q$), we can increase the entries in the pair alternatively to refine the anchor net, for example, $(10,2), (11,2), (11,3), (12,3)$.
In Algorithm \ref{alg:anchornet}, we choose $\mathcal{T}$ to be larger than $\anntpi$ to tessellate the dataset more efficiently, especially in high dimensions. 
{For example, empirical results show that the size of $\mathcal{T}$ can be chosen to be 2 to 20 times larger than the size of $\anntpi$, with larger ratio for higher dimensions.}

\begin{table}
\caption{Datasets ($n$ instances in $d$ dimensions)}
\label{tab:data}
\begin{center}
\begin{tabular}{ccccccc}
\toprule
 %& \multicolumn{3}{c}{Low dimensional datasets}  & \multicolumn{3}{c}{High-dimensional datasets} \\
%\cmidrule{2-7}
&  Donkey Kong  & Abalone & Anuran Calls (MFCC) & Covertype  \\
\hline
$d$ & 2  & 8  & 22 & 54 \\
\hline
$n$ & 3000 & 4177 & 7195 & 581012 \\
\bottomrule
\end{tabular}
\end{center}
\end{table}

\subsection{Indefinite kernels}
\label{sub:indefinite}
We consider the following indefinite kernels:
\[
\begin{aligned}
    \text{Multiquadrics}:\quad &\kappa(x,y) = \sqrt{|x-y|^2/\sigma^2 + 1},\\
    \text{sigmoid}:\quad &\kappa(x,y) = \tanh(x\cdot y /\sigma + 1),\\
    \text{Thin plate spline}:\quad &\kappa(x,y) = \frac{|x-y|^2}{\sigma^2}\ln\left(\frac{|x-y|^2}{\sigma^2}\right).
\end{aligned}
\]
Those kernels are commonly seen in deep learning, kernel density estimation, statistics, etc.
To the best of our knowledge, the only \nys methods that could potentially work for indefinite kernels are the uniform method \cite{nys2001} and the $k$-means \nys method \cite{nys2008kmeans,nys2010kmeans}.
Hence we compare our method to those two.
(Note that leverage-score sampling based \nys methods, such as \cite{nys2005,leverage2016,nys2017}, can not be applied here since they require the kernel matrices to be SPSD.)
The $k$-means method is implemented with an efficient vectorized function to compute $L_2$ distances between points and centroids at each iteration (Bunschoten, 1999). The iteration number is set to 5.
%For the low dimensional dataset, we also compare to the interpolation-based low-rank factorization (cf. \cite{hack2015book,smash}).
%We test different schemes on one low dimensional dataset: the 2D Donkey Kong data, and 
We test the three \nys methods over the following high-dimensional datasets from the UC Irvine Machine Learning Repository\footnote{https://archive.ics.uci.edu/ml/index.php}:
Abalone, Anuran Calls (MFCC), Covertype. See Table \ref{tab:data} for the statistics of the datasets.
The datasets are standardized to have zero mean and unit variance.
For each kernel, we choose $\sigma$ to be the half radius of the standardized dataset, where the radius is defined as the maximum distance from a point to the center. The choice ensures that the resulting kernel matrices have fast singular value decay and are thus suitable for low-rank approximations. For the Covertype dataset ($n=581012$), the \nys approximation error is evaluated over 10000 randomly sampled points from the dataset.

The error-rank plots in Figures \ref{fig:MQ} -- \ref{fig:Spline} illustrate how the \nys approximation error changes as the number of landmark points increases.
The computational cost associated with each method is shown in the error-time plots in Figures \ref{fig:MQTime} -- \ref{fig:SplineTime}, where the runtime for each method is computed over ten repeated runs and the approximation error for uniform \nys method is averaged over ten runs.

We have the following observations regarding the accuracy and stability of the \nys schemes under comparison for approximating different kinds of \emph{indefinite} kernel matrices.
\begin{enumerate}
    \item According to Figures \ref{fig:MQ} -- \ref{fig:SplineTime}, we see that, for different indefinite kernels and datasets, the anchor net method achieves overall the best accuracy for a given approximation rank (i.e., the number of landmark points) and requires least computation time. It is overall more stable than uniform sampling and $k$-means methods. We also note that the advantage of anchor net method is more prominent for large scale high dimensional datasets like Covertype.
    \item Compared to uniform sampling and anchor net methods, the $k$-means clustering can be quite unstable as one increases the approximation rank, as illustrated in Figures \ref{fig:MQ}-right, \ref{fig:Sigmoid}, \ref{fig:Spline}. This is due to the heuristic and iterative nature of the $k$-means clustering: the computed cluster centers after a few iterations are unpredictable, and it's hard to predict whether the final output can yield a better \nys approximation accuracy than the initial guess. 
    \item It's easy to see from Figure \ref{fig:MQ}-right and Figure \ref{fig:Sigmoid}-right that, the anchor net method converges much faster to a fixed accuracy.
    \item For large-scale datasets in high dimensions (e.g. Covertype), the $k$-means \nys method is quite slow and is outperformed by uniform sampling according to Figures \ref{fig:MQTime}(right) -- \ref{fig:SplineTime}(right).
    \item For the sigmoid kernel with MFCC dataset in Figure \ref{fig:SigmoidTime}-middle, all three \nys schemes display oscillatory behaviors, but the anchor net method stays at a much lower error level, so it actually oscillates with a much smaller amplitude than the other two methods.
    \item We see that indefinite kernel matrices are in general much harder for \nys methods to approximate than SPSD matrices. This is because  indefinite kernels have both positive  and negative eigenvalues around the origin. As a result, the \nys approximation is more sensitive to numerical instability. Existing general \nys schemes (uniform sampling and $k$-means) can be quite unstable for indefinite kernels, while the anchor net method is very robust and meanwhile achieves better accuracy with less computational cost. 
\end{enumerate}

\begin{figure}[htbp] 
    \centering 
    \includegraphics[scale=.22]{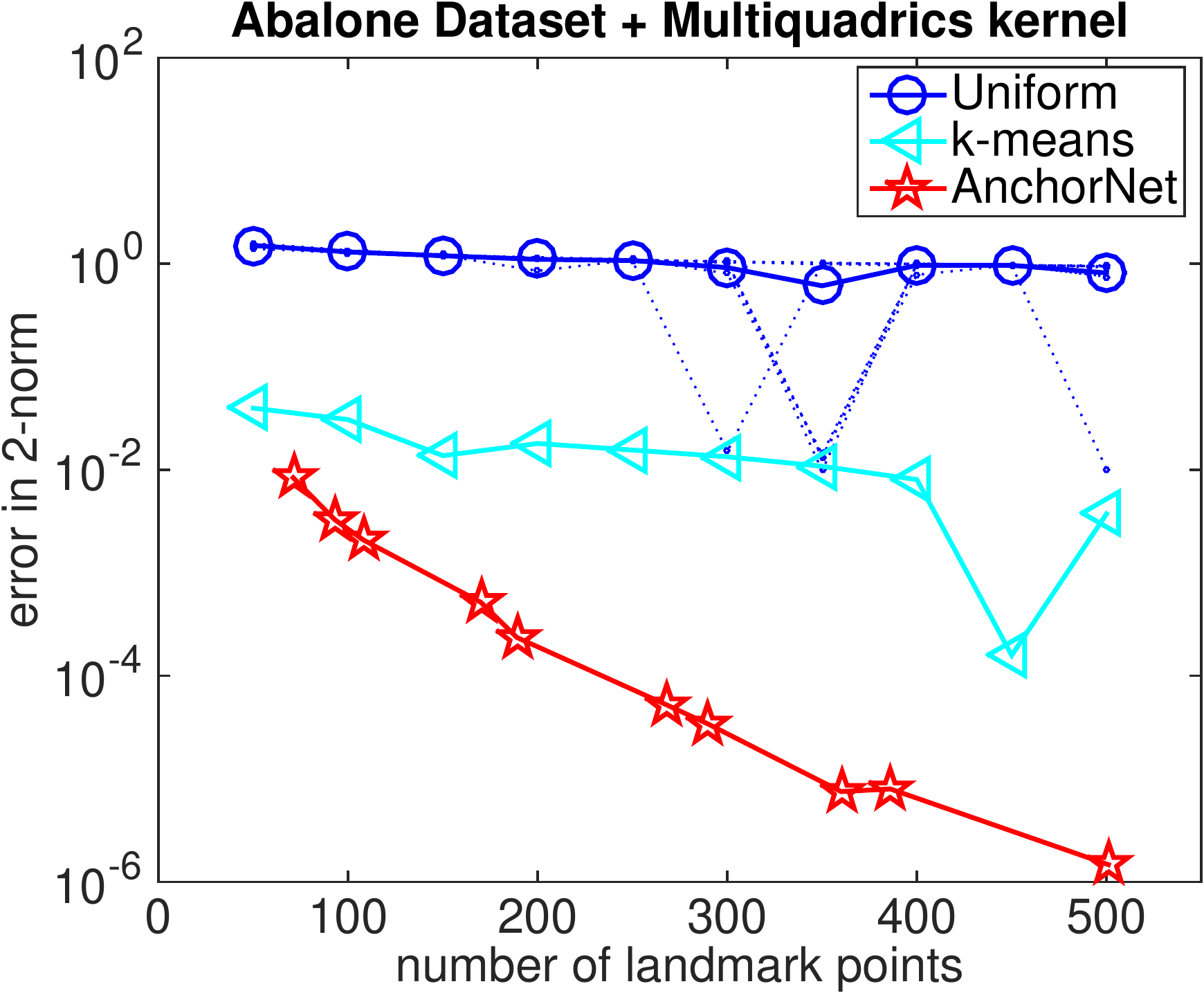}
    \includegraphics[scale=.22]{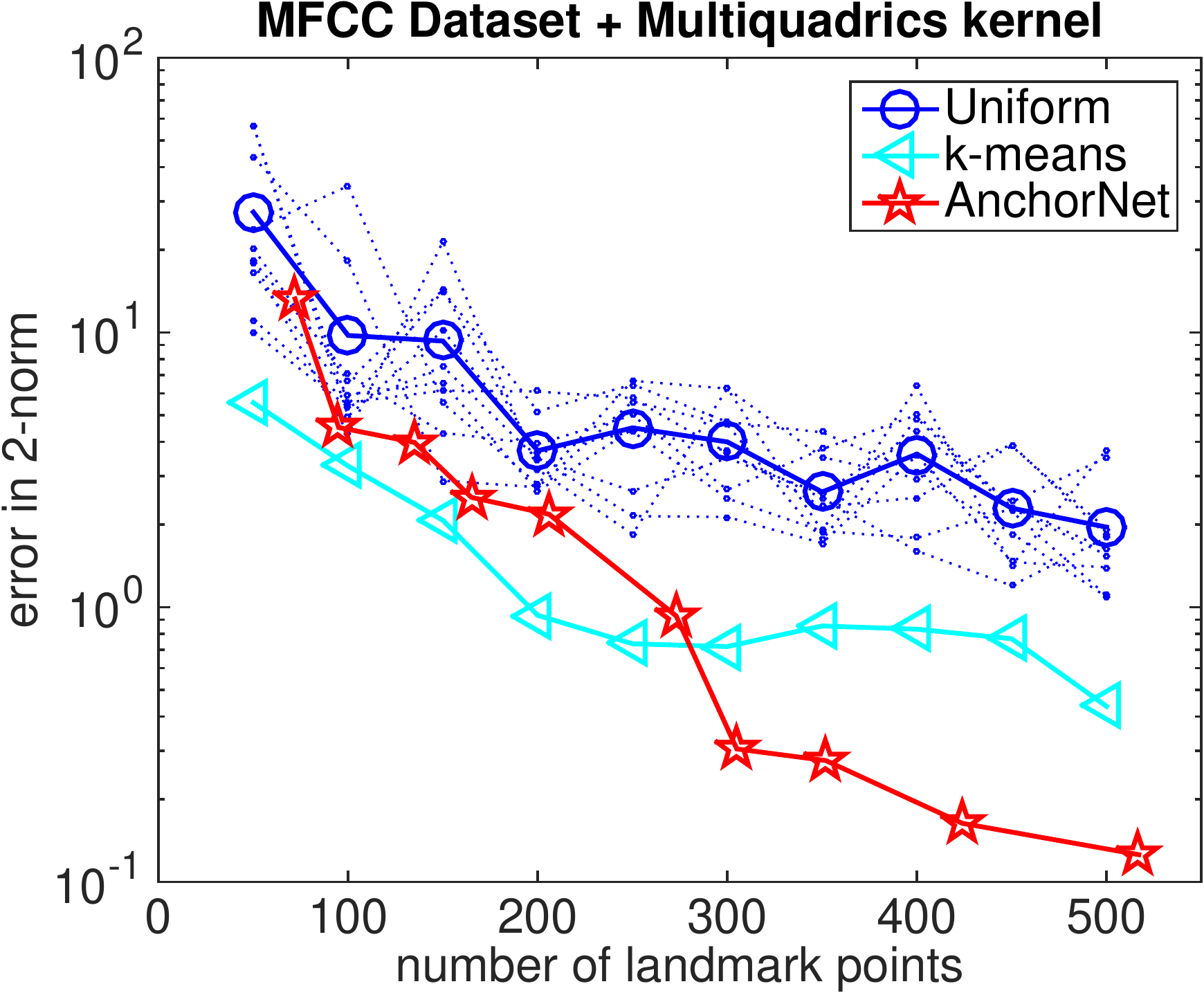}
    \includegraphics[scale=.22]{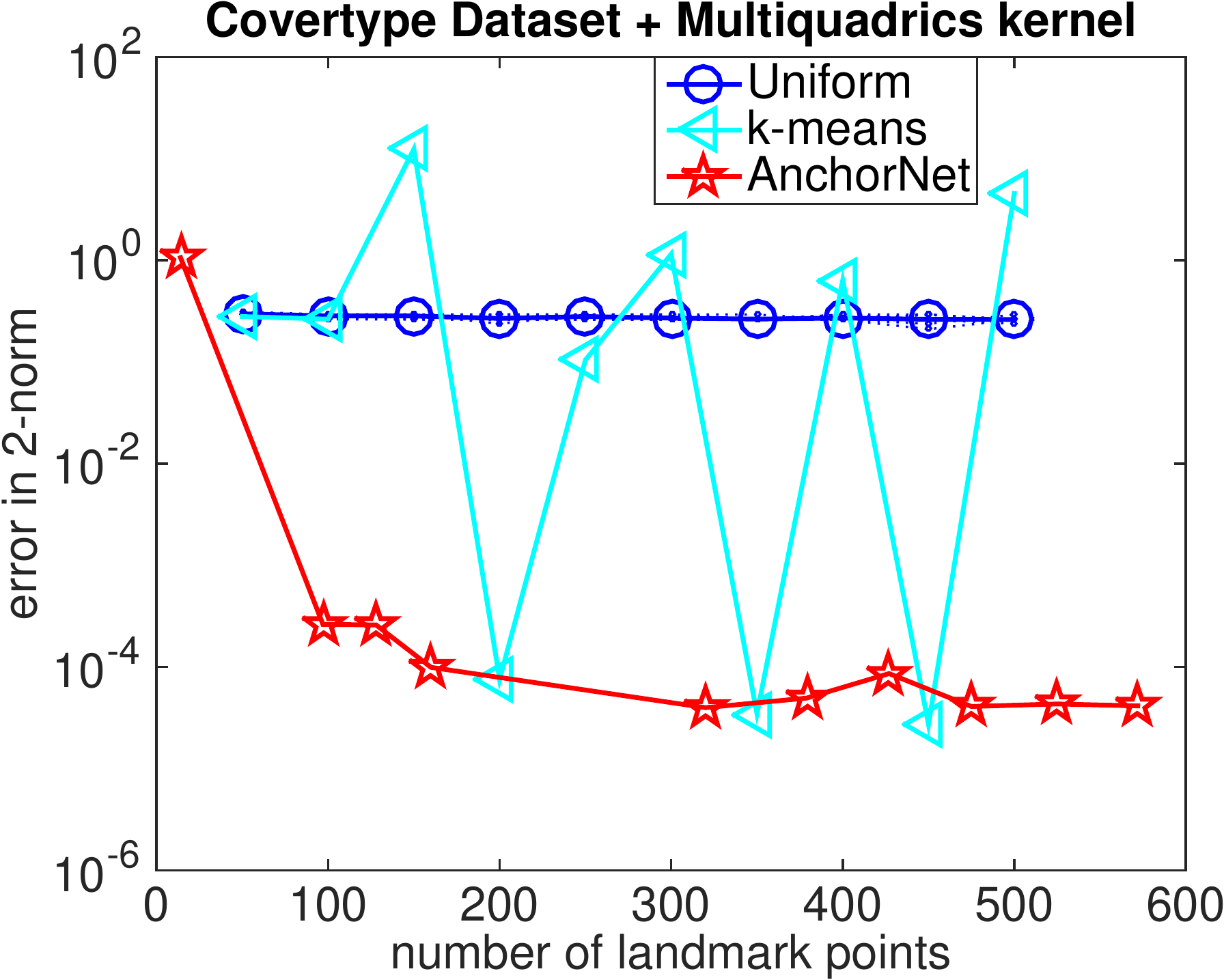}
    \caption{Multiquadrics: Abalone (left), MFCC (middle), Covertype (right)}
    \label{fig:MQ} 
\end{figure}

\begin{figure}[htbp] 
    \centering 
    \includegraphics[scale=.22]{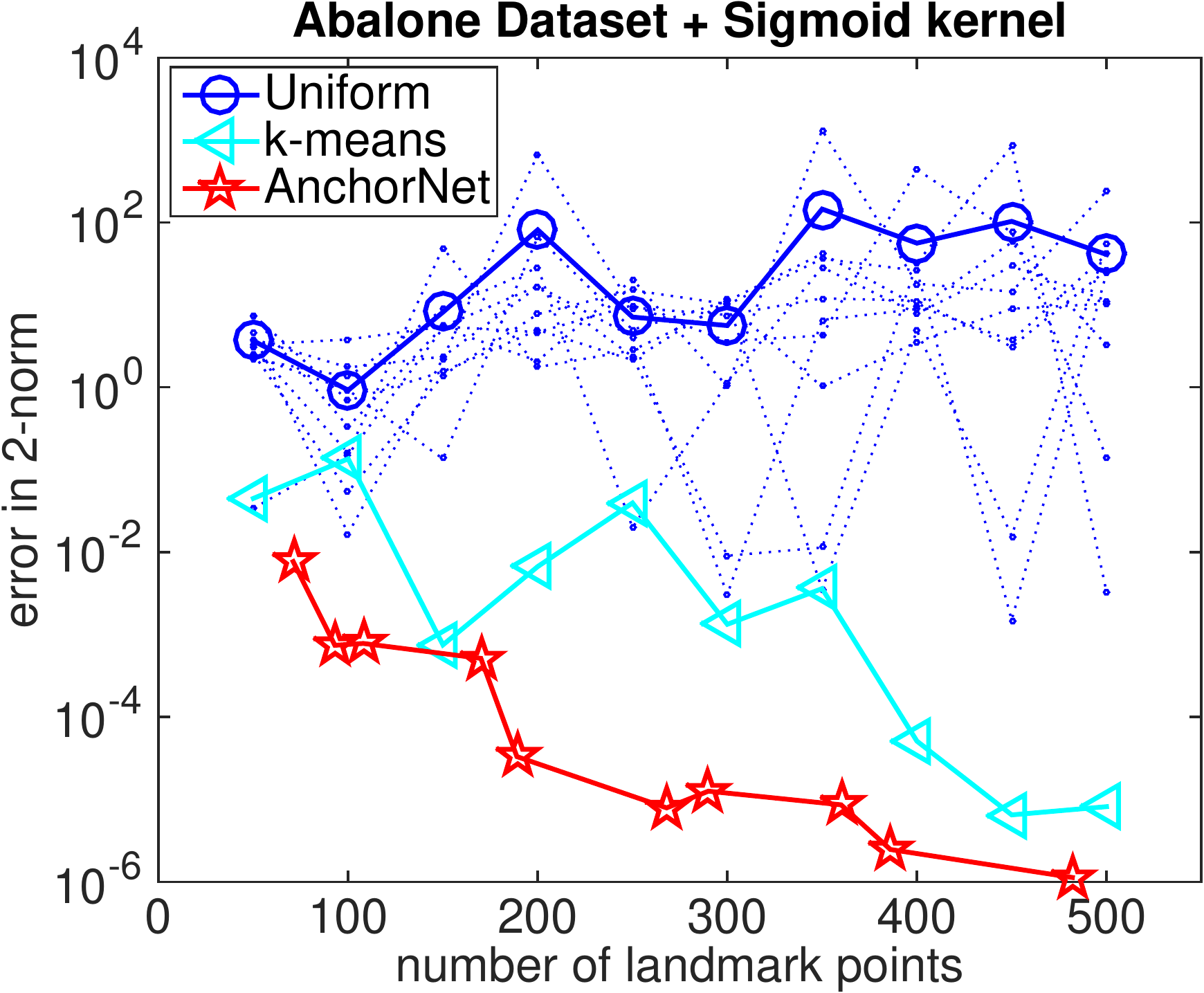}
    \includegraphics[scale=.22]{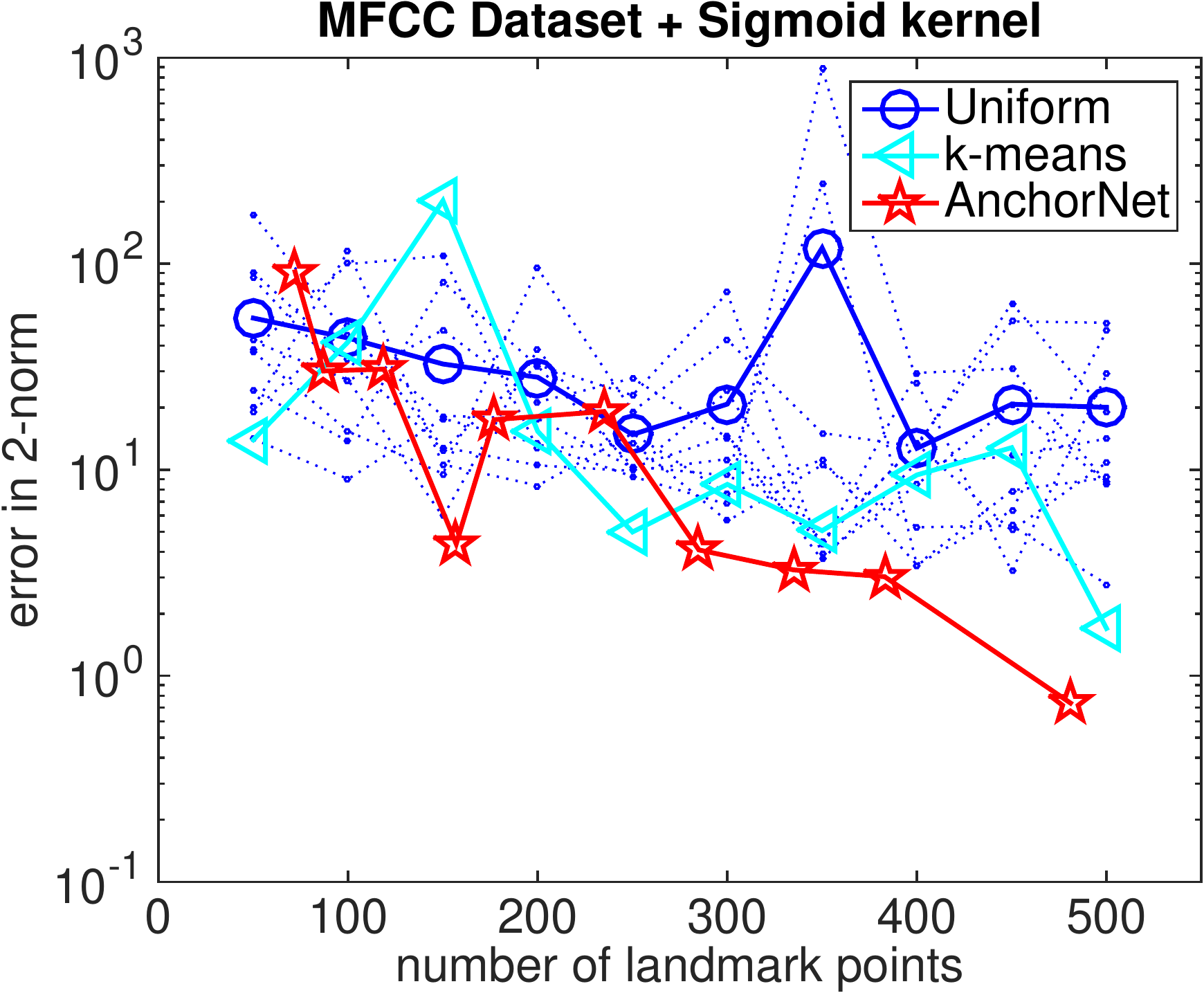}
    \includegraphics[scale=.22]{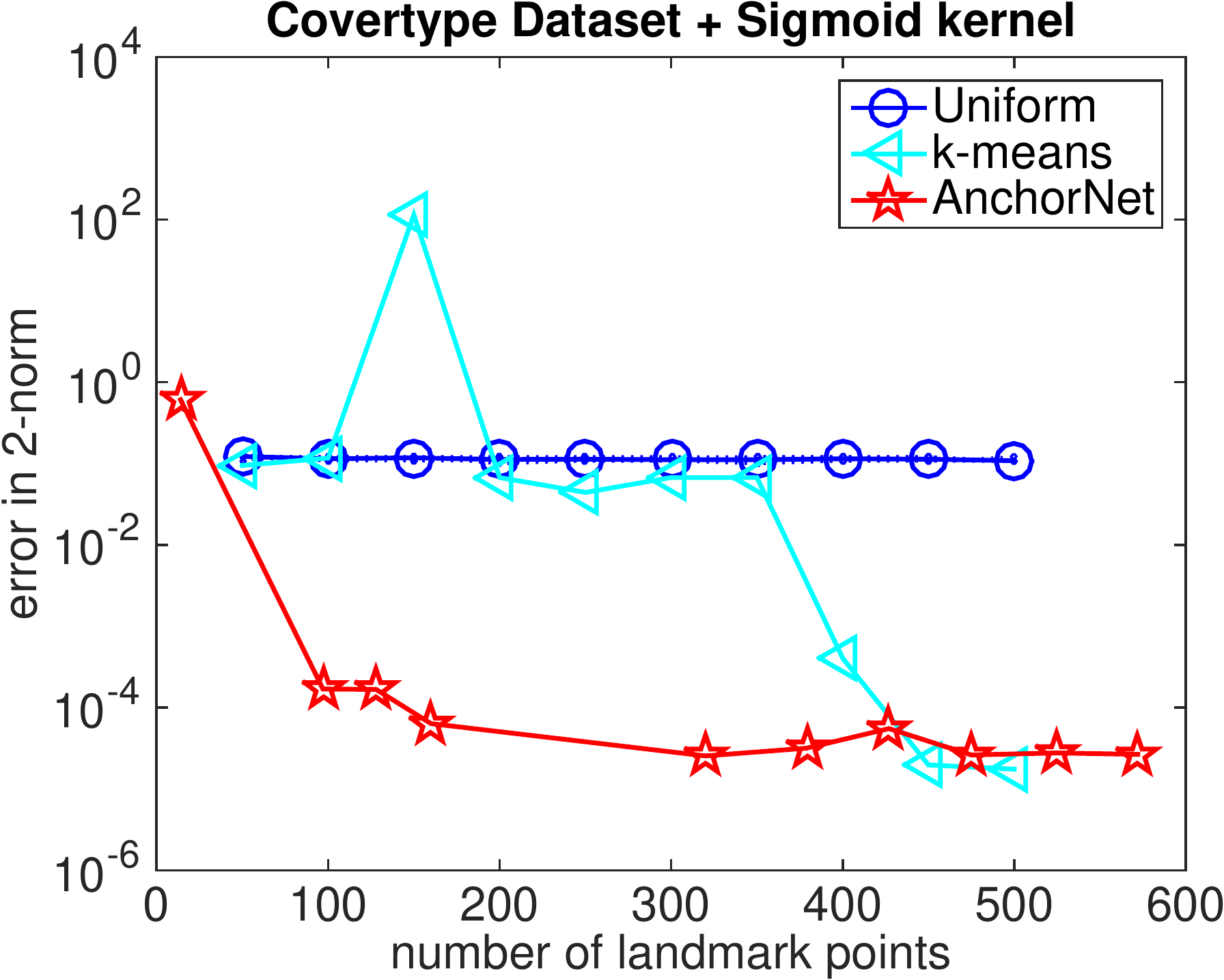}
    \caption{Sigmoid: Abalone (left), MFCC (middle), Covertype (right)}
    \label{fig:Sigmoid}
\end{figure}

\begin{figure}[htbp] 
    \centering 
    \includegraphics[scale=.22]{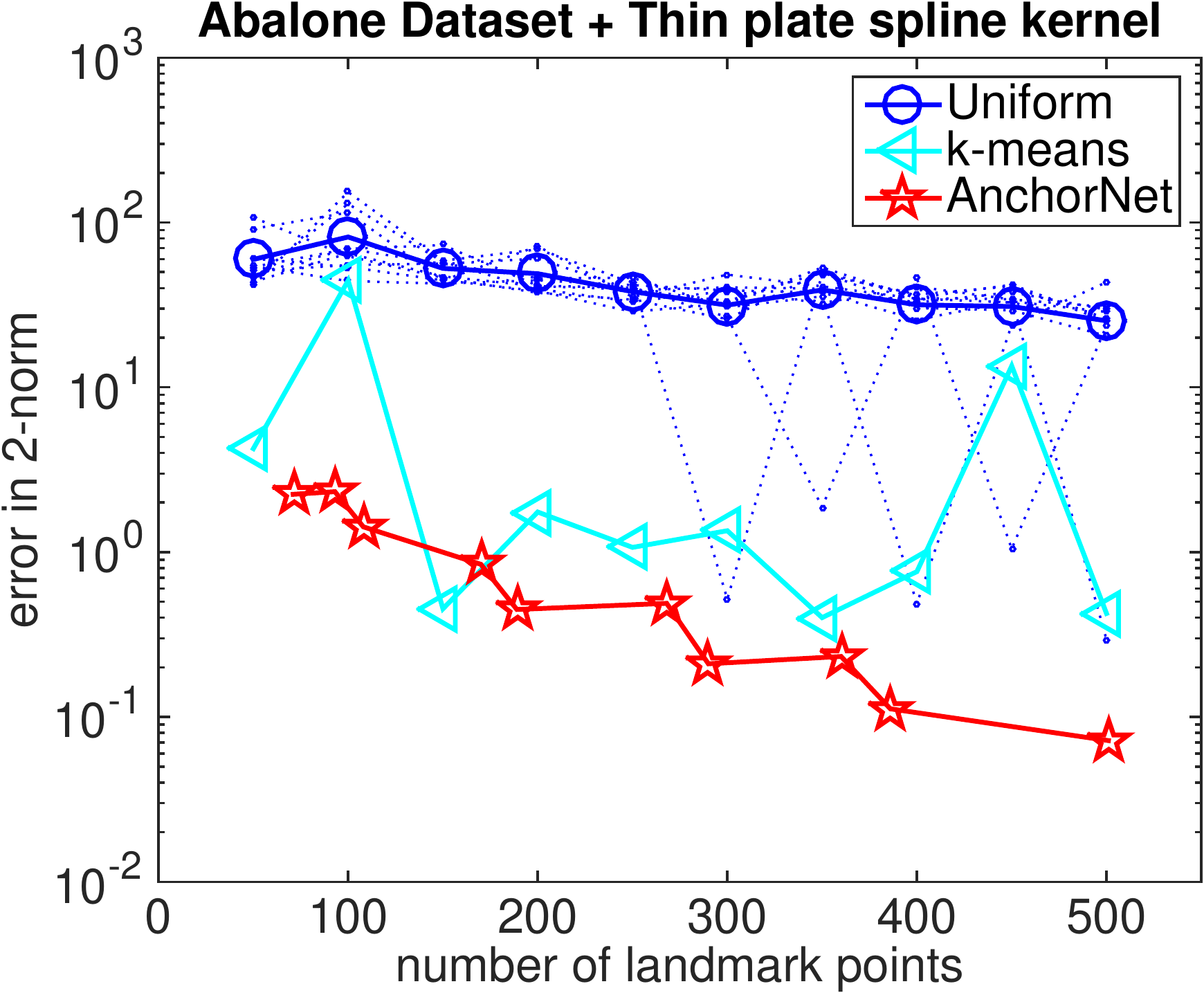}
    \includegraphics[scale=.22]{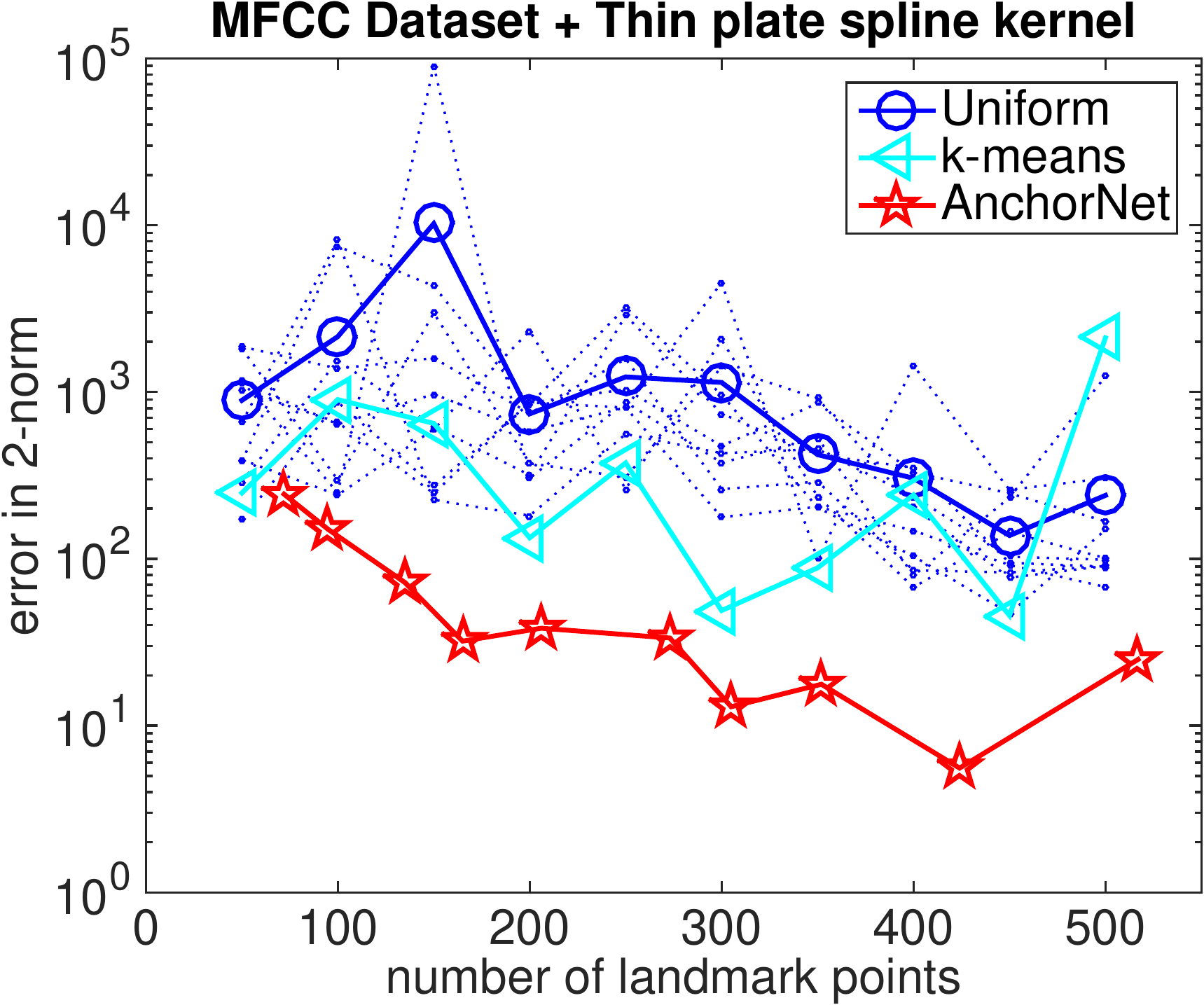}
    \includegraphics[scale=.22]{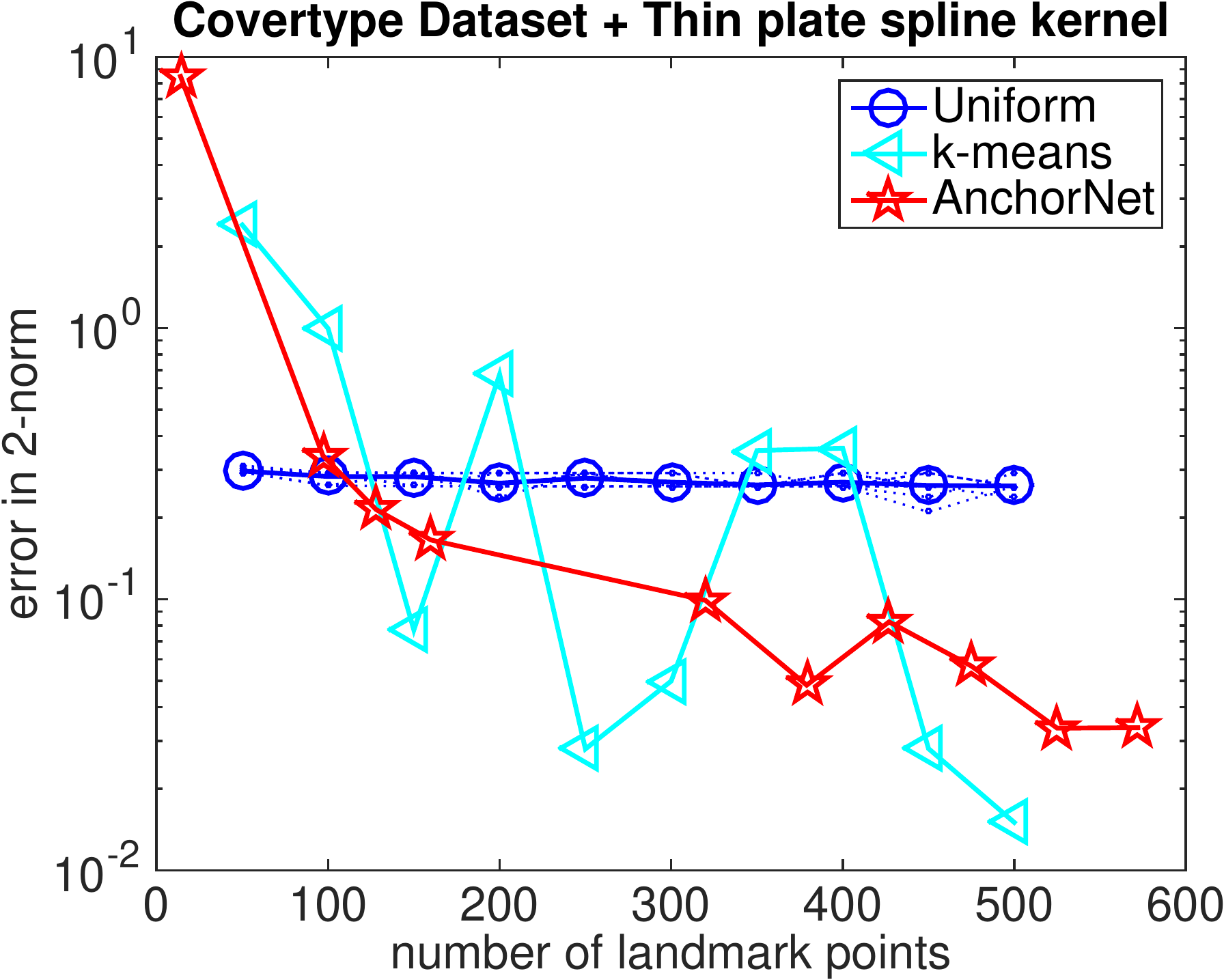}
    \caption{Thin plate spline: Abalone (left), MFCC (middle), Covertype (right)}
    \label{fig:Spline} 
\end{figure}

\begin{figure}[htbp] 
    \centering 
    \includegraphics[scale=.22]{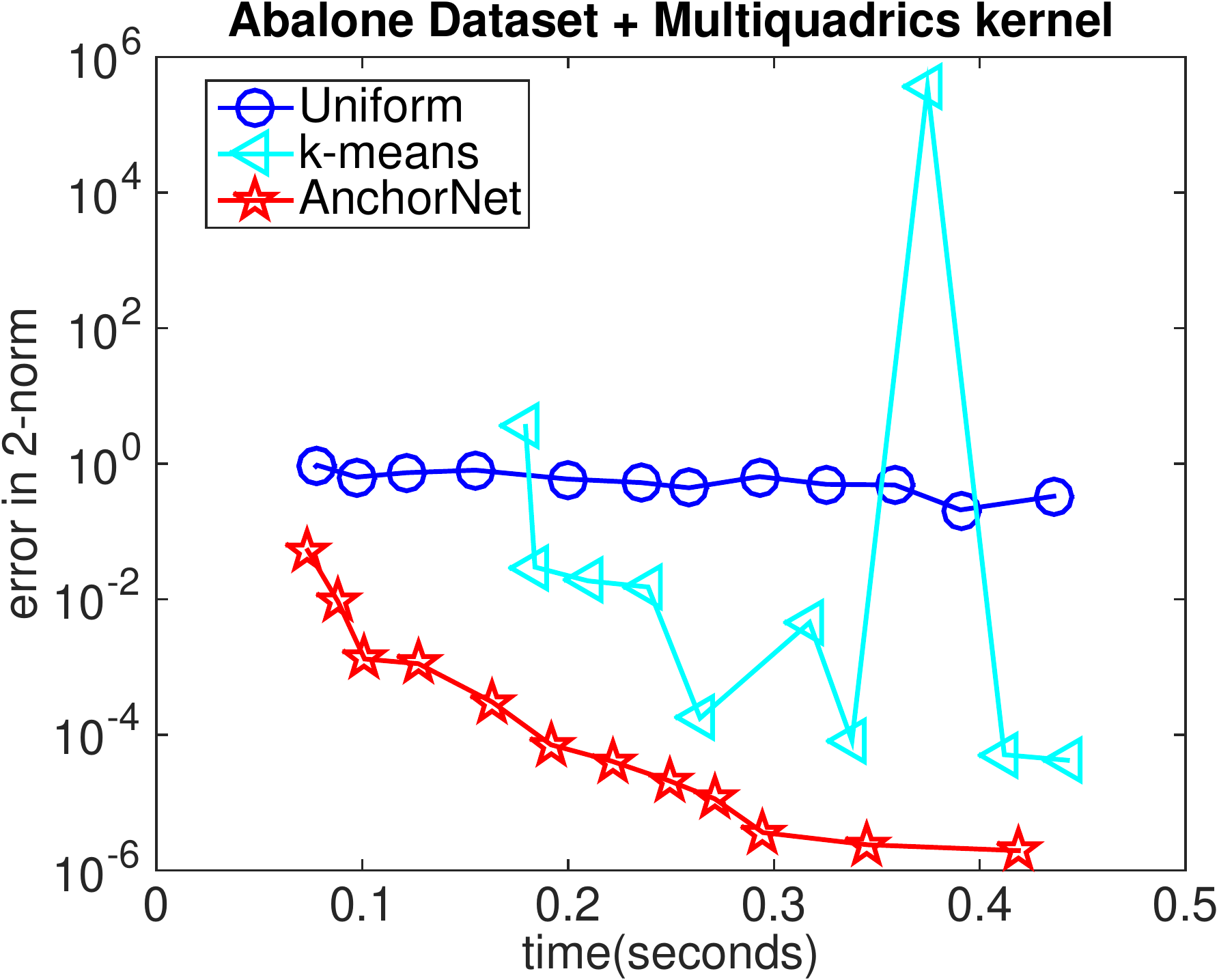}
    \includegraphics[scale=.22]{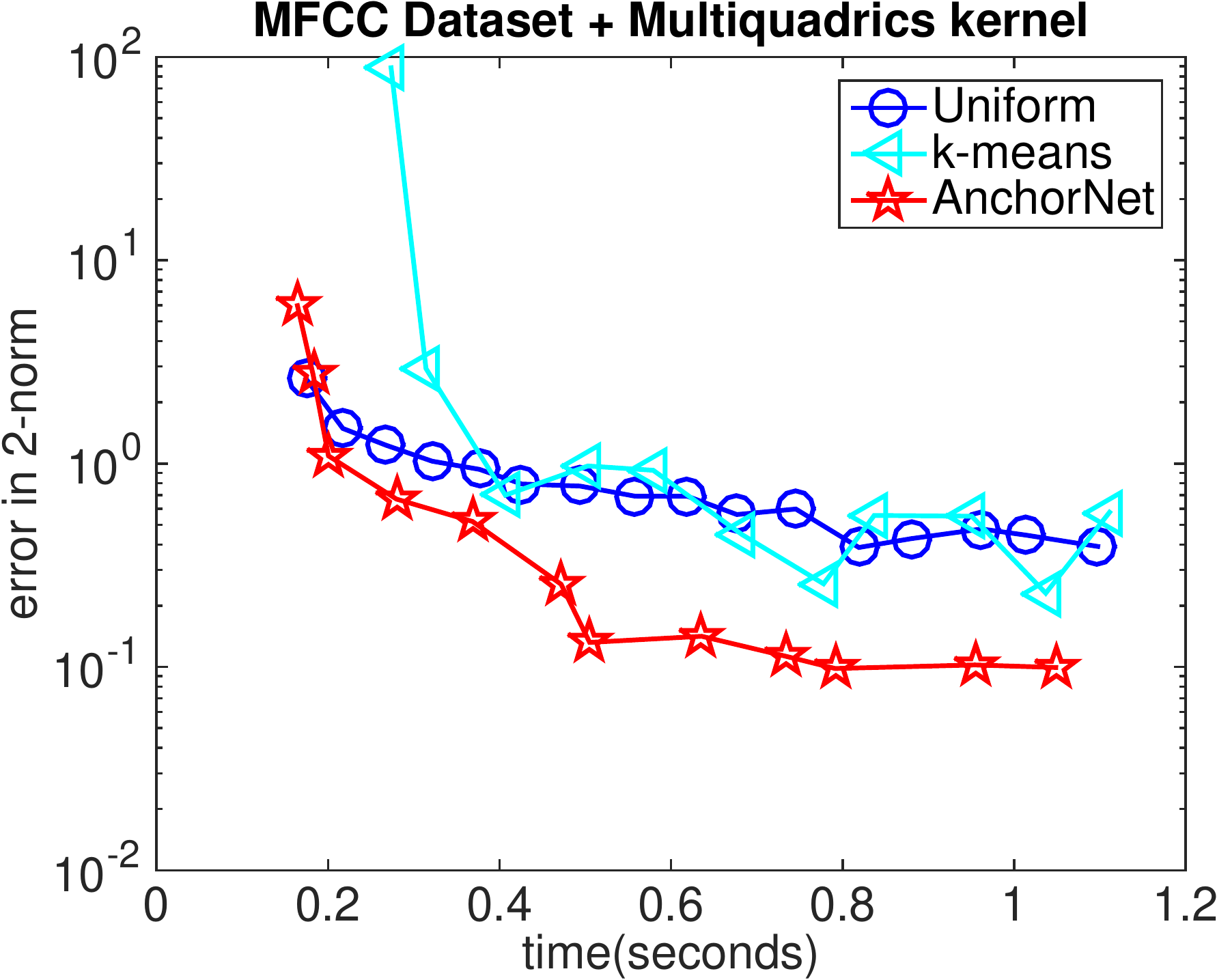}
    \includegraphics[scale=.22]{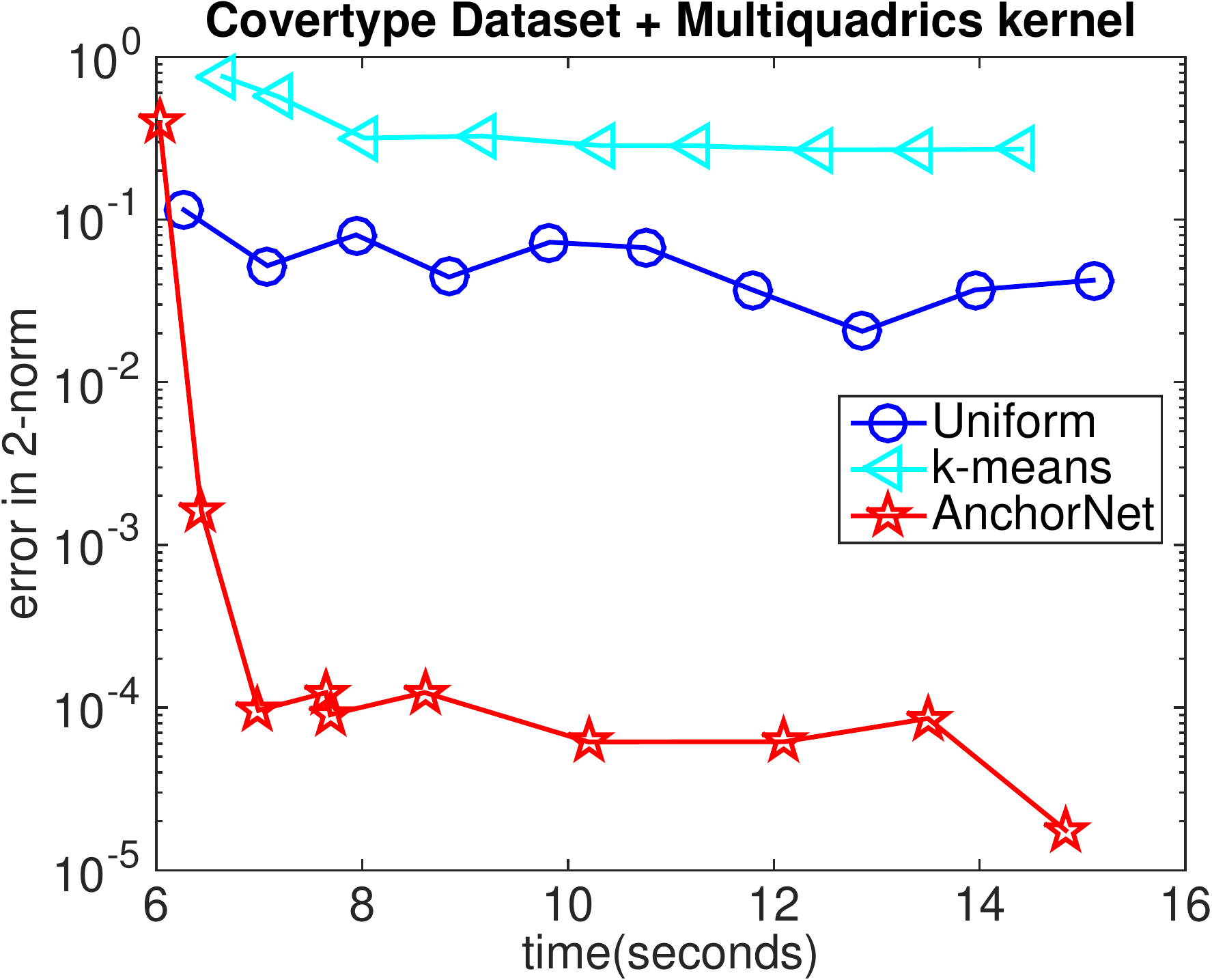}
    \caption{Multiquadrics Error-Time plot: Abalone (left), MFCC (middle), Covertype (right)}
    \label{fig:MQTime}
\end{figure}

\begin{figure}[htbp] 
    \centering 
    \includegraphics[scale=.22]{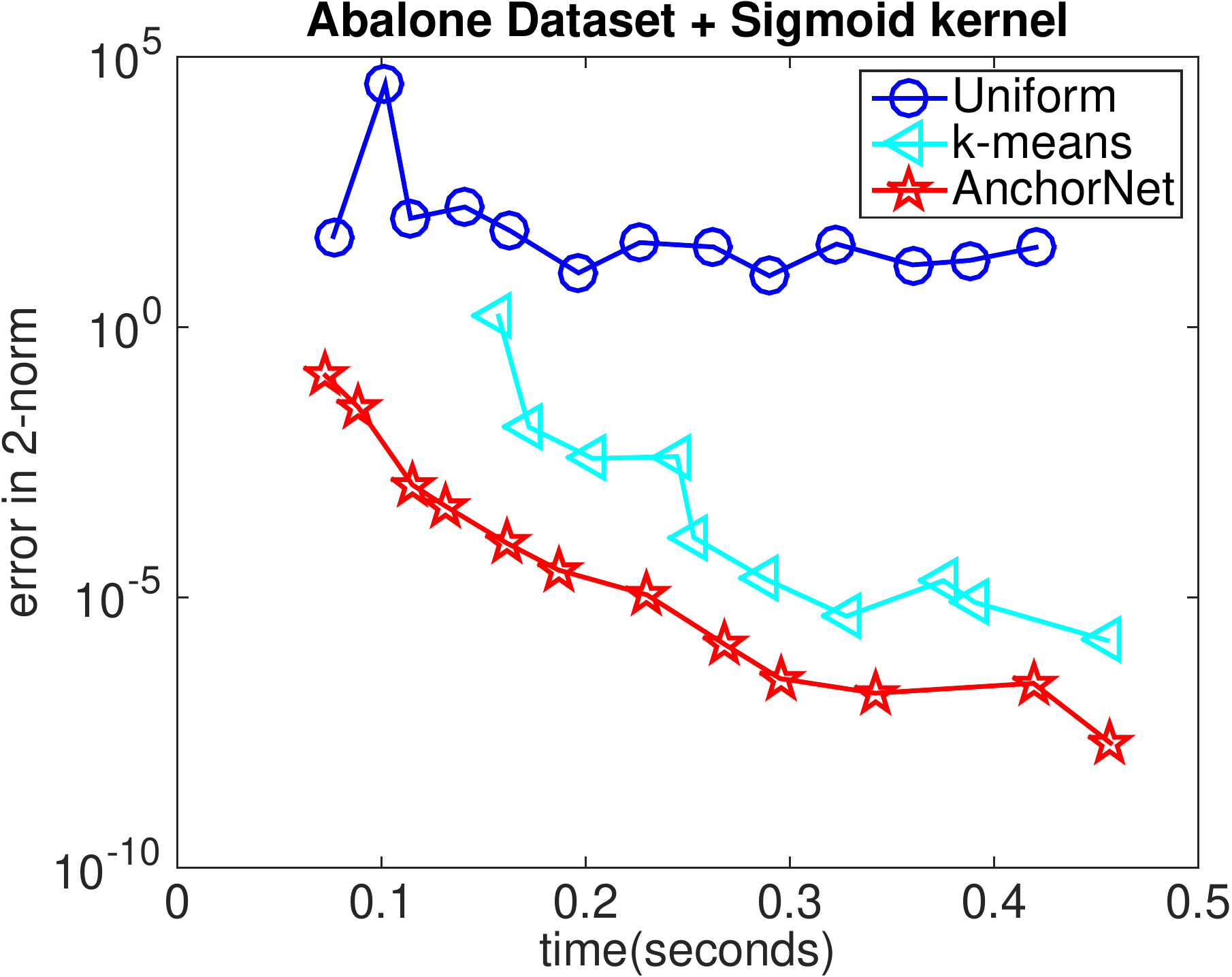}
    \includegraphics[scale=.22]{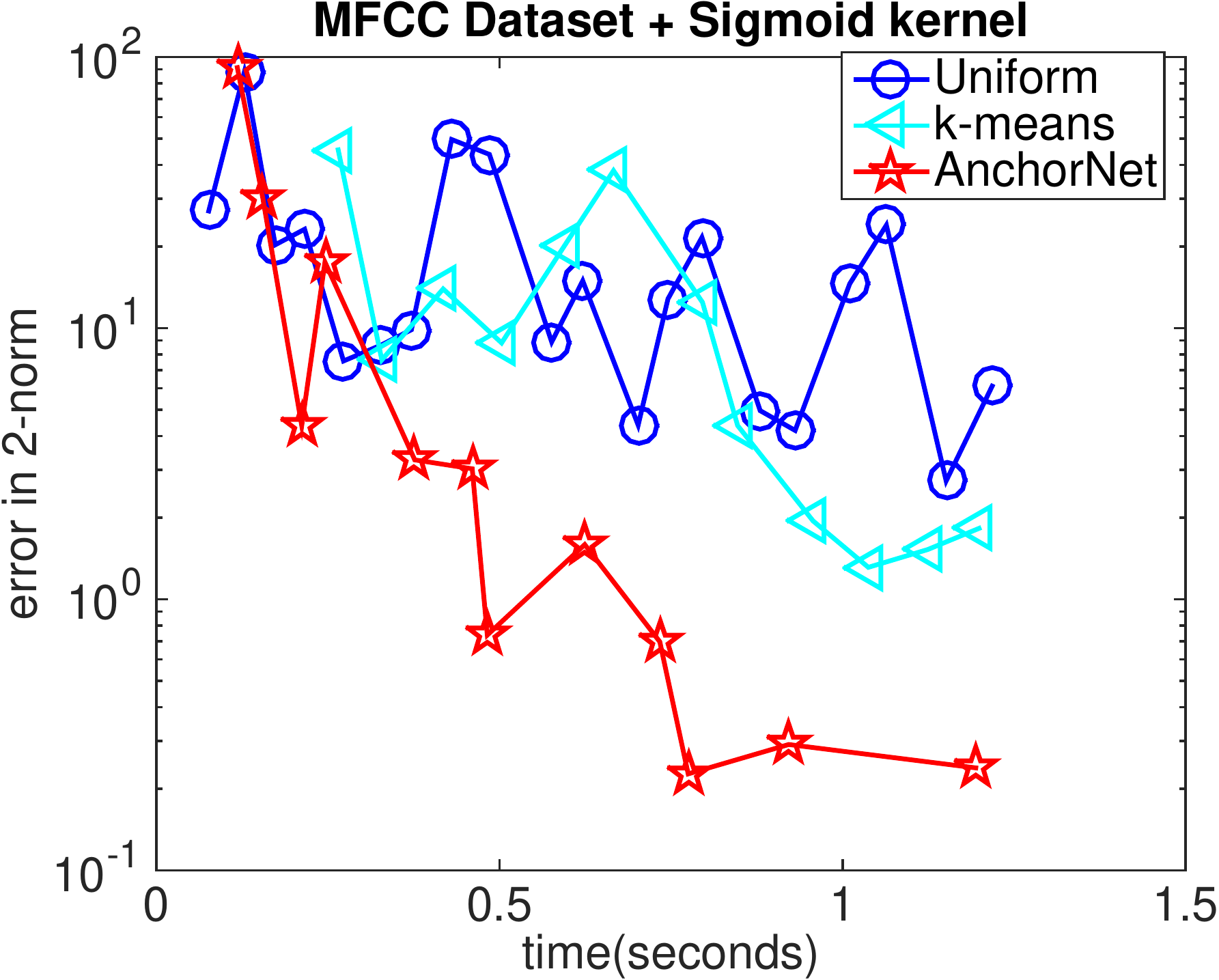}
    \includegraphics[scale=.22]{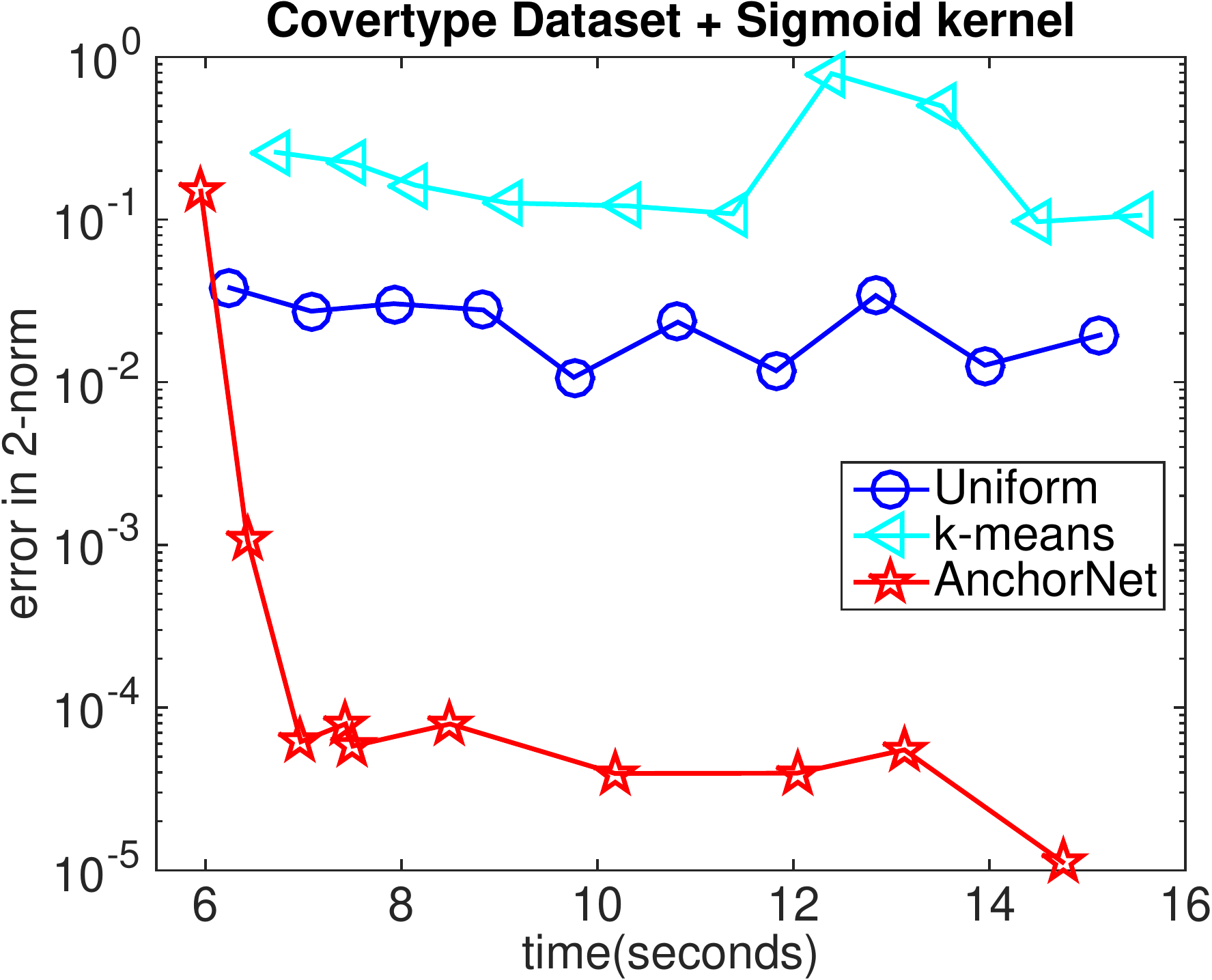}
    \caption{Sigmoid Error-Time plot: Abalone (left), MFCC (middle), Covertype (right)}
    \label{fig:SigmoidTime}
\end{figure}

\begin{figure}[htbp] 
    \centering 
    \includegraphics[scale=.22]{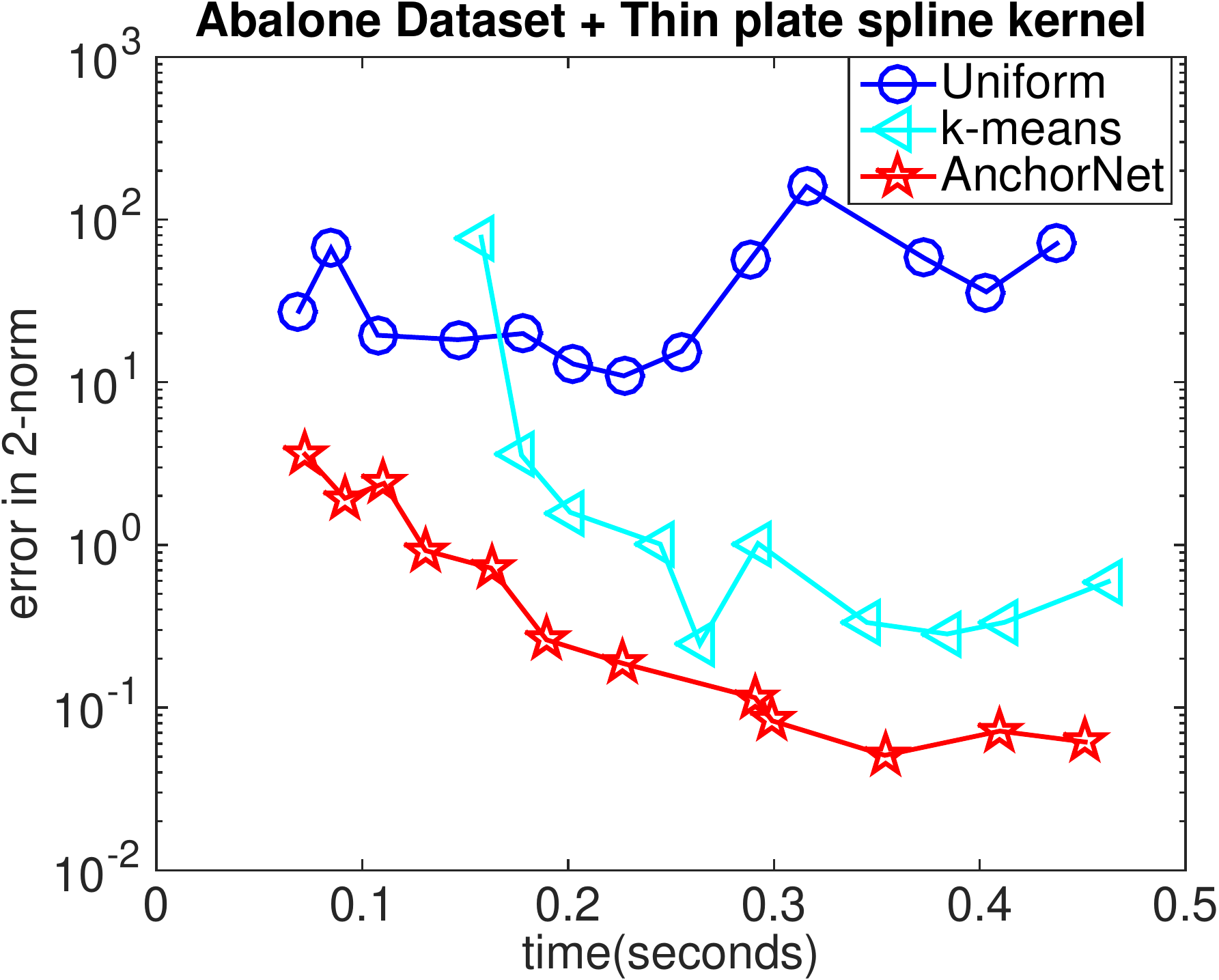}
    \includegraphics[scale=.22]{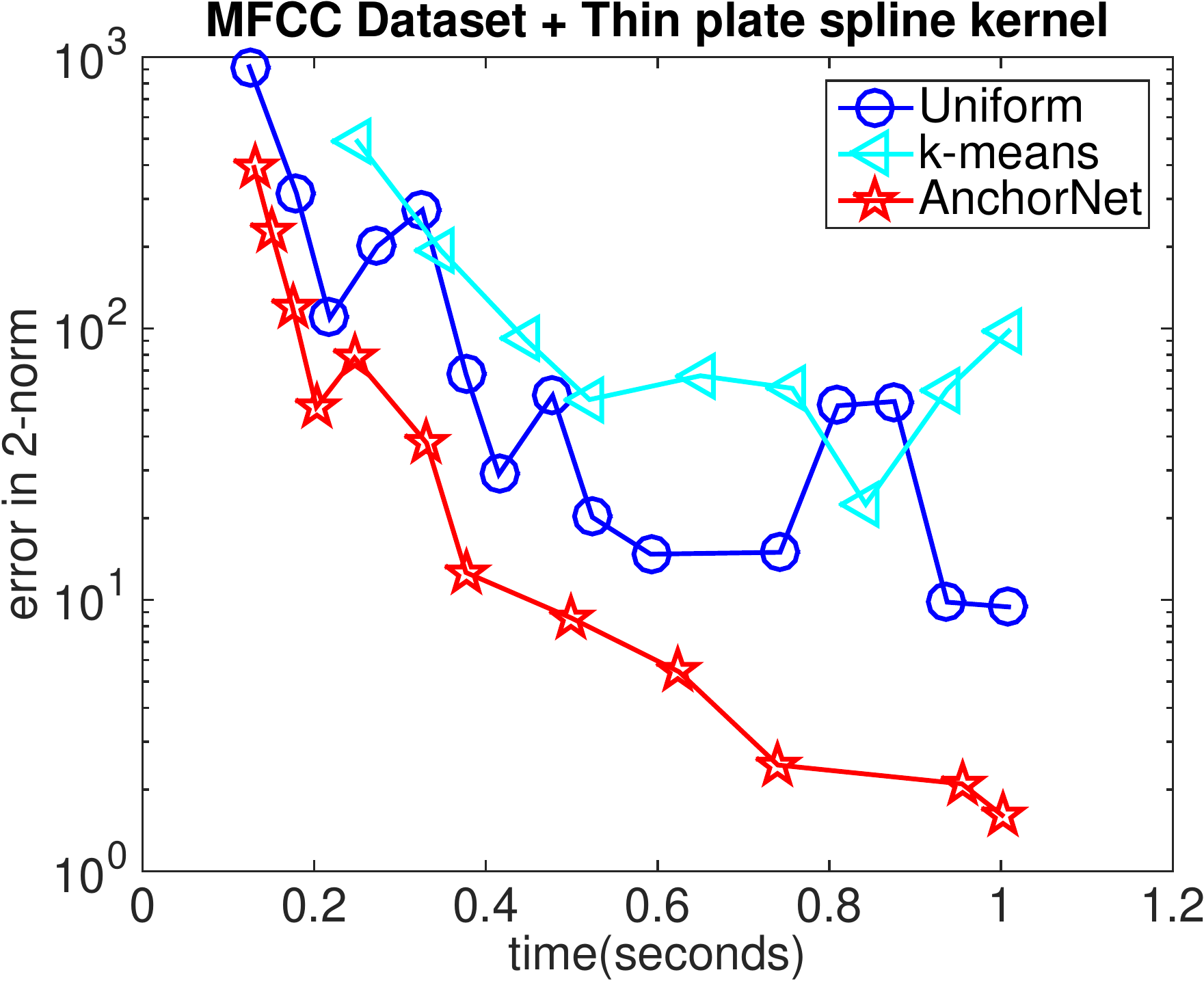}
    \includegraphics[scale=.22]{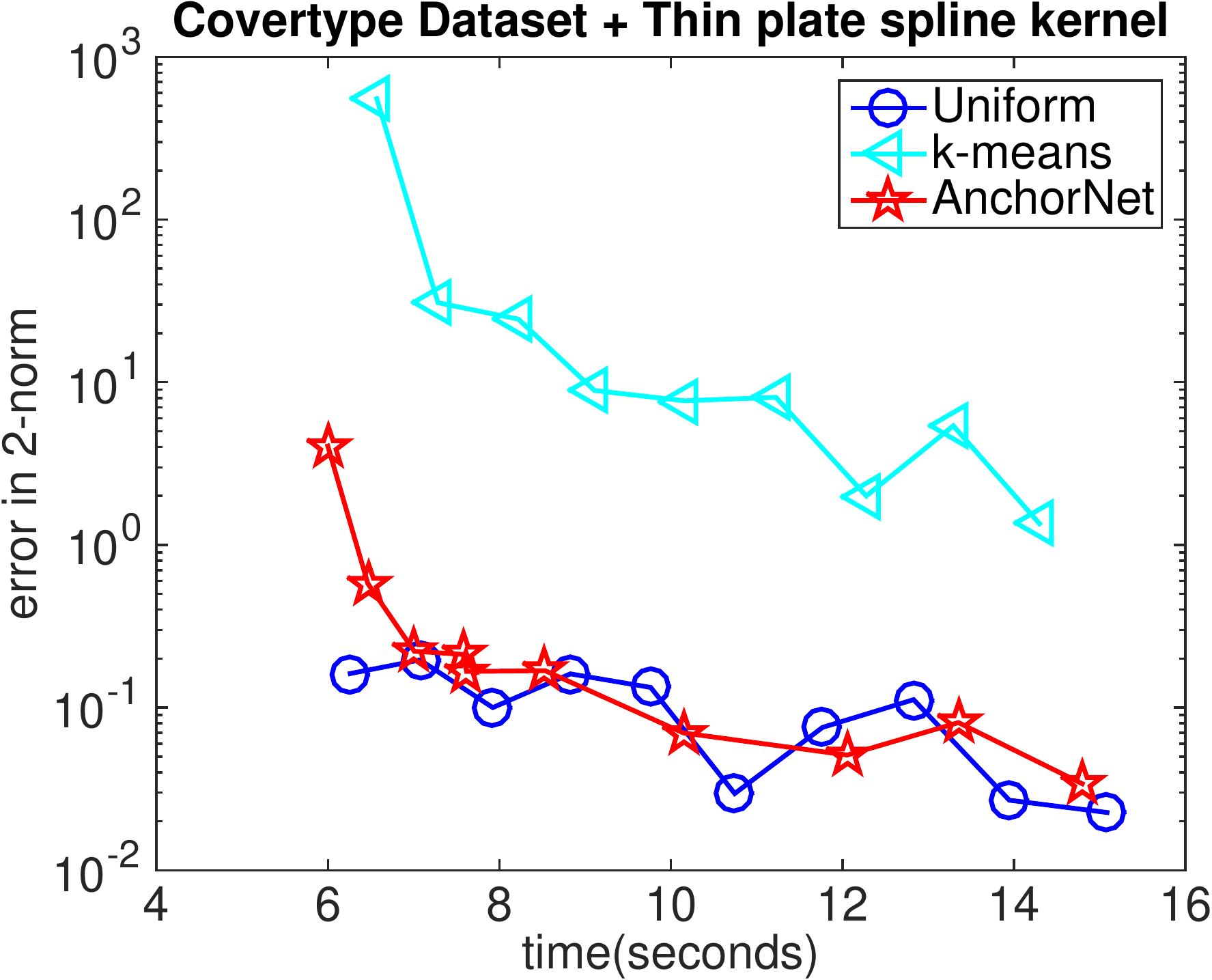}
    \caption{Thin plate spline Error-Time plot: Abalone (left), MFCC (middle), Covertype (right)}
    \label{fig:SplineTime} 
\end{figure}

\subsection{Geometry of landmark points and numerical issues for indefinite kernels}
\label{sub:DK}
In this subsection, we investigate two issues: 
(1) how the geometry of landmark points impacts the accuracy as well as numerical stability of the resulting \nys approximation;
(2) how the stabilization techniques \eqref{eq:nyspinv}-\eqref{eq:nysQR} influence the accuracy of \nys approximation.

\textbf{Geometry of landmark points.}
To illustrate the effect of geometry of landmark points on the \nys approximation, we consider the sigmoid kernel with $\sigma=1$ over a two-dimensional  highly non-uniform dataset illustrated in Figure \ref{fig:dk500}-left. The singular values of the corresponding kernel matrix decay rapidly, and as a result, \nys approximation is subject to numerical instability if landmark points are not well-chosen.

In terms of the selection of landmark points, it can be clearly seen from Figure \ref{fig:dk500} that both uniform sampling and $k$-means clustering tend to generate more landmark points in denser regions of the dataset, for example, around $(0.4,0.3)$, $(0.5,0.7)$, etc. This does \emph{not} contribute to a better approximation and, conversely, may lead to numerical instability and possibly a much worse approximation than the one with fewer landmark points. 

As reflected in the error plot in Figure \ref{fig:dk500}-right, over ten repeated runs, uniform sampling often becomes ineffective due to the poor choice of landmark points $S$, which causes the approximation error to blow up when computing $K_{SS}^{+}$. 
The $k$-means \nys method, on the other hand, can sometimes achieve high accuracy when $k$ is small, but becomes quite unstable as $k$ increases. Figure \ref{fig:dk500}-right shows that the $k$-means \nys method breaks down when $k$ increases from around 220 to 440.  {As the number of clusters increases, computing the centroids of the clusters puts more weight on small dense clusters that contain a large number of points close to each other. This will result in more landmark points (centroids) close to those dense clusters, eventually  causing numerical instability when computing the \nys approximation.}
It can be seen that the anchor net method remains robust besides being the most accurate as the number of landmark points increases.
Overall, for indefinite kernel matrices and highly non-uniform data, existing \nys methods tend to generate landmark points that result in an extremely unstable and inaccurate approximation, while the anchor net method is able to yield accurate and robust approximation by choosing geometrically well-balanced landmark points with no clumps. 
%The experiment results justify the theoretical findings in Section \ref{sec:theory}.

\textbf{Performance of stabilization techniques.}
We then consider the same problem as in Figure \ref{fig:dk500} but use the ``stabilized" \nys approximations based on \eqref{eq:nyspinv} and \eqref{eq:nysQR} to investigate the impact of using the approximate pseudoinverse $K_{SS,\epsilon}^+$ as compared to $K_{SS}^+$. 
We compute each of the two ``stabilized" \nys approximations in \eqref{eq:nyspinv}  and \eqref{eq:nysQR} using three methods: uniform sampling, $k$-means and anchor net.
To study the impact of truncation in \eqref{eq:nyspinv}  and \eqref{eq:nysQR}, we use four different values of truncation tolerance: $\epsilon =$  $10^{-8}, 10^{-10}, 10^{-12}, 10^{-14}$.
For each $\epsilon$, we compare the performance of three \nys schemes.
The resulting four error-rank plots are shown in Figure \ref{fig:dkpinv}.
As expected, we see that the truncation techniques \emph{do} stabilize the \nys approximation for uniform sampling and $k$-means as compared to Figure \ref{fig:dk500}.
However, we also see that the stabilized \nys approximation  in \eqref{eq:nyspinv}  significantly worsens the \emph{accuracy} of the \nys approximation.
In Figure \ref{fig:dk500}, we see that despite stability, all three methods are able to achieve high accuracy, for example, around 9 to 11 digits when the rank is 200.
According to Figure \ref{fig:dkpinv} (top), with the stabilized approximation, all three methods can at most achieve around 5 digits of accuracy.
Meanwhile, different values of $\epsilon$ yields quite different approximation accuracy and in practice it is hard to determine which one should be used.

The results in Figure \ref{fig:dkpinv} also show that stabilization techniques may harm the accuracy when the original \nys approximation is accurate enough.
This is easily seen in Figure \ref{fig:dkpinv} by comparing stabilized anchor net-based approximation (red solid line) to the original version (red dotted line), where both stabilization techniques lead to orders of magnitude loss of  accuracy.
This can be seen from the right-most plots in Figure \ref{fig:dkpinv}.
We also see that the stabilized approximation may not achieve as good accuracy as the original \nys method.
% To further illustrate the issue of the truncated pseudoinverse in \eqref{eq:nyspinv},
% we compare the \eqref{eq:nyspinv} to the original \nys formula with exact pseudoinverse $K_{SS}^+$.
% in Figure \ref{fig:dkpinv} (top right), we test the approximation error of $k$-means \nys method with stabilized approximation \eqref{eq:nyspinv} versus tolerance $\epsilon$.
% As a comparison, we also include the error (dotted line) of the rank-250 approximation with the original \nys formula (without truncating $K_{SS}$).
% We see that the optimal choice of $\epsilon$ depends on the approximation rank and is hard to predict.
% More importantly, the stabilized \nys formula in \eqref{eq:nyspinv} leads to significantly \emph{less} accurate approximation, as can be seen from the two rank-250 approximations in Figure \ref{fig:dkpinv} where the stabilized version is about 3 digits less accurate than the original \nys approximation.
% Similar problems can be found in Figure \ref{fig:dkpinv} (bottom right) for the QR-based stabilization \eqref{eq:nysQR}.

By looking at the fourth plot $\epsilon=10^{-14}$ on the bottom row in Figure \ref{fig:dkpinv}, we see that the QR-based stabilization in \eqref{eq:nysQR} is accurate when the rank is small but then leads to numerical instability as rank increases (see red solid line).
Neither of the two stabilization techniques is able to achieve the same level of accuracy that the anchor net method attains without stabilization.
Overall, the results show that numerical techniques to resolve stability issues may lead to worse approximation  and the error from the $\epsilon$-truncation may dominate the \nys approximation error, especially in the high accuracy regime.
Thus we see that stabilization techniques are not able to fully resolve the numerical issues associated with \nys method and a more appropriate solution should come from a good choice of landmark points, as demonstrated by the anchor net method in Figure \ref{fig:dkpinv}.

\begin{figure}[htbp] 
    \centering 
    \includegraphics[scale=.15]{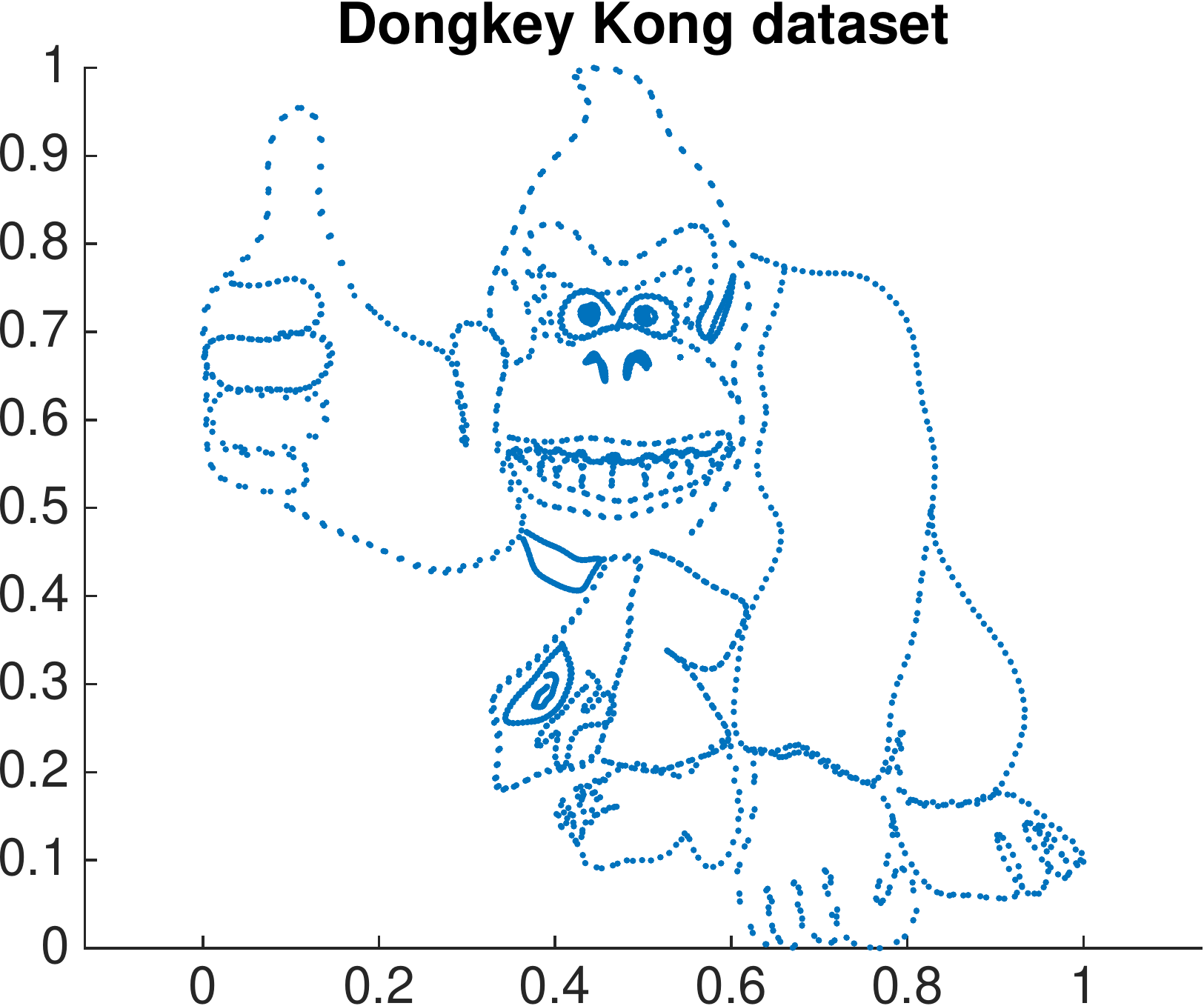}
    \includegraphics[scale=.15]{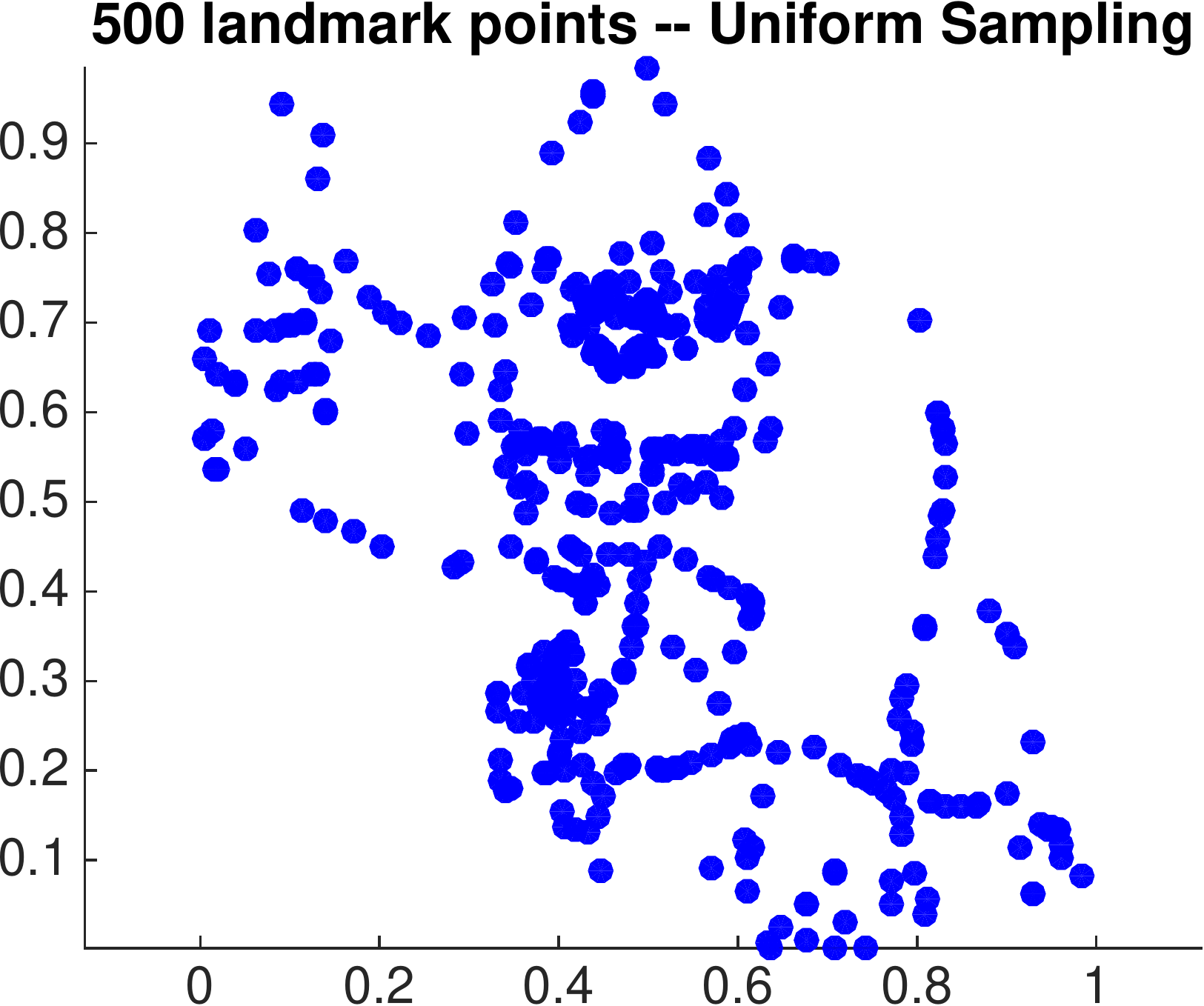}
    \includegraphics[scale=.15]{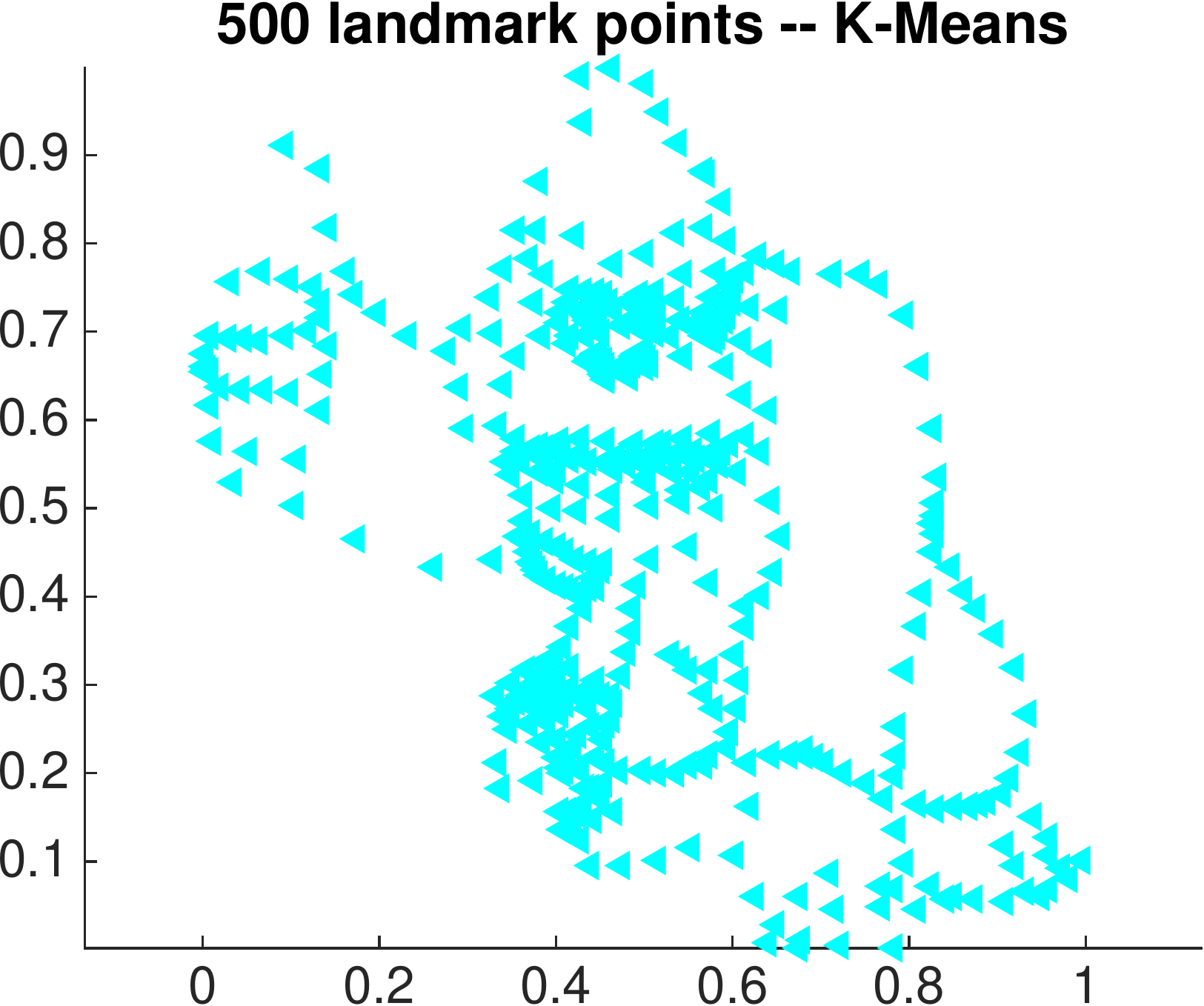}
    \includegraphics[scale=.15]{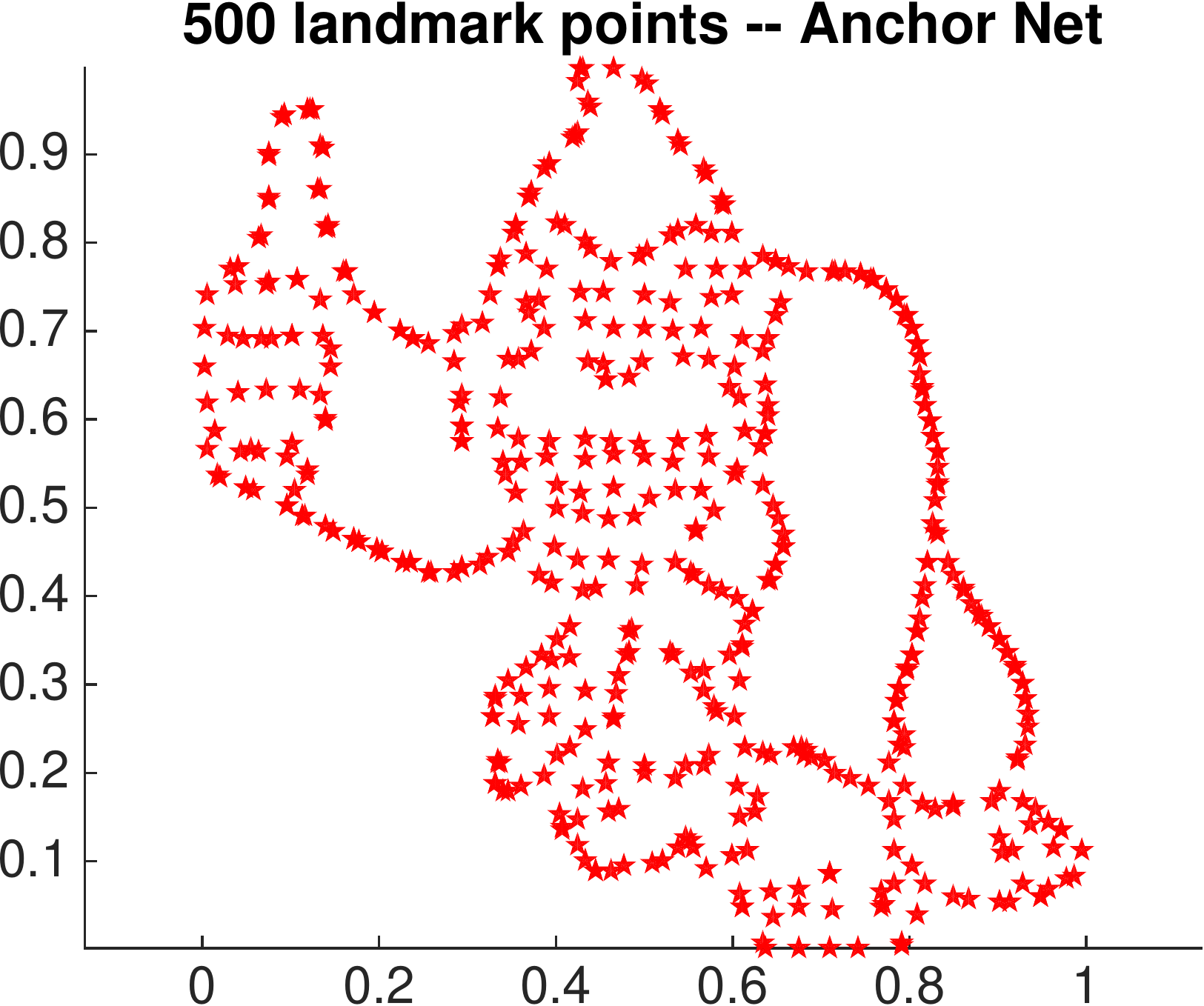}
    \includegraphics[scale=.15]{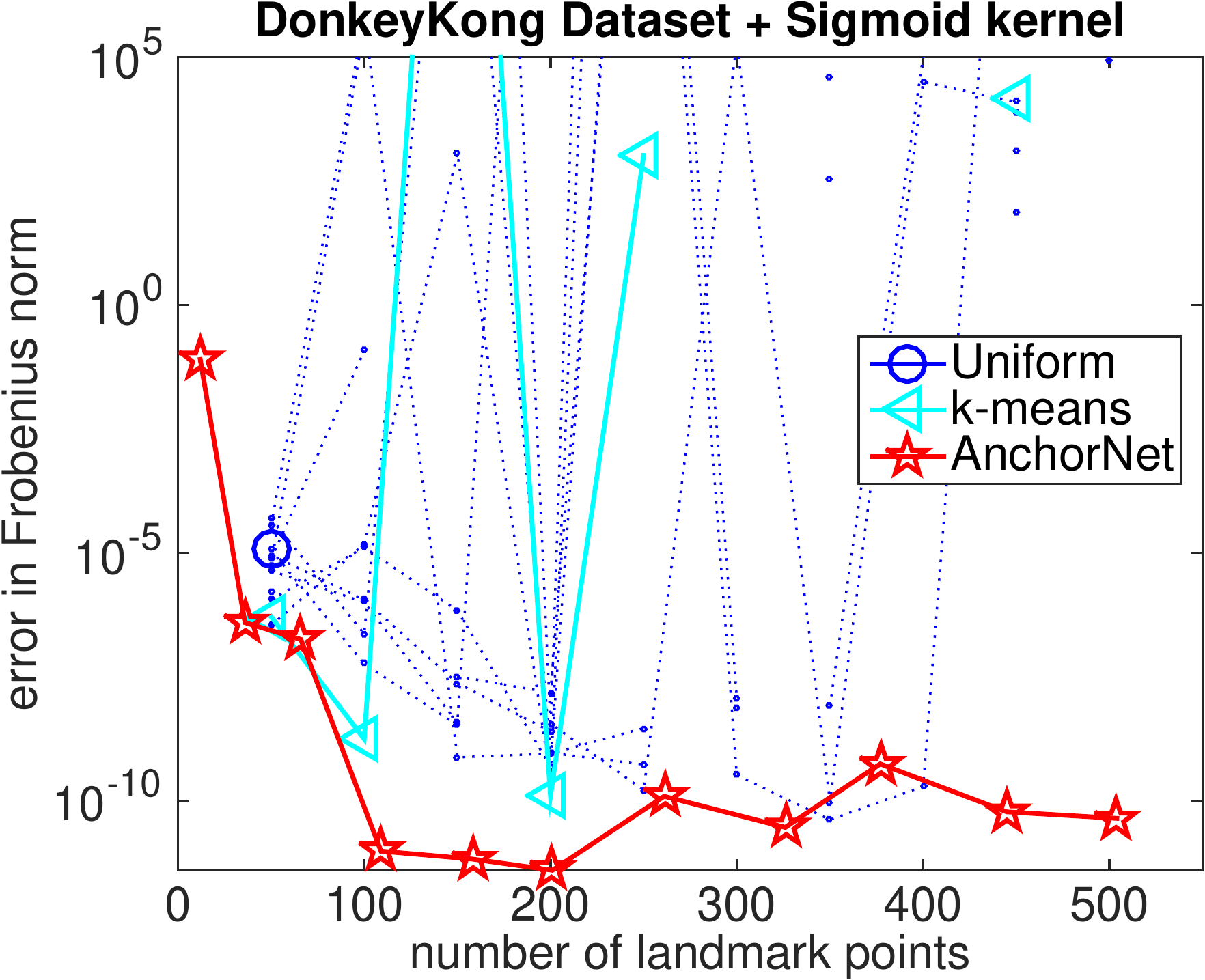}
    \caption{Left to right: Donkey Kong dataset, 500 landmark points generated by three methods, error-rank plot for approximating the sigmoid kernel matrix.}
    \label{fig:dk500} 
\end{figure}

\begin{figure}[htbp] 
    \centering 
    \includegraphics[scale=.13]{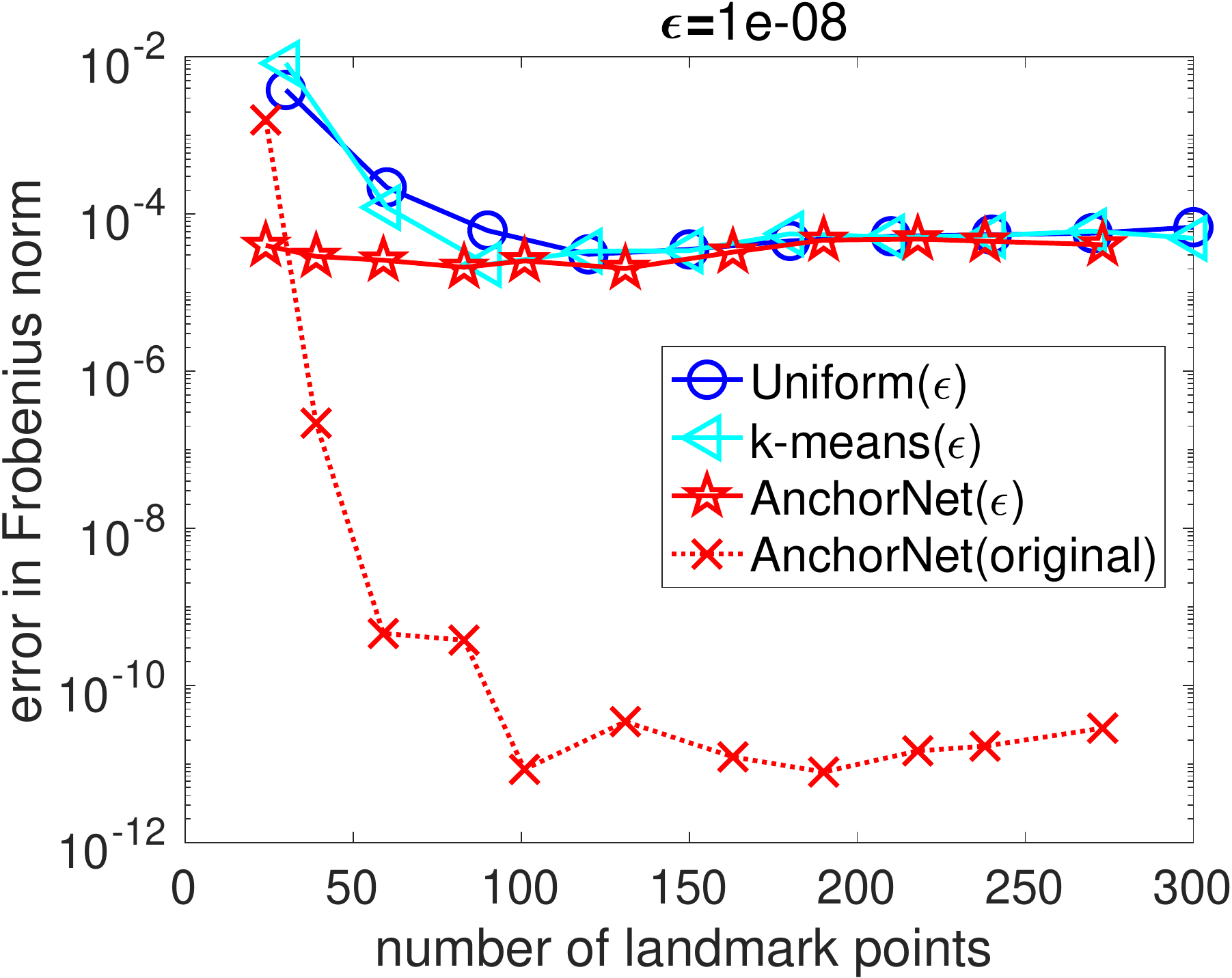}
    \includegraphics[scale=.13]{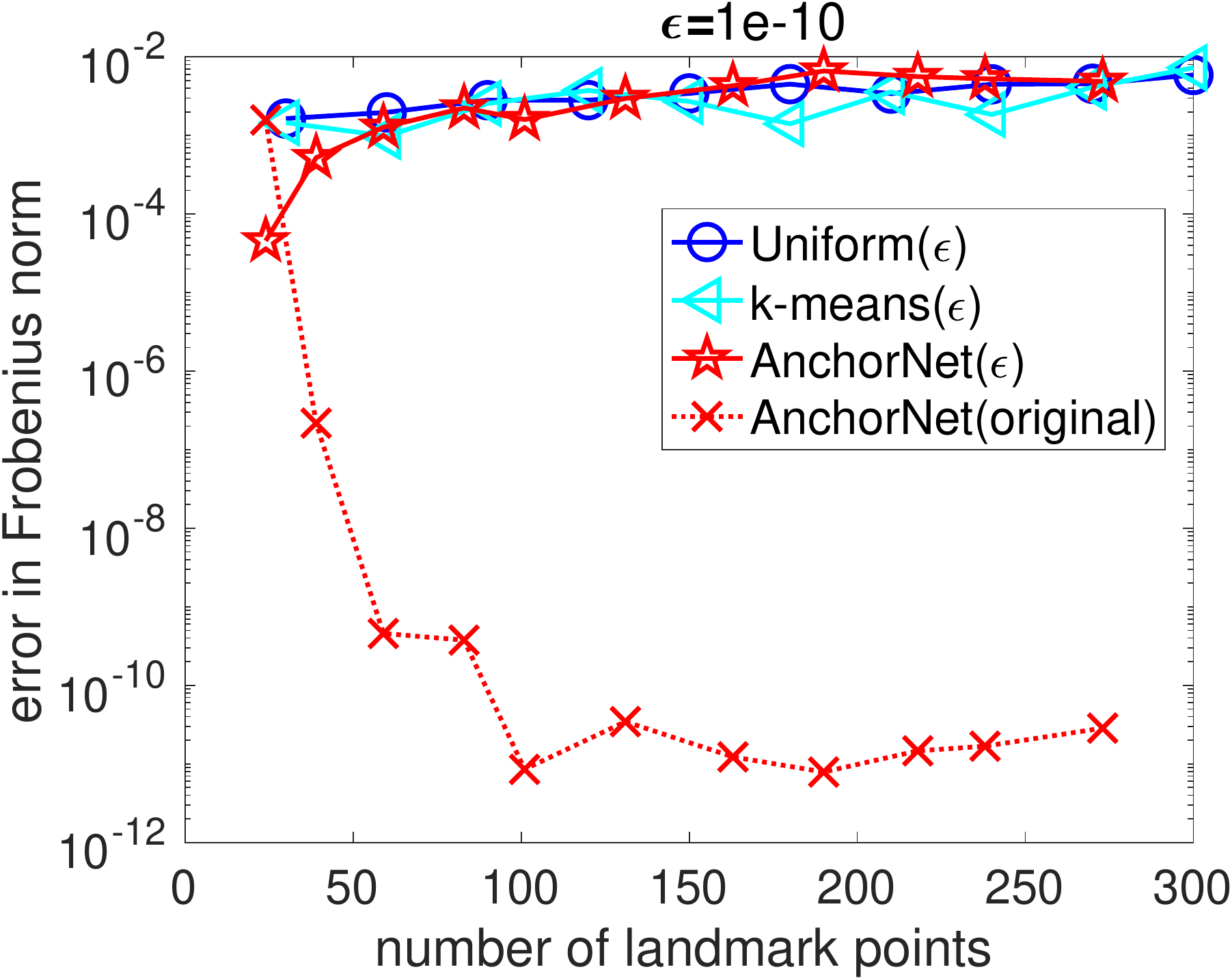}
    \includegraphics[scale=.13]{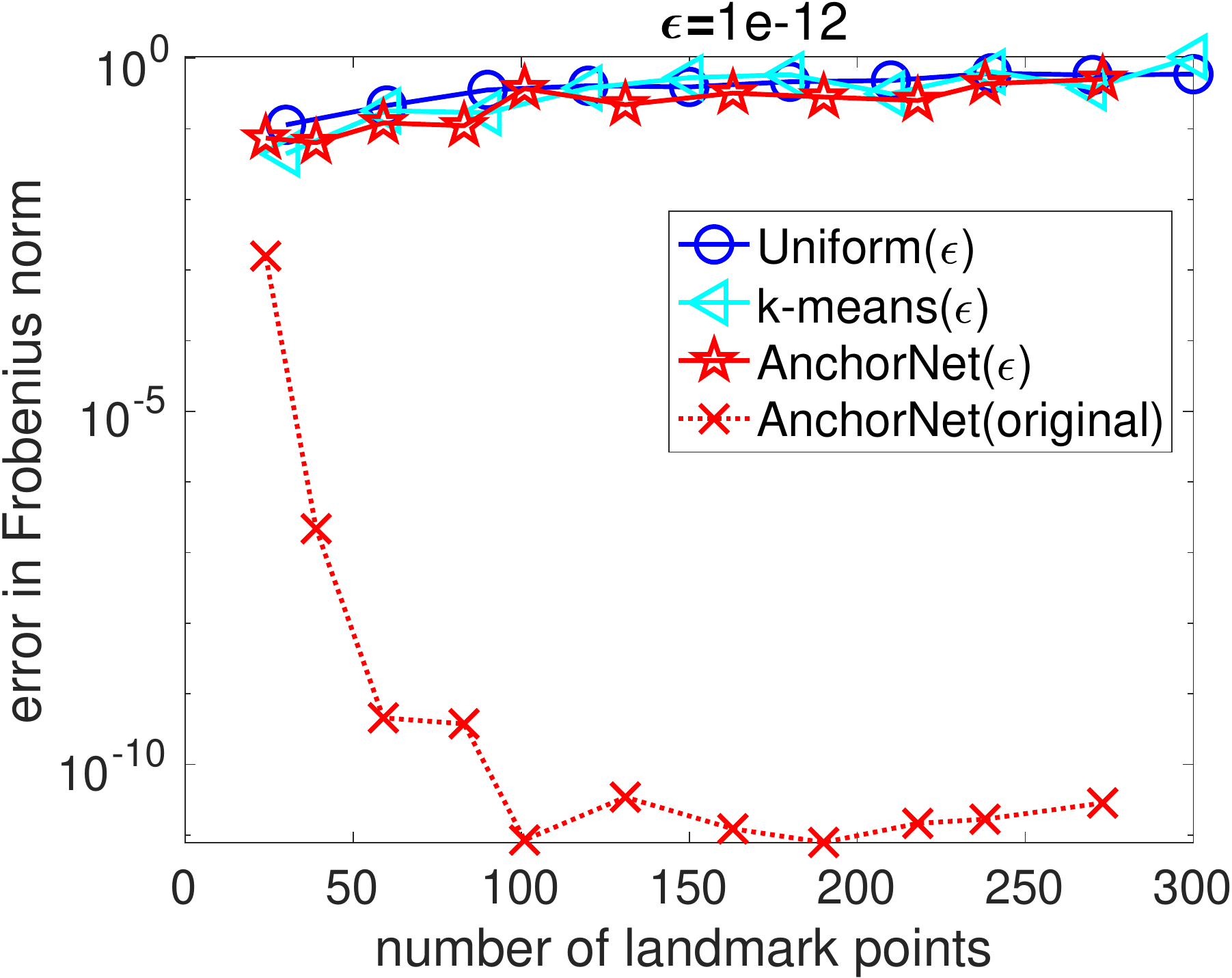}
    \includegraphics[scale=.13]{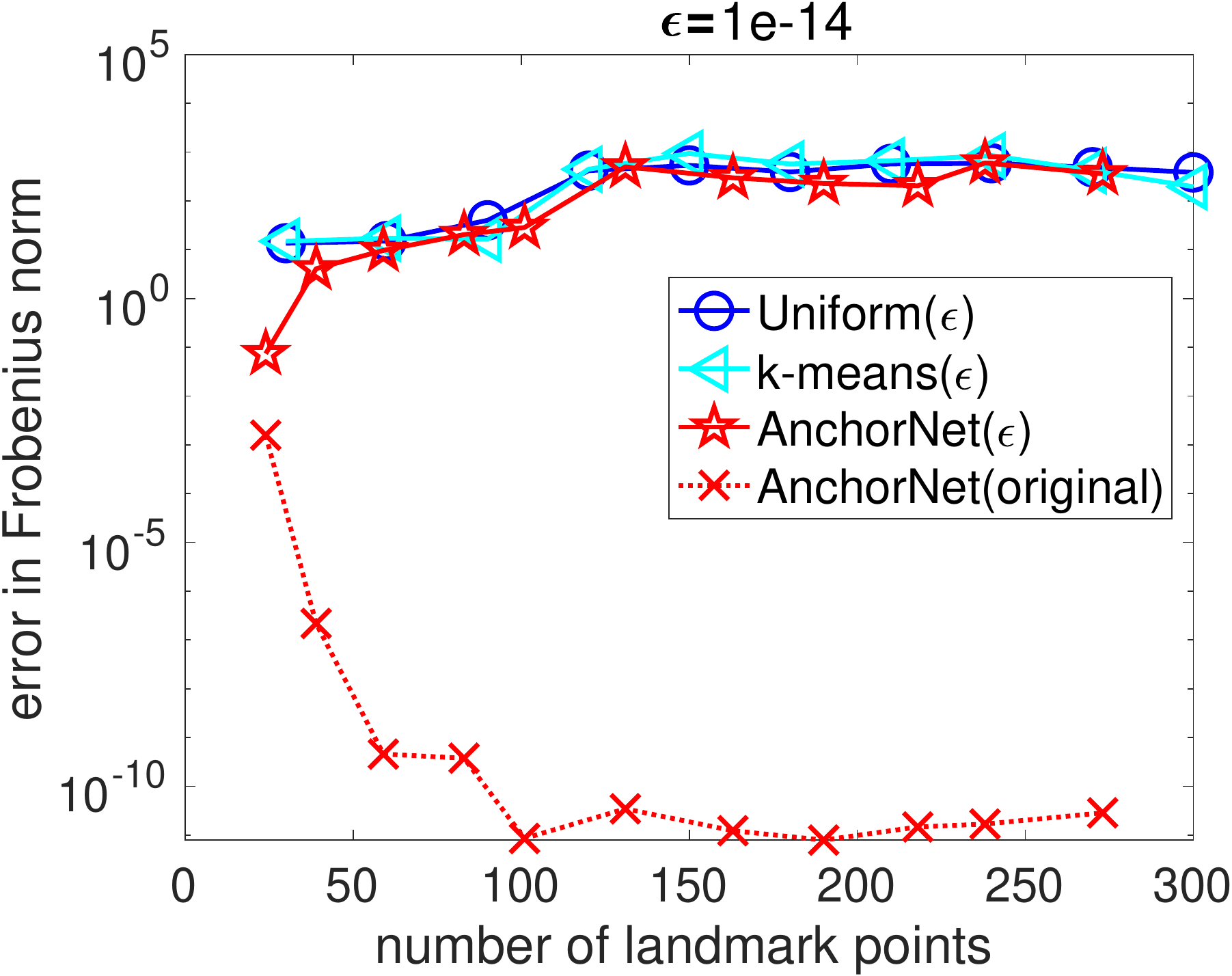}
    \includegraphics[scale=.13]{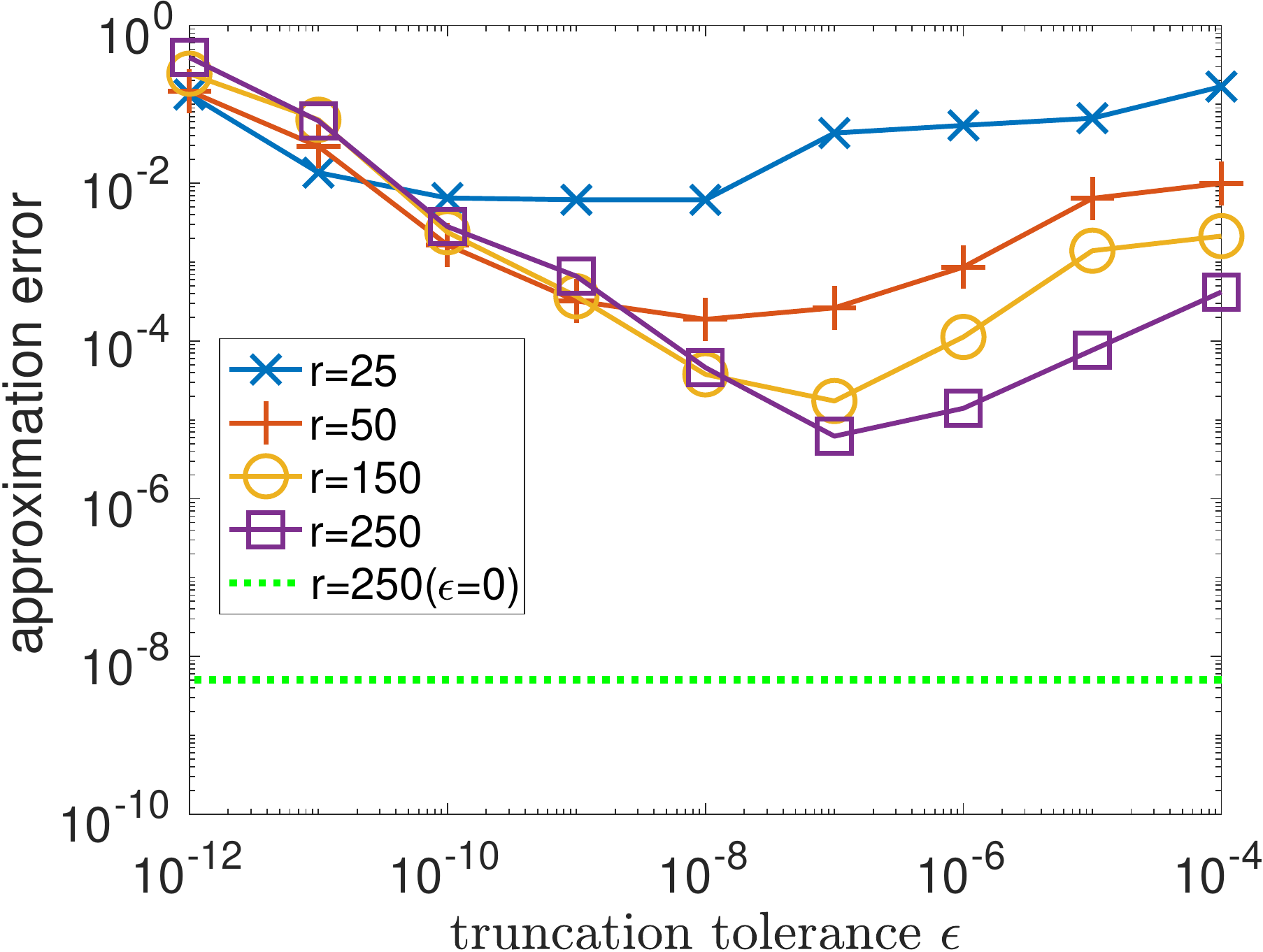}
      \includegraphics[scale=.13]{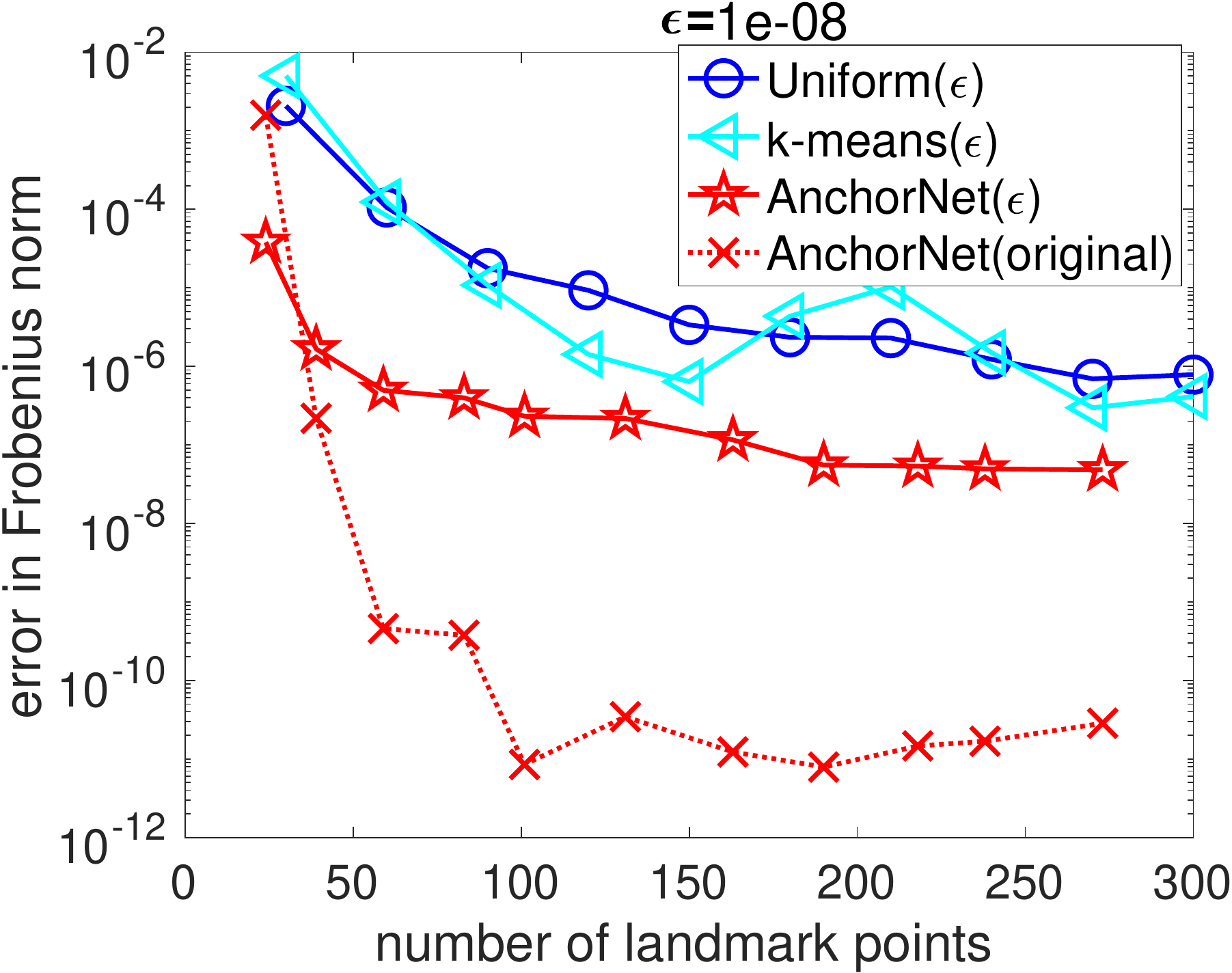}
    \includegraphics[scale=.13]{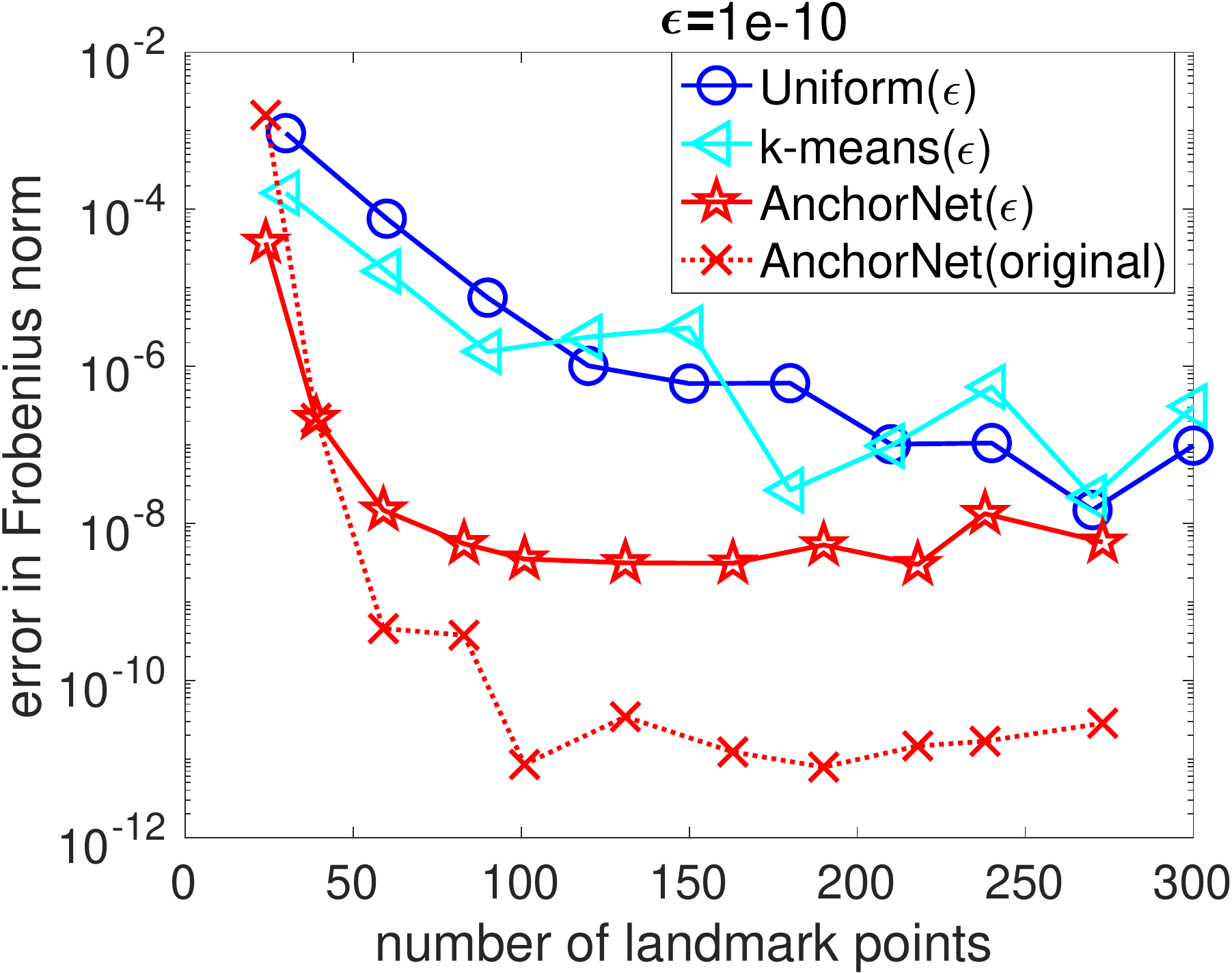}
    \includegraphics[scale=.13]{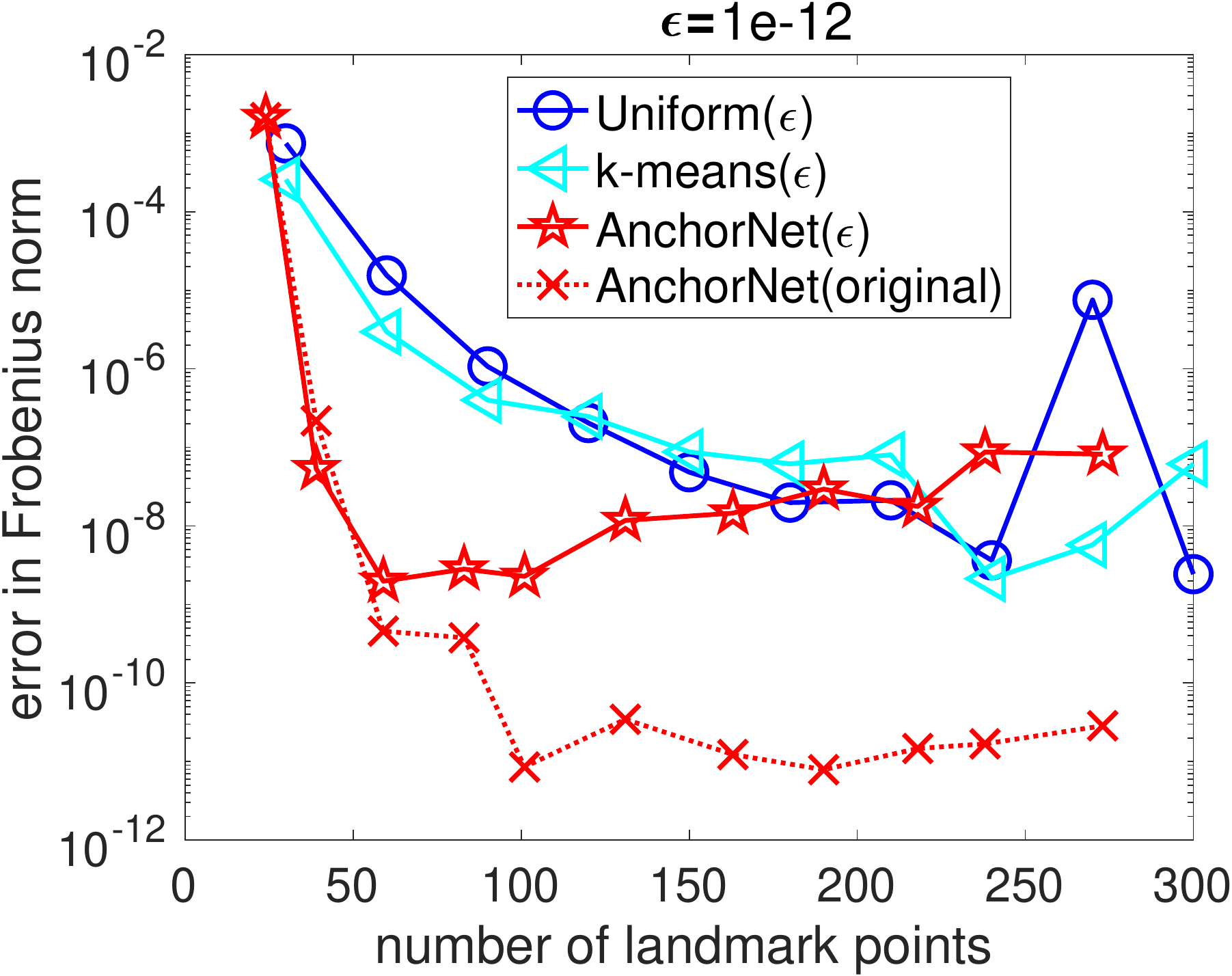}
    \includegraphics[scale=.13]{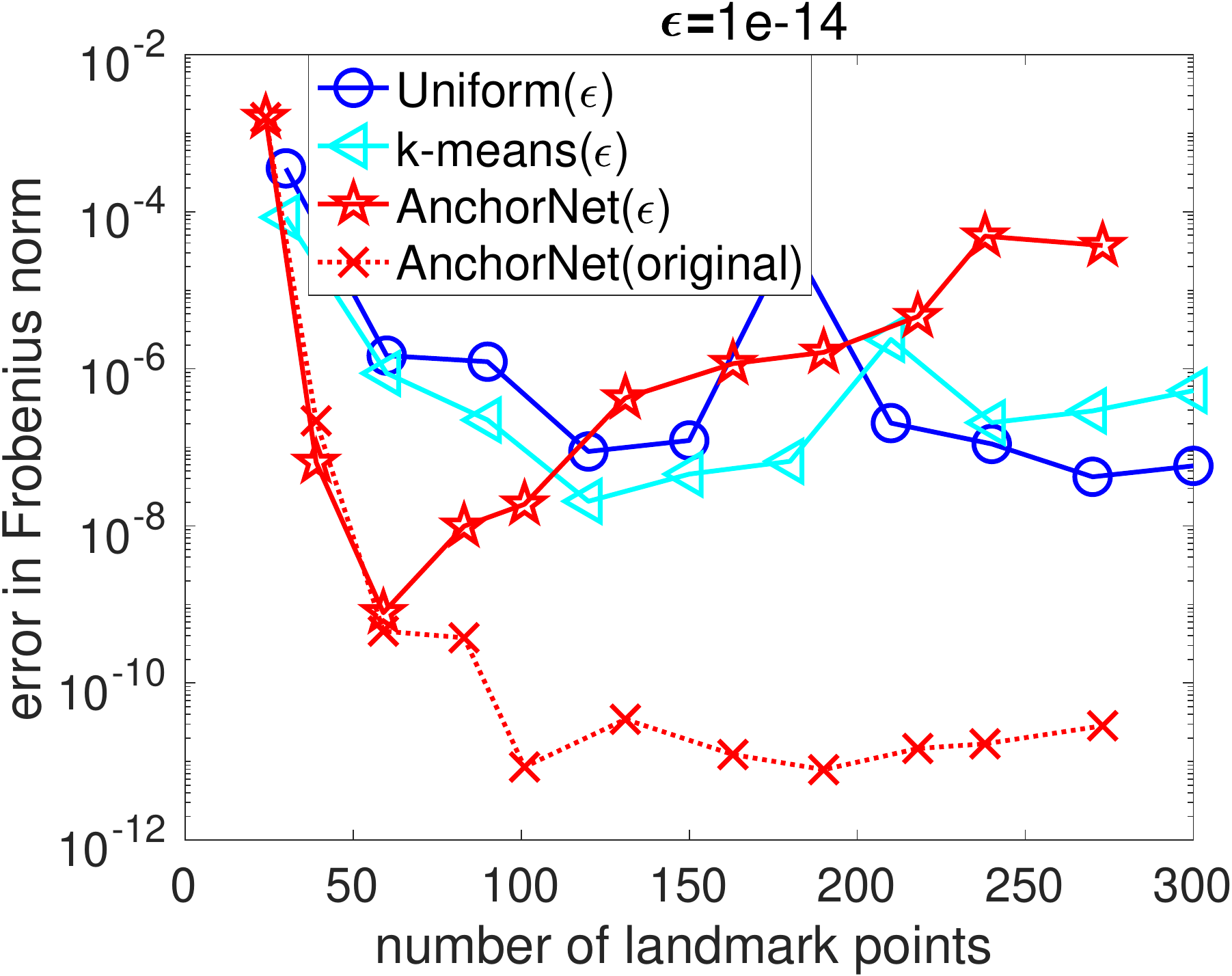}
    \includegraphics[scale=.13]{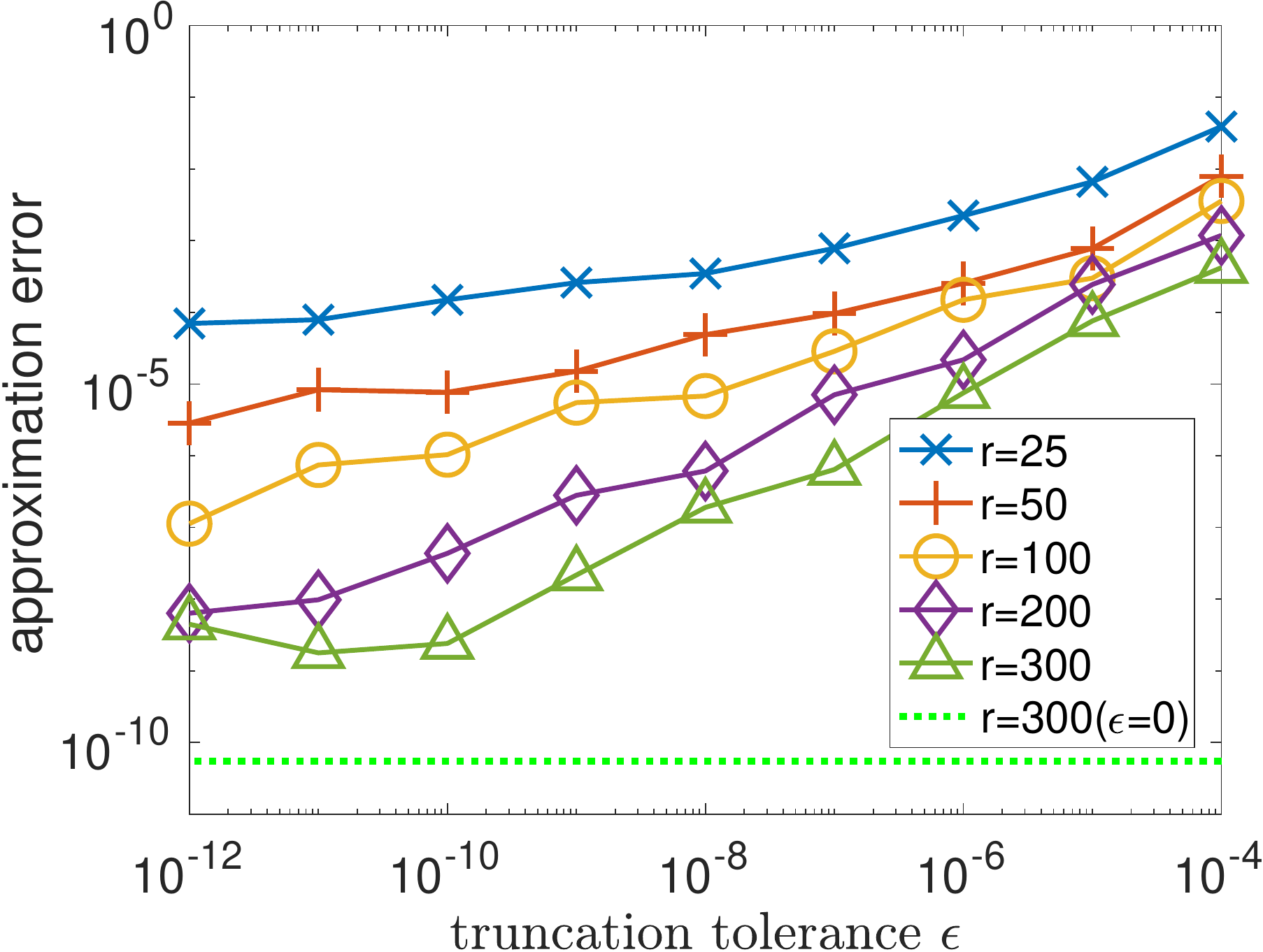}
    \caption{Approximation errors using stabilization techniques: $\epsilon$-pseudoinverse (top row) in \eqref{eq:nyspinv} and $\epsilon$-QR (bottom row) in \eqref{eq:nysQR}. 
    First four figures in each row are error-rank plots of ``stabilized" \nys methods (uniform sampling, $k$-means, anchor net) with $\epsilon=10^{-8},10^{-10},10^{-12},10^{-14}$ and original anchor net \nys (dotted line with '$\times$'), respectively. The last figure shows approximation errors of $k$-means \nys method versus truncation tolerance $\epsilon$, where several ranks are used and the dotted line with '$\times$' symbol denotes a fixed rank approximation error using the original \nys formula \emph{without} stabilization.}
    \label{fig:dkpinv} 
\end{figure}

\subsection{\nys variants for SPSD kernels}
\label{sub:stability}

To illustrate the possible numerical instability of existing \nys methods for SPSD kernel matrices, we consider the approximation of the Gaussian kernel matrix (which is SPSD) with rapidly decaying singular values.  {Since the kernel is SPSD, the numerical instability can be remedied via regularization, i.e. approximating $K+\beta I$ for a small constant $\beta>0$.
We present results for both $K$ and $K+\beta I$ and choose $\beta=10^{-9}$.}
The proposed method (AnchorNet) is compared to the following \nys schemes:
(1) the original uniform \nys method \cite{nys2001}, which was observed in \cite{nys2012} to yield satisfactory overall performance (error-time trade-off) compared to several other methods;
(2) the $k$-means clustering \nys method \cite{nys2008kmeans,nys2010kmeans}, which usually yields better accuracy than the uniform \nys method;
(3) the recursive ridge leverage-score (RLS) \nys method \cite{nys2017}, which improves the efficiency of the original leverage-score based sampling;
(4) the accelerated recursive ridge leverage-score (RLS-x) \nys method \cite{nys2017}, which is much faster than RLS but may not be as robust.
For probabilistic approaches (uniform samplig, RLS, RLS-x), the error is averaged over ten repeated runs.

The methods above are compared from two perspectives: numerical stability and computational efficiency. The Gaussian kernel $\kappa(x,y)=e^{-|x-y|^2/\sigma^2}$ is used and the two experiment settings are listed below.
\begin{enumerate}
    \item \textbf{Numerical stability.} We consider two Gaussian kernels with different choices of the bandwidth parameter $\sigma$: $10\%$ and $50\%$ times the radius of the standardized Abalone dataset. Note that larger $\sigma$ leads to faster singular value decay of the kernel matrix. {Without regularization, the results are presented in Figure \ref{fig:AbaGS}. With regularization, the results are shown in Figure \ref{fig:AbaGSreg}.}
    \item \textbf{Computational efficiency}. We consider two datasets: Abalone ($n=4177, d=8$) and Covertype $(n=581012,d=54)$. For Abalone, we choose $\sigma=2.3$; for Covertype, $\sigma$ is same as the one used in Section \ref{sub:indefinite}. 
    %We do not use $\sigma=11.8$ for Abalone because most methods will be inefficient according to Figure \ref{fig:AbaGS}. 
    {The experiment results are collected as error-time plots in Figure \ref{fig:TimeGS} for $K$ and \ref{fig:TimeGSreg} for $K+10^{-9} I$.}
    The Covertype dataset is quite large and high-dimensional compared to the Abalone dataset, and the results for the two datasets are quite different, as can be seen in Figure \ref{fig:TimeGS}.
\end{enumerate}

\begin{figure}[htbp] 
    \centering
    \includegraphics[scale=.22]{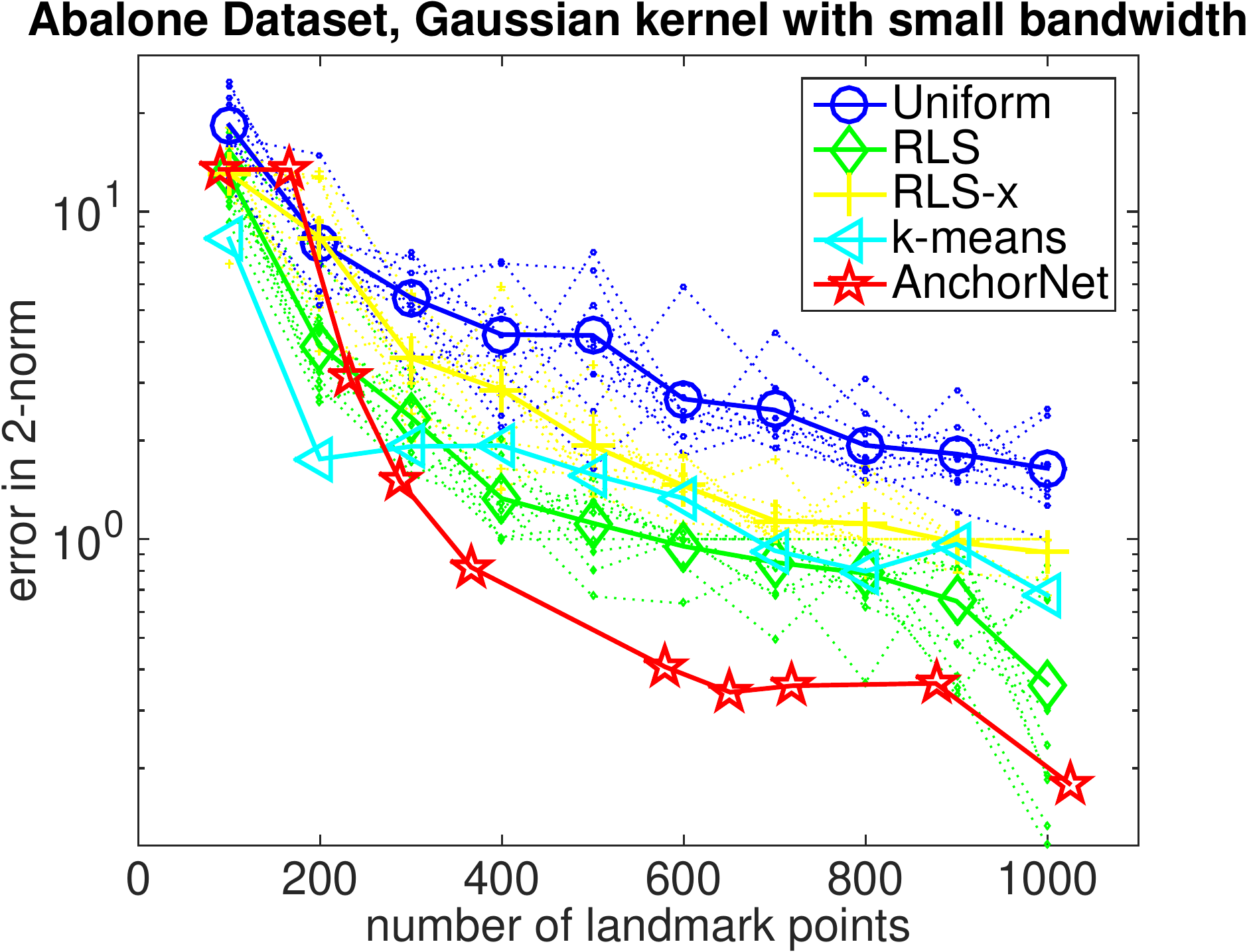}
    \includegraphics[scale=.22]{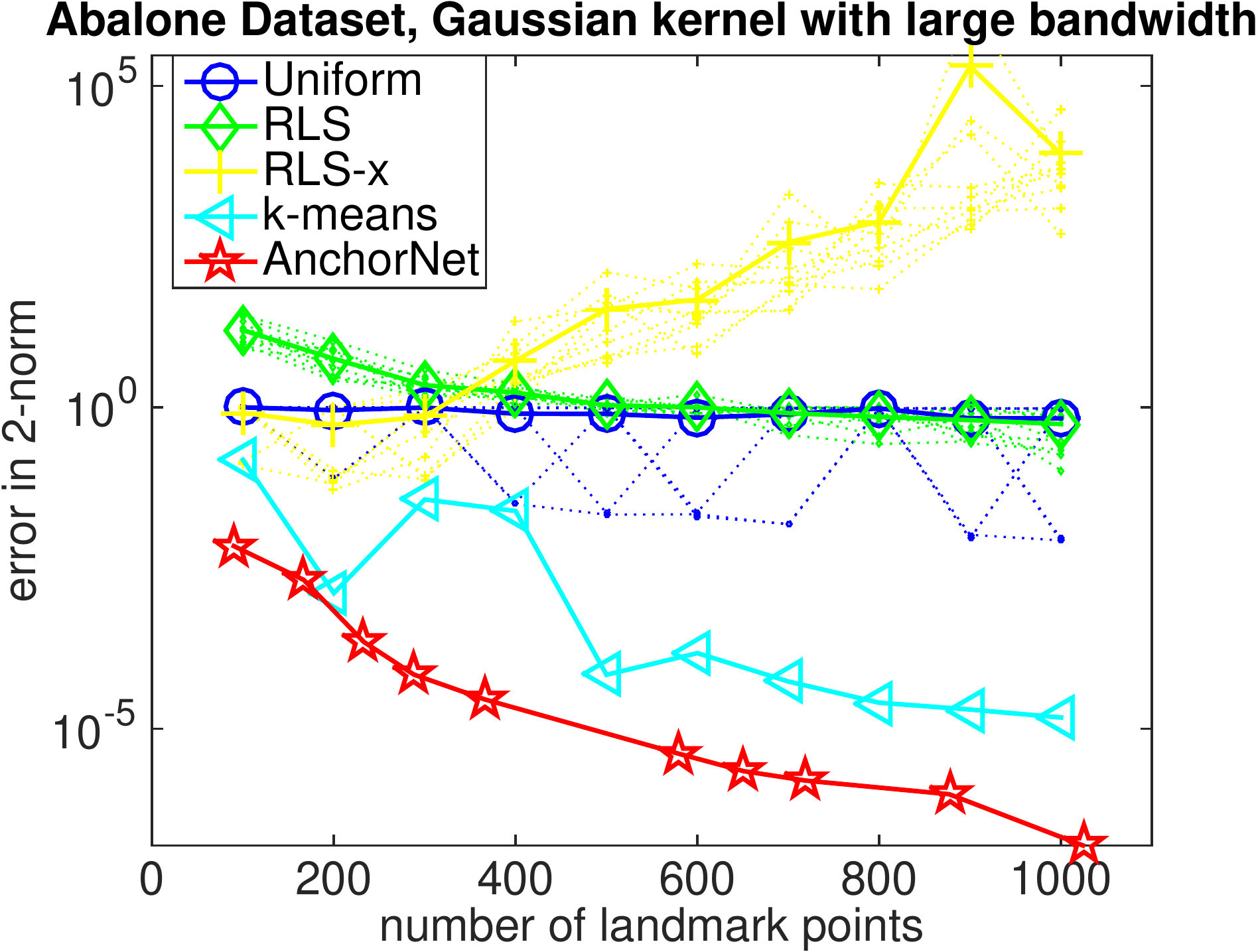}
    \includegraphics[scale=.22]{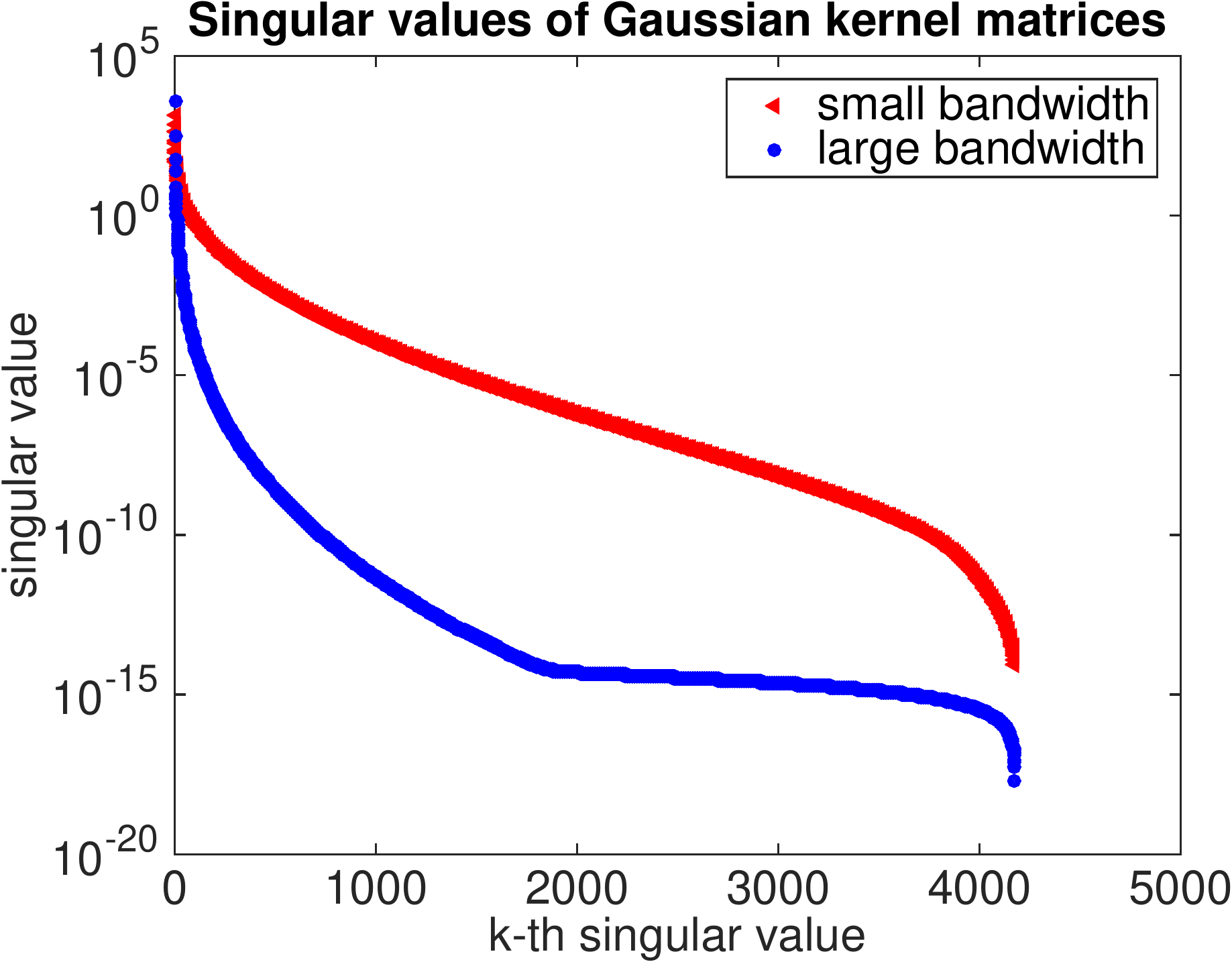}
    \caption{{Error-Rank plot for approximating Gaussian kernel matrix with $\sigma=2.3$ (left) and $\sigma=11.8$ (middle) and singular values of the two kernel matrices (right)}.}
    \label{fig:AbaGS}
\end{figure}
\begin{figure}[htbp] 
    \centering
    \includegraphics[scale=.22]{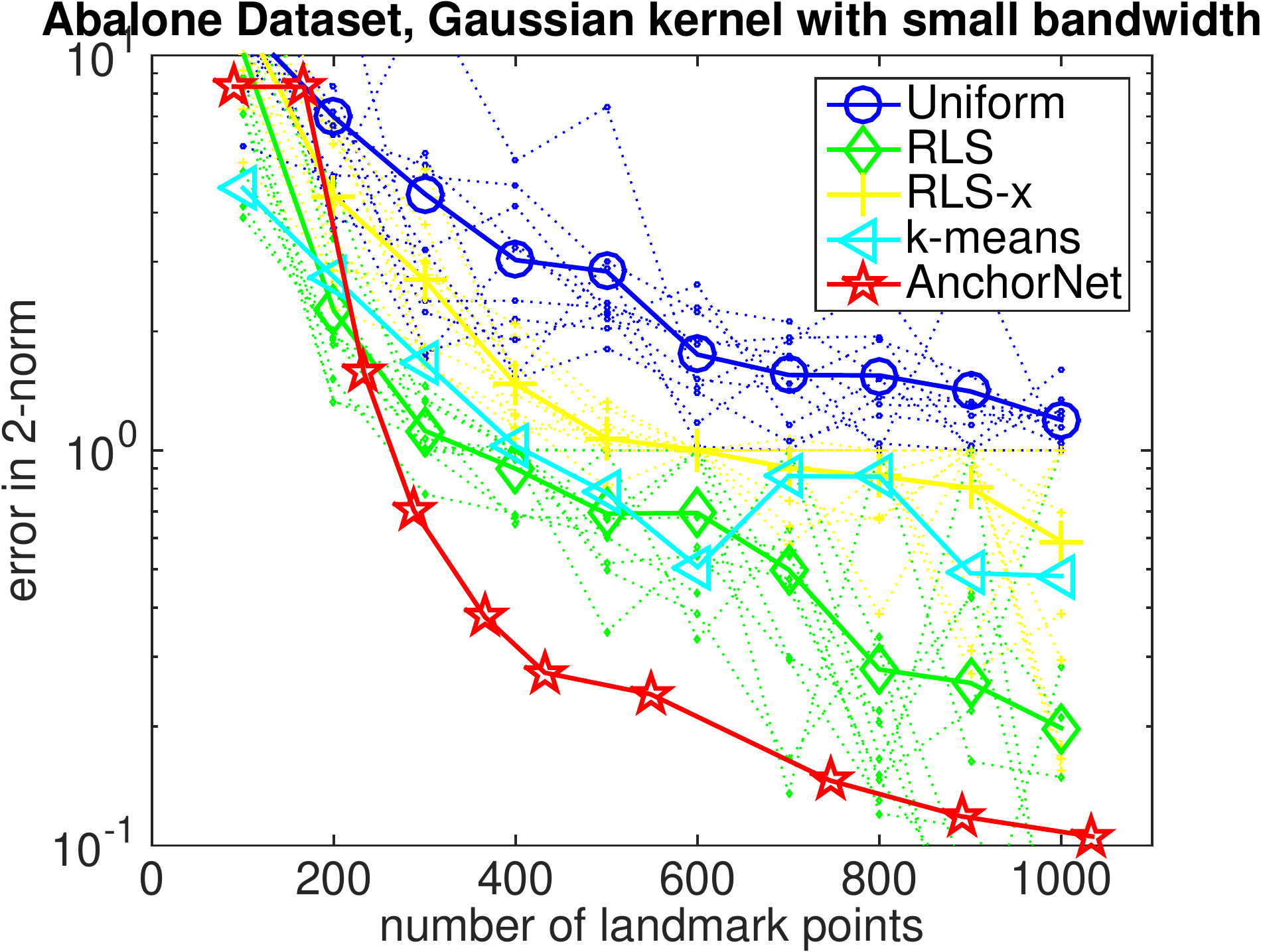}
    \includegraphics[scale=.22]{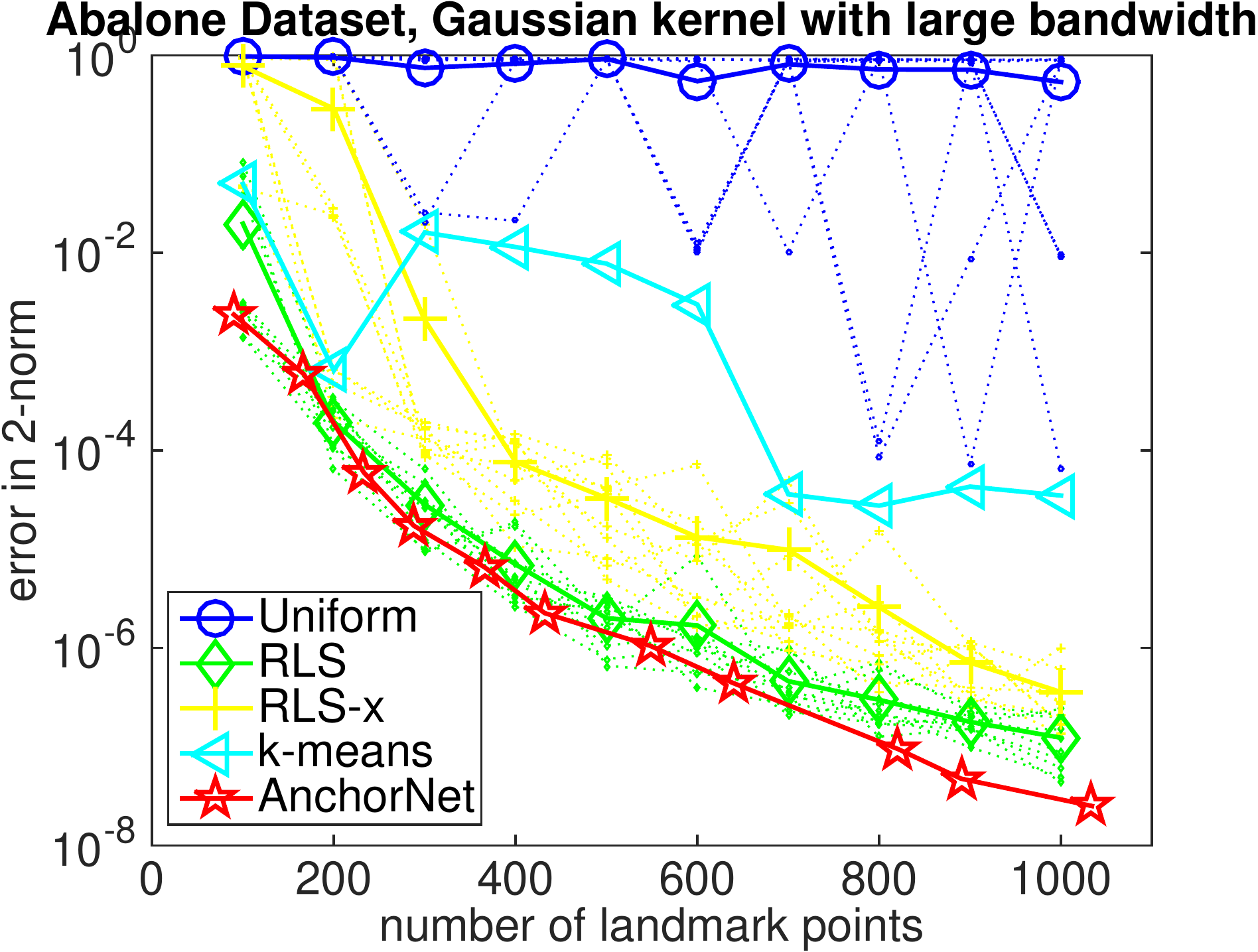}
    \caption{{Error-Rank plot for approximating regularized Gaussian kernel matrix with $\sigma=2.3$ (left) and $\sigma=11.8$ (right).}}
    \label{fig:AbaGSreg}
\end{figure}

\begin{figure}[htbp] 
    \centering
    \includegraphics[scale=.22]{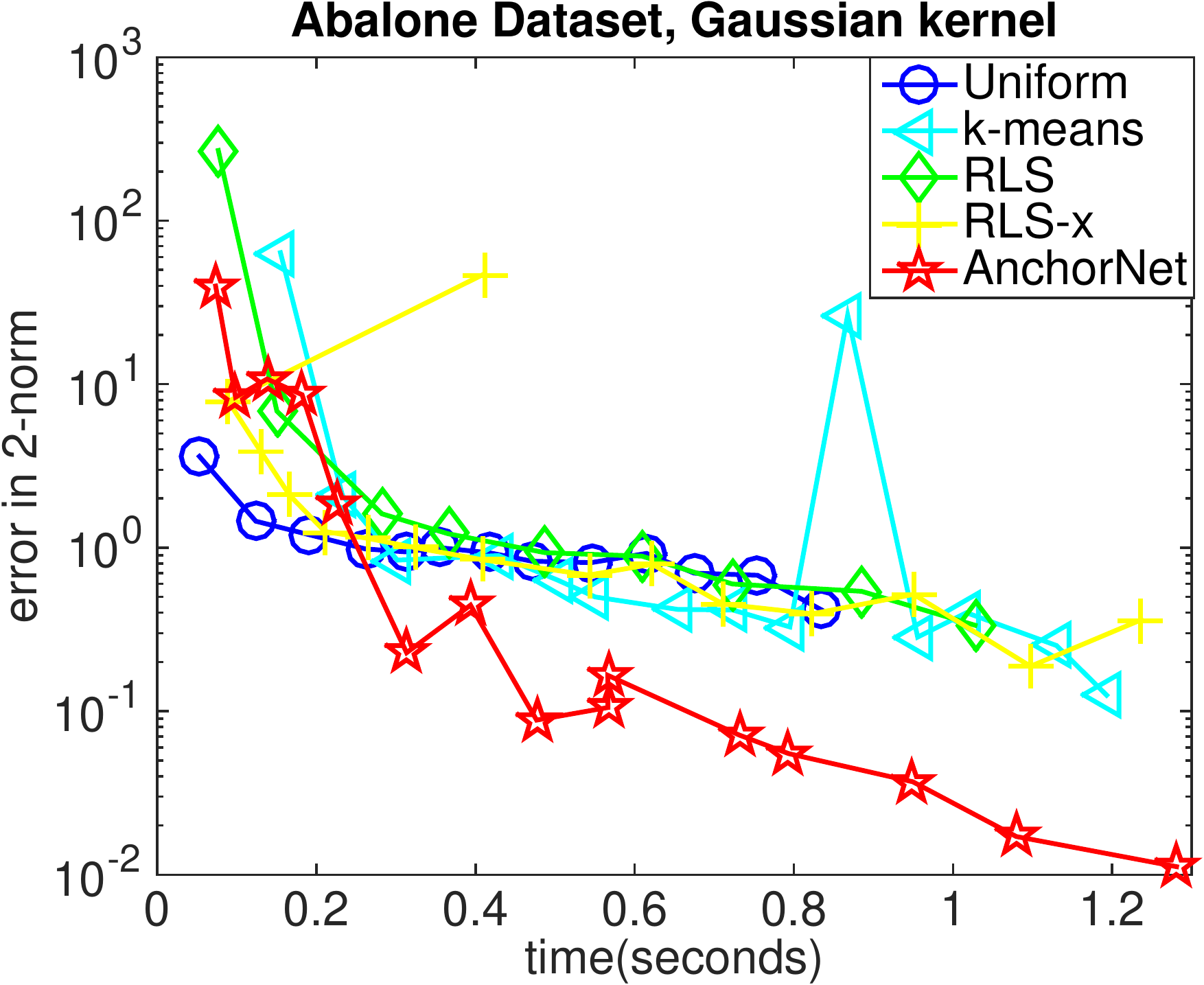}
    \includegraphics[scale=.22]{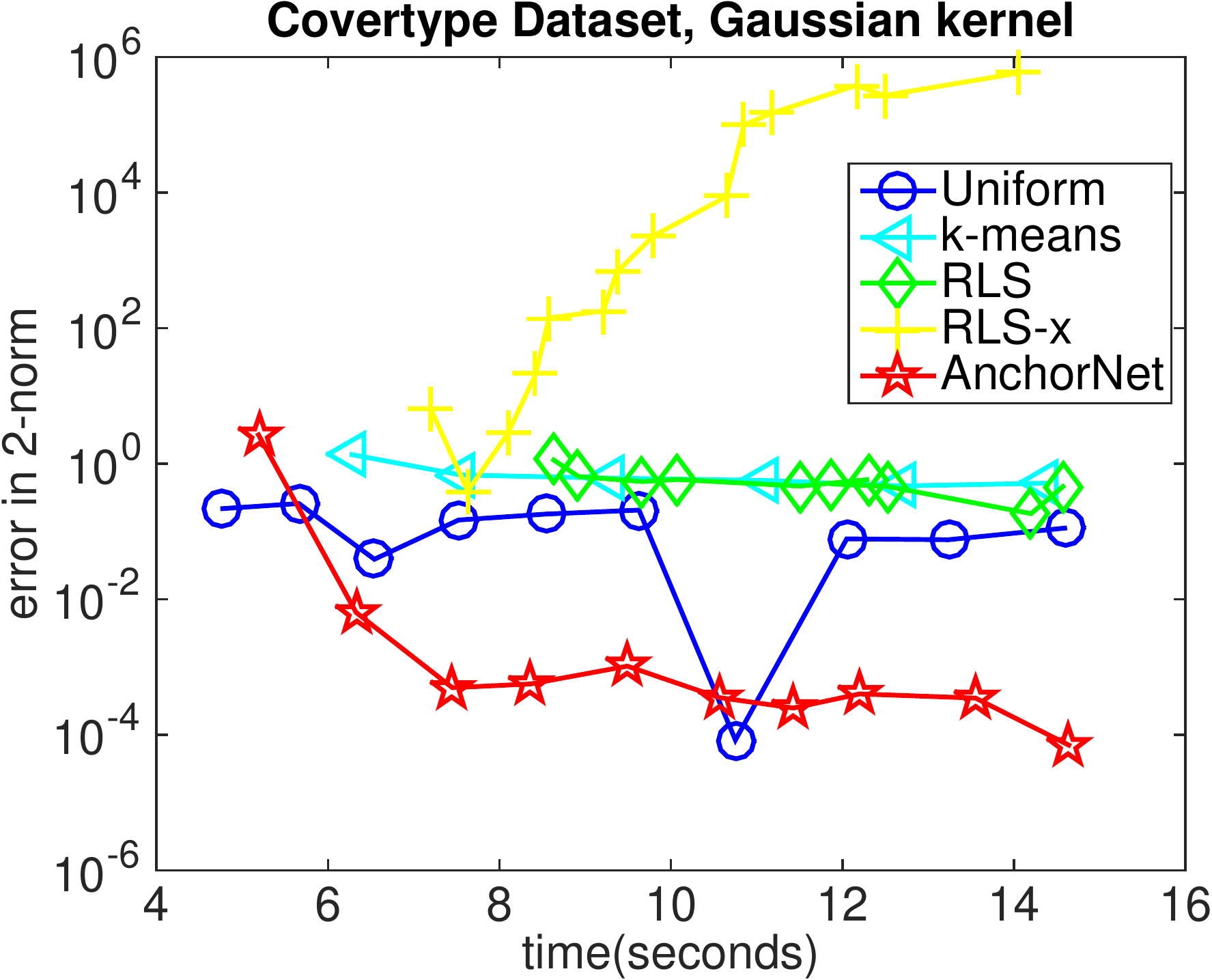}
    \caption{Error-Time plot for approximating Gaussian kernel matrix with Abalone Dataset (left) and Covertype Dataset (right).}
    \label{fig:TimeGS} 
\end{figure}

\begin{figure}[htbp] 
    \centering
    \includegraphics[scale=.22]{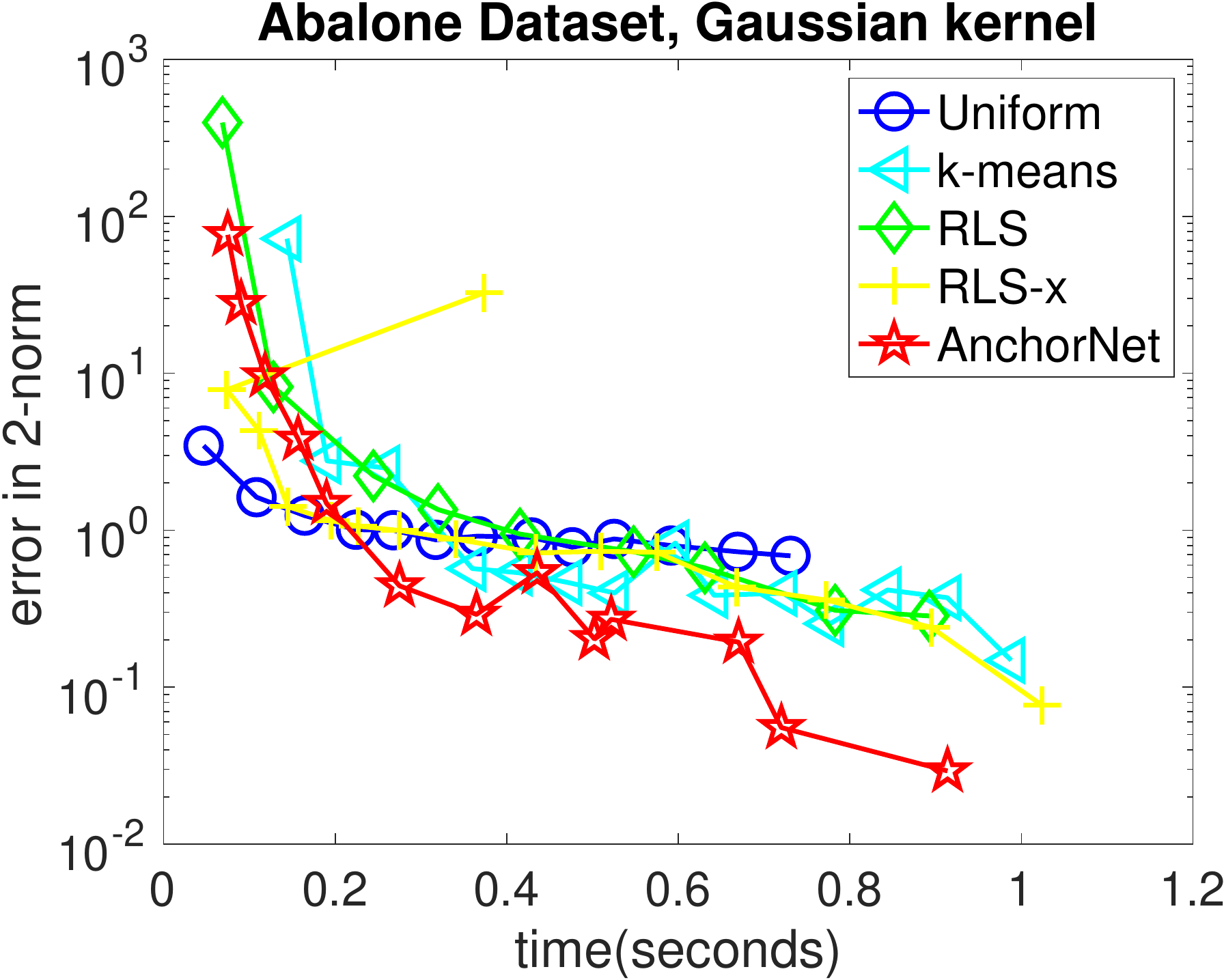}
    \includegraphics[scale=.22]{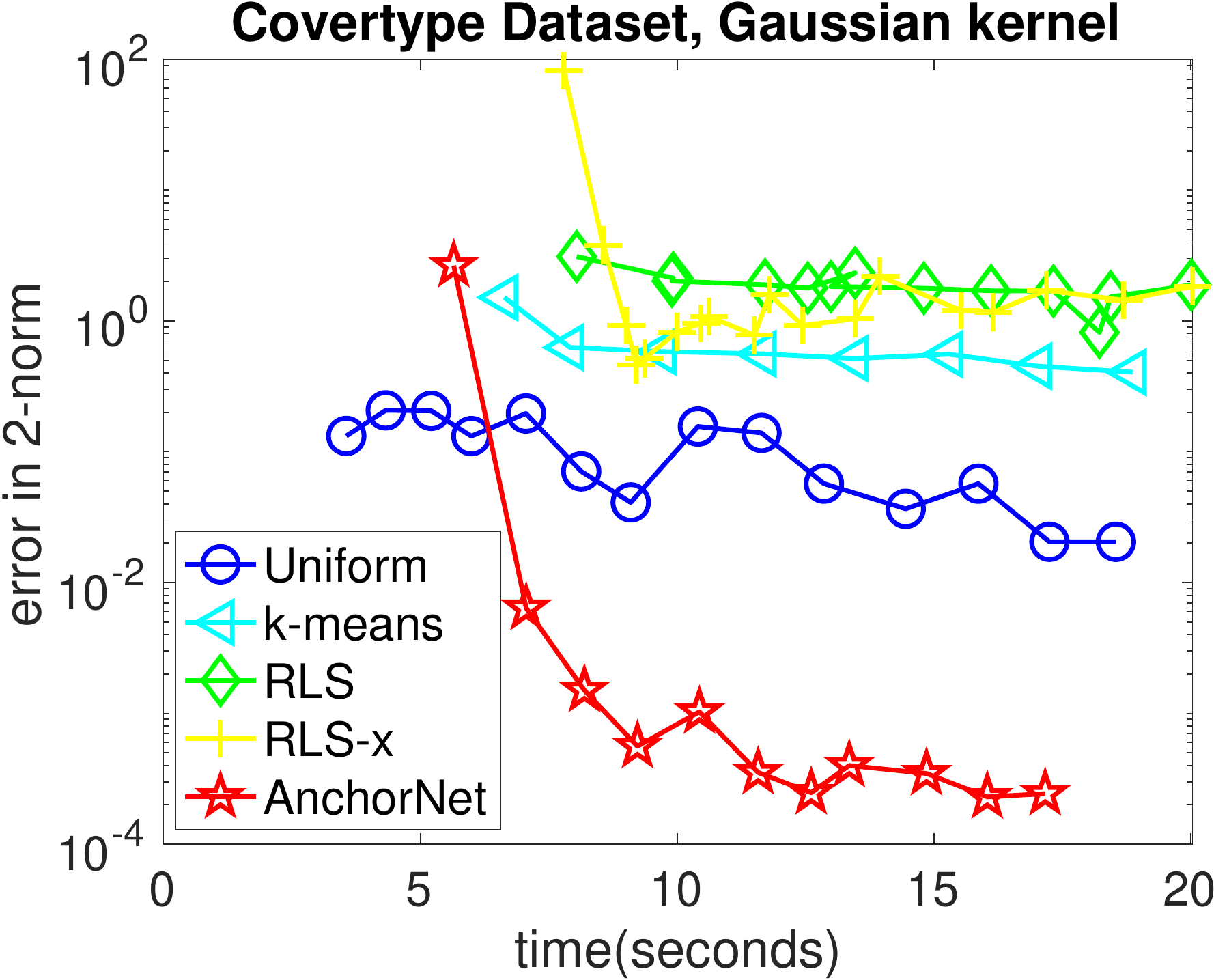}
    \caption{{Error-Time plot for approximating regularized Gaussian kernel matrix with Abalone Dataset (left) and Covertype Dataset (right)}.}
    \label{fig:TimeGSreg} 
\end{figure}

%TODO:
%\begin{enumerate}
% \item comparison of uniform anchor net (NysAN-U) and Chebyshev anchor net (NysAN-C): Chebyshev outperforms uniform for uniformly distributed data; for general case, uniform anchor net work better
%    \item Chebyshev works better for smoother kernels(lower ranks)
%    \item curse of dimensionality? comparison to $(1+p)^d$ interpolation with quasi-uniform data in $\mathbb{R}^d$
%    \item 2D uniform data: compare NysAN to unif-Nys, k-means, for uniform data. plot accuracy-time tradeoff ?
%    \item 2D non-uniform data: do the same
%    \item why choosing ``consistent" landmark pts instead of exotic ones? comparison to \nys with interpolation nodes. use non-uniform real data. comment on SKI, SKIP papers
%    \item Show $O(n)$ scaling of our method (for a fixed approximation level)
%\end{enumerate}

According to Figures \ref{fig:AbaGS} -- Figure \ref{fig:TimeGSreg}, we have the following observations.
\begin{enumerate}
    \item Overall, the anchor net method is more accurate and robust compared to other \nys methods. It achieves significantly better error-time trade-off for large scale high-dimensional datasets.
    % \item Figure \ref{fig:AbaGS}-left and Figure \ref{fig:TimeGS}-left show that, when the singular values of the Gaussian kernel matrix ($\sigma=2.3$) do not decay too fast, all methods perform well. The anchor net method achieves better accuracy.
    \item {As can be seen from Figure \ref{fig:AbaGS}-middle, for SPSD kernel matrices with rapidly decaying singular values, probabilistic methods are subject to numerical instability. 
    Via regularization, the issue can be resolved for RLS and RLS-x  but not for uniform sampling, cf. Figure \ref{fig:AbaGSreg}-right.
    The anchor net method, on the other hand, achieves best accuracy without requiring regularization.}
    %The drawback makes probabilistic methods less favourable in case of indefinite kernels, where regularization becomes ineffective. 
    \item For large scale high-dimensional datasets like Covertype, {we see from Figure \ref{fig:TimeGS} and Figure \ref{fig:TimeGSreg} that the anchor net method is able to reach high accuracy in significantly shorter time than other methods. Aside from numerical stability, this demonstrates the superior efficiency of anchor net method in practice.}
\end{enumerate}

\begin{remark}
As shown in Figure \ref{fig:AbaGS}-right, the kernel matrix with larger $\sigma$ has faster singular value decay, and consequently is more suitable for low-rank approximations. Nevertheless, it should be emphasized that better spectral property does \emph{not} necessarily imply more accurate \nys approximations. Instead, it poses a great numerical challenge for the effective use of \nys approximations: $K_{SS}$ may have many singular values near 0 and computing $K_{SS}^{+}$ will be numerically \emph{unstable} unless the landmark points $S$ are well chosen. This indicates that the \nys approximation accuracy can become even worse as the number of landmark points increases. As one can see in Figure \ref{fig:AbaGS}-middle as well as Figure \ref{fig:dk500}-right, this is indeed the case for many \nys schemes. 
%The numerical challenge calls for the study in Section \ref{sec:theory} on landmark points selection and its impact on the spectrum of the kernel matrix.
\end{remark}

\subsection{\nys methods and pivoted Cholesky factorization for SPSD matrices}
In this section, we compare $k$-means \nys method and anchor net method to partially pivoted Cholesky decomposition in \cite{chol2012}, which was shown to work well for SPSD kernel matrices associated with \emph{low} dimensional datasets.
We consider the Gaussian kernel $\kappa(x,y)=\exp(-|x-y|^2/\sigma^2)$ and form the matrix $K_{XX}$ with Abalone dataset $(n=4177,d=8)$ .
For the bandwidth parameter $\sigma$, we use three different values: $\sigma=2.3, 11.8, 47.4$ to investigate the performance of three methods.
The matrix with $\sigma=2.3$ has slowest singular value decay while the matrix with $\sigma=47.4$ has the fastest singular value decay.

{We consider approximating kernel matrices without and with regularization, i.e. $K$ and $K + \beta I$ where the regularization parameter is chosen as $\beta=10^{-9}$.
The test results are presented in Figure \ref{fig:GSchol} and Figure  \ref{fig:GScholreg}, respectively.
From Figure \ref{fig:GSchol}, we see that the performance of partially pivoted Cholesky decomposition is quite sensitive to the bandwidth parameter if no regularization is applied to $K$.
In this case, large $\sigma$ can lead to numerical instability as approximation rank increases, while small $\sigma$ can lead to slow error decay and poor approximation accuracy.
The numerical instability of partially pivoted Cholesky method is not seen when regularization is applied to $K$ according to Figure \ref{fig:GScholreg}.
}
The \nys methods achieve better accuracy than pivoted Cholesky decomposition in all cases. 
It is easy to see that the anchor net method is least sensitive to $\sigma$, achieving the best accuracy and numerical stability.
\begin{figure}
    \centering
    \includegraphics[scale=0.2]{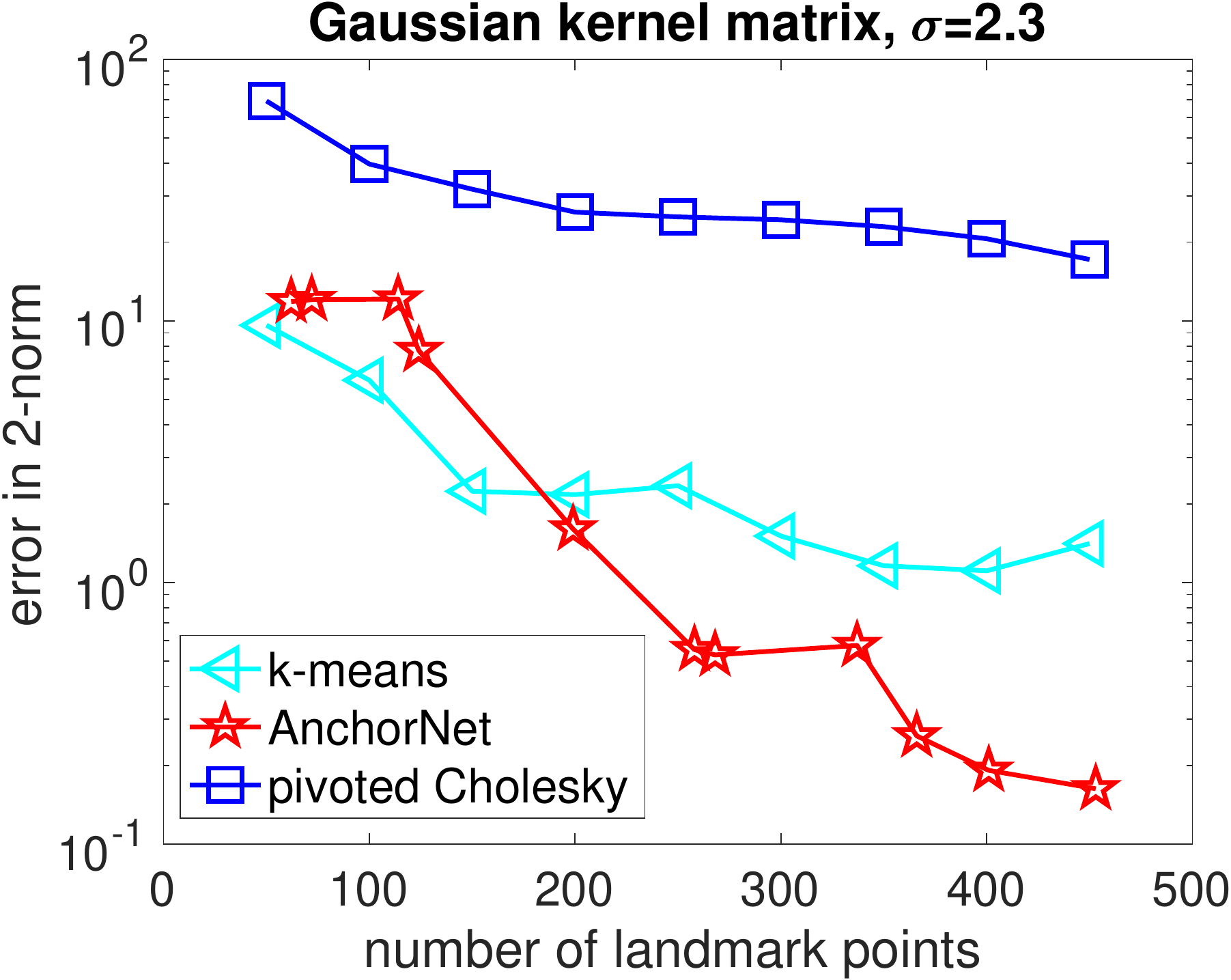}
    \includegraphics[scale=0.2]{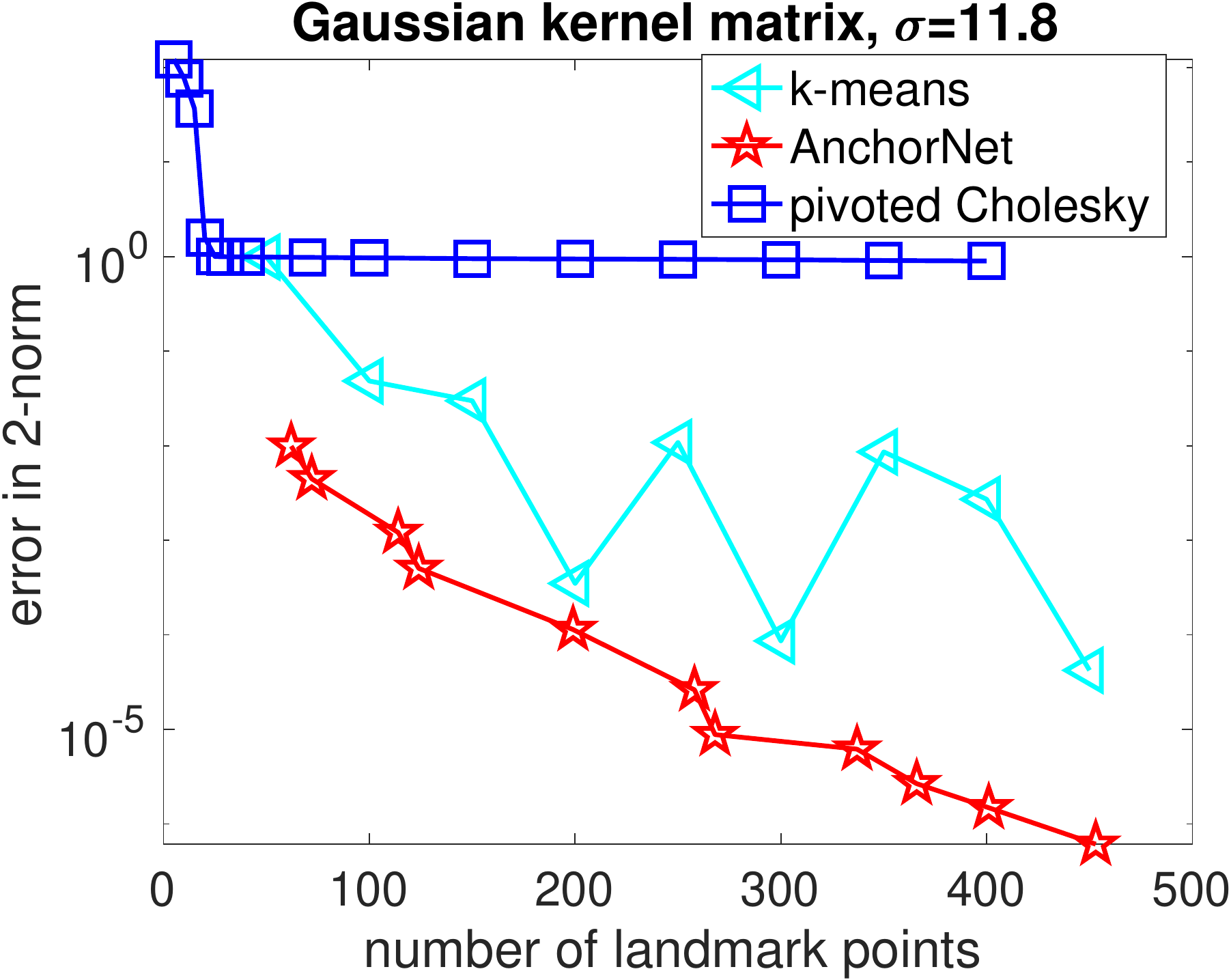}
    \includegraphics[scale=0.2]{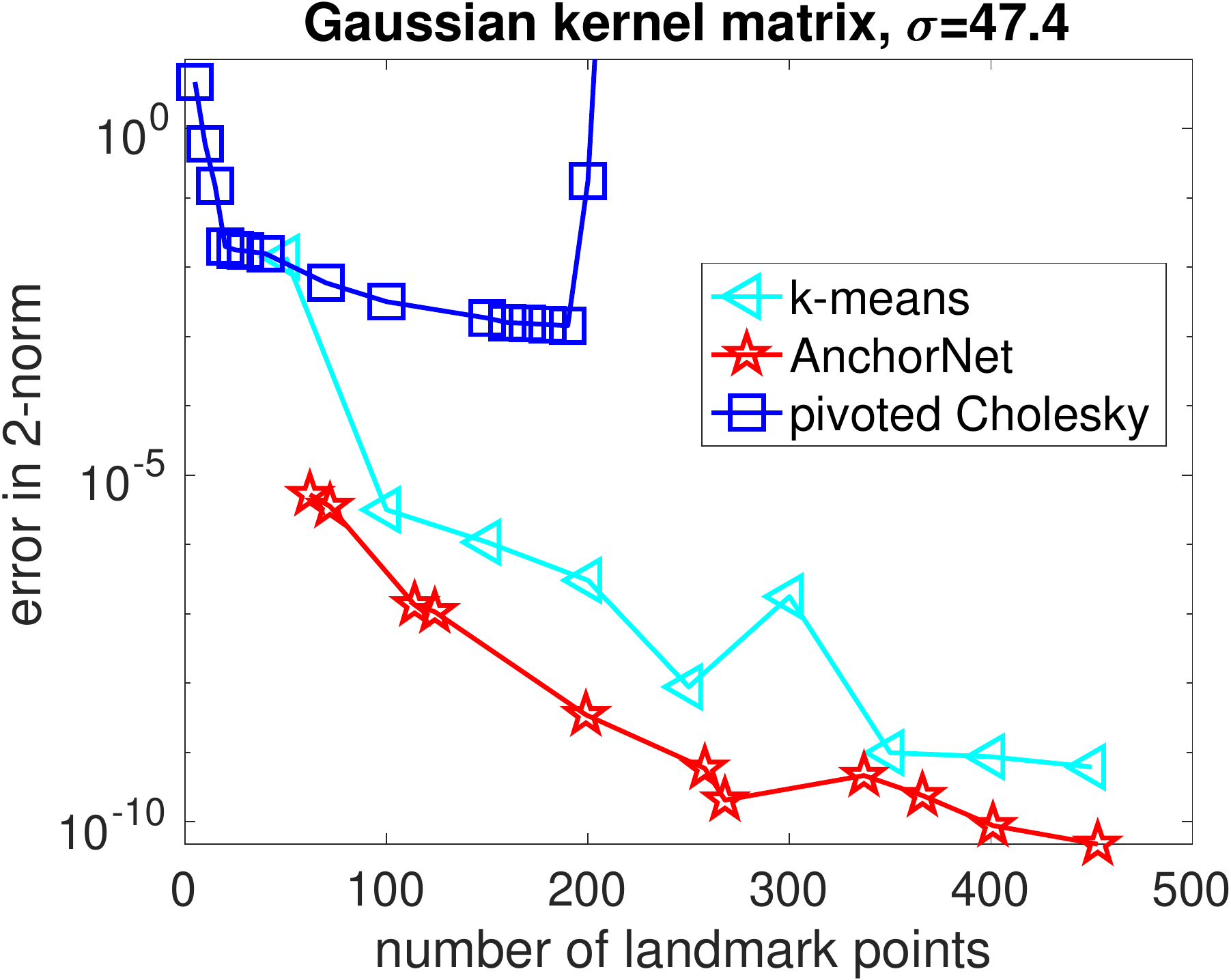}
    \caption{Approximating three Gaussian kernel matrices with bandwidths : $\sigma=2.3, 11.8, 47.4$ (left to right).}
    \label{fig:GSchol}
\end{figure}
\begin{figure}
    \centering
    \includegraphics[scale=0.2]{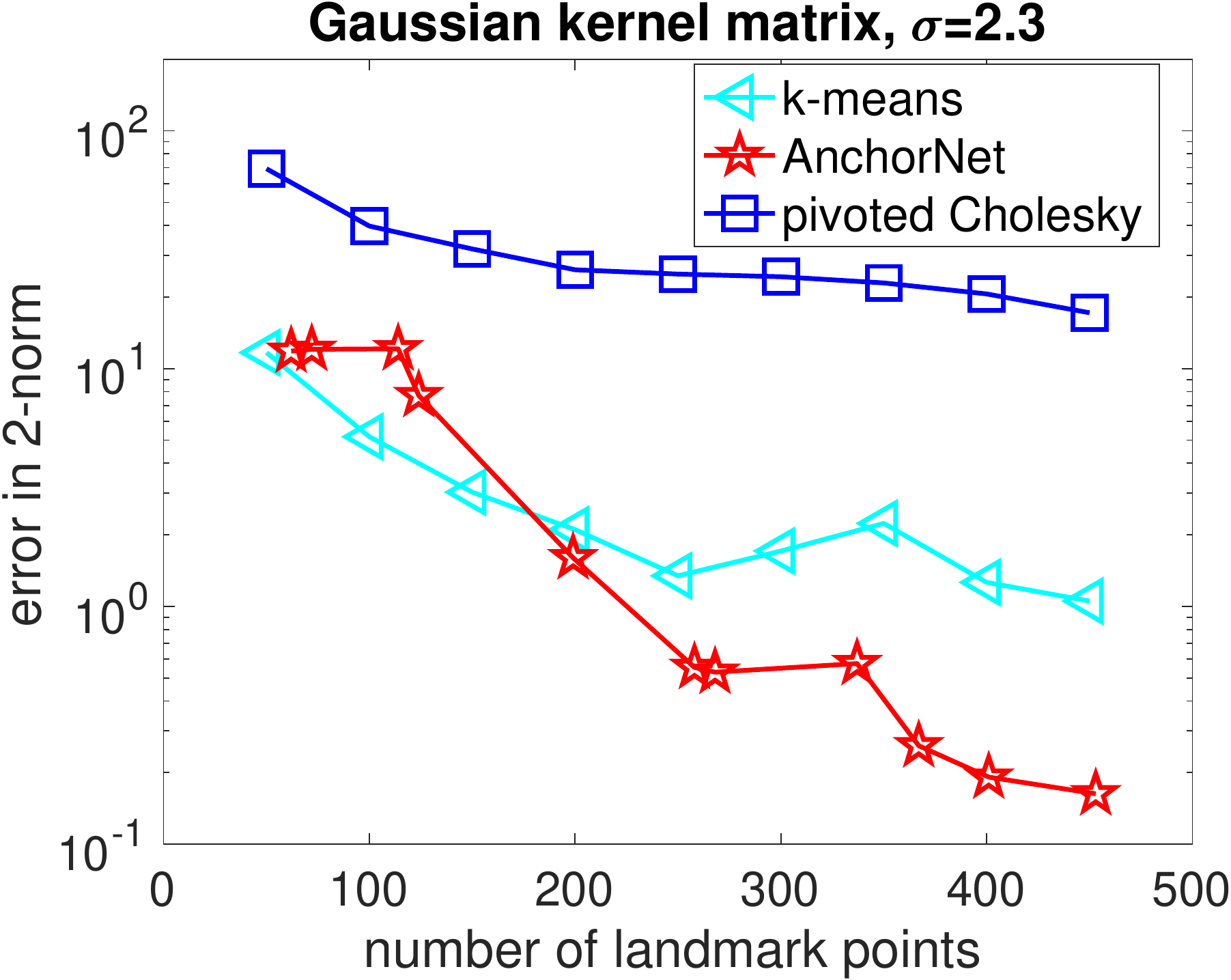}
    \includegraphics[scale=0.2]{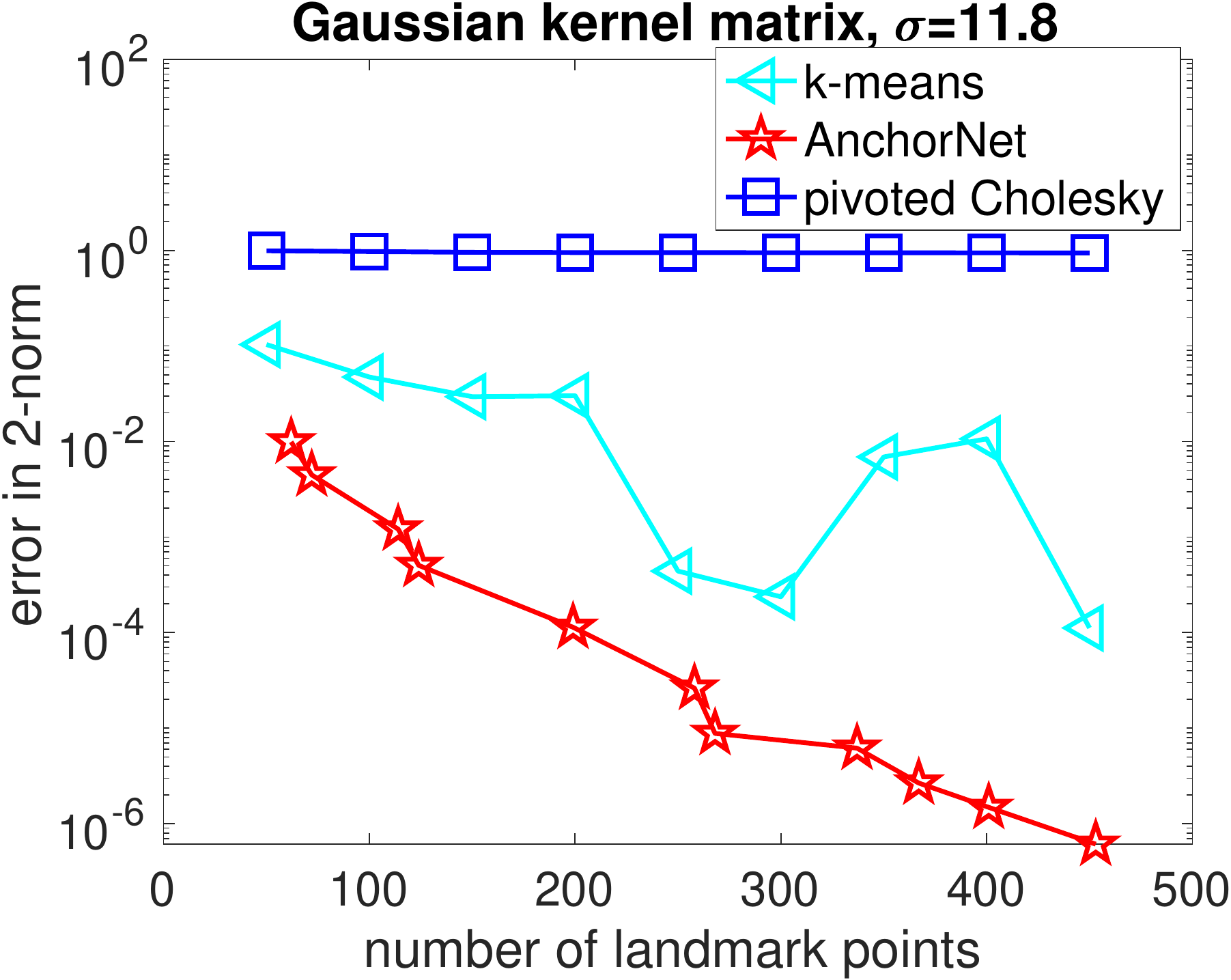}
    \includegraphics[scale=0.2]{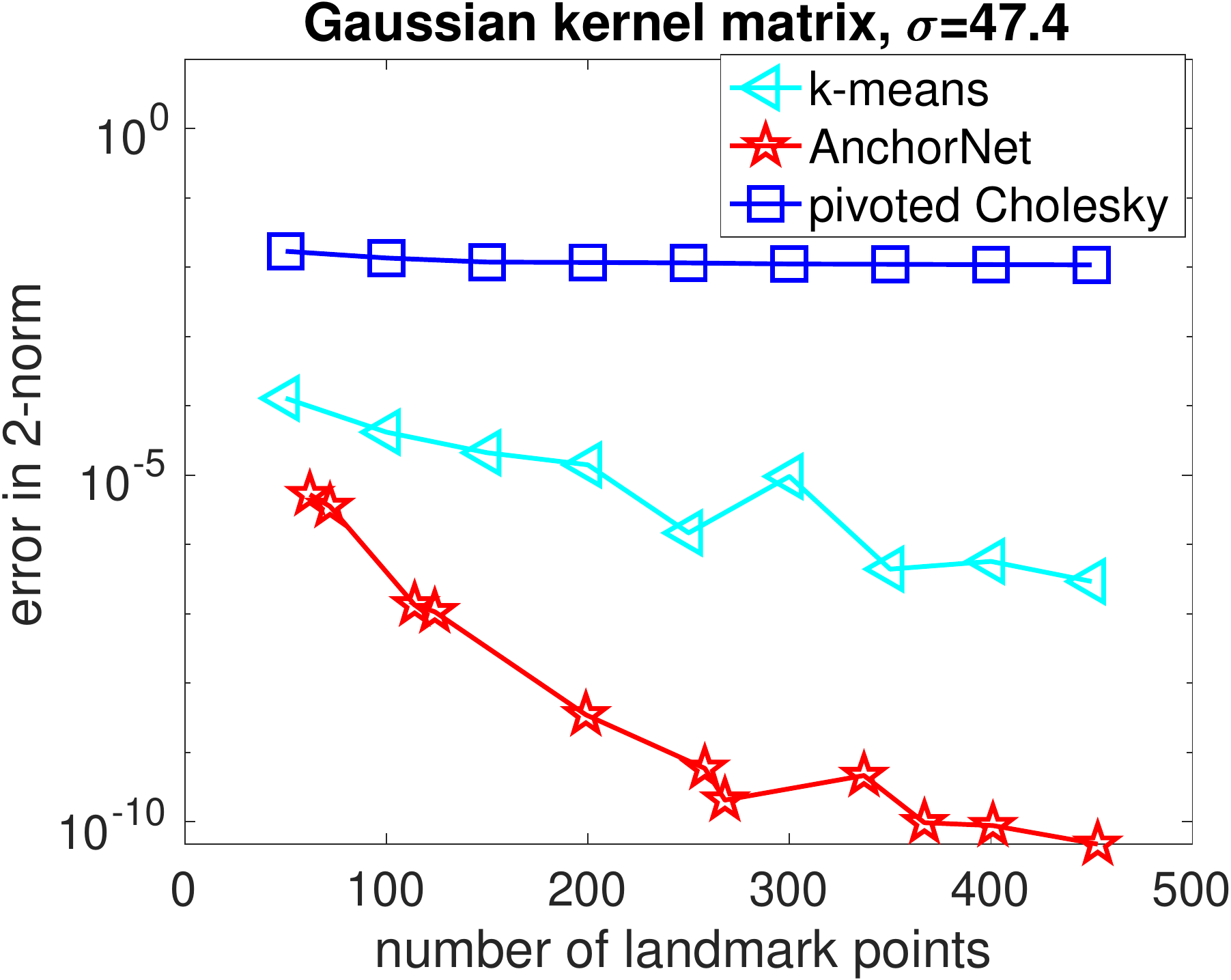}
    \caption{{Approximating regularized Gaussian kernel matrices with bandwidths : $\sigma=2.3, 11.8, 47.4$ (left to right).}}
    \label{fig:GScholreg}
\end{figure}

%\subsection{Anchor Net Condensation} 
%\label{sub:CondensationResults}
%In this section, we demonstrate that the condensed anchor net has a much smaller size than the original anchor net and is able to yield the same level of approximation accuracy.
%We consider the high-dimensional datasets used in Section \ref{sub:High Dimension}.
%To show the computational save achieved by the condensed anchor net, 
%we plot the size of an anchor net (total number of points) with respect to approximation accuracy and compare it to the original full anchor net.
%The results are presented in Figure \ref{fig:ANsize}.
%It can be seen that, to achieve a high accuracy,
%the condensed anchor net requires a much smaller size compared to the full anchor net.
%\cdf{Moreover, we also show that the condensed anchor net outperforms the full anchor net for the Wine Quality dataset.? error-rank plot???}

% subsection Anchor Net Condensation (end)

%    \includegraphics[scale=.3]{./src/WineT-FroNysAn.eps}
%    \caption{Error vs Time for Wine Quality dataset. Left: three methods; Right: new method}

%\begin{figure}[htbp]
%    \centering
%    \includegraphics[scale=.45]{./src/abalone_frogsAN.eps}
%    \includegraphics[scale=.45]{./src/cov_wineAN.eps}
%    \caption{Size of full(blue) or condensed(red) anchor nets vs approximation error. Datasets: top-left: Abalone; top-right: Anuran; bottom-left: Covertype; bottom-right: Wine Quality}
%    \label{fig:ANsize}
%\end{figure}

%%%%%%%%%%%% section  (end) %%%%%%%%%%%%%%%%

\section{Conclusion}
\label{sec:conclusion}
In this paper, we first analyze the \nys approximation error in the most general setting covering both symmetric positive semi-definite (SPSD) and indefinite kernel matrices. %A universal, computable \nys error estimate is derived for general symmetric kernels and can be applied to almost all variants of the \nys method.
%The results show that indefinite kernel matrices are much harder to approximate numerically than SPSD kernel matrices and existing methods display numerical instability for many kernels and datasets. To address the issues for general symmetric kernels, 
The theoretical finding indicates that landmark points should encode the geometry of the dataset to avoid numerical instability and meanwhile to improve the approximation accuracy.
Guided by the theoretical results, we propose the anchor net method for performing \nys approximation with linear complexity in time and space.
%The numerical stability of different \nys schemes is studied in various experiments.
%and the associated submatrix should admit a large numerical rank. 
%We also show the link between the data geometry and the kernel matrix rank, which serves as  used as the guideline on the selection of landmark points.
%Based on the theoretical findings, 
%and the efficient low-rank approximation of kernel matrices.
%associated with dataset in both low and high dimensions.
The proposed method is valid for both SPSD and indefinite kernels and is efficient in high dimensions.
%Computationally, it operates entirely on the dataset and has the optimal computational complexity.
%as it operates entirely on the dataset without accessing the kernel matrix or its matrix-vector multiplication.
%Meanwhile, it is able to achieve better accuracy without any breakdown. 
Comprehensive experiments covering indefinite and SPSD kernels, low and high dimensional data, original and stabilized \nys approximations, are performed to investigate the performance of existing methods in terms of accuracy, numerical stability, and speed.
%The state-of-the-art probabilistic methods and clustering methods are included in the comparison.
%The results show that the \nys method does work for indefinite kernel matrices but is in general numerically less stable compared to approximating SPSD kernel matrices unless the landmark points are well chosen.
%Uniform sampling and $k$-means \nys methods can be applied to indefinite kernels but may not yield a satisfactory accuracy. For SPSD kernel matrices with rapidly decaying singular values, probabilistic methods are in general not as numerically stable as $k$-means and anchor net methods.
Overall, the anchor net method displays the best numerical stability and computational efficiency. It is able to achieve better accuracy than other \nys schemes with smaller computational costs and demonstrate excellent accuracy and numerical stability for indefinite kernels compared to other methods with stabilized techniques. 
%We plan to extend these techniques to the case when the kernel matrix is associated with two sets of points and . 
{We plan to integrate the method into the computation of hierarchical matrices \cite{hack2015book,DBLP:journals/nm/Bebendorf00,bauer2020kernelindependent,smash}, which will significantly extend the scope of applications.}

\section*{Acknowledgments}
The authors are indebted to Michele Benzi for his suggestion on improving the presentation of the theoretical analysis and  Yuji Nakatsukasa for the helpful discussion on the stable implementation of pseudoinverse.

%%%%%%%%%%%% section  (end) %%%%%%%%%%%%%%%%

%\section{Future Work} 
%\label{sec:future}
%We see that the technique of using anchor nets resolves the inability of existing approaches to achieve high accuracy for high-dimensional dataset. 
%In general, one challenge for high-dimensional problems is that the spectral property of the kernel matrix is not as appealing as the low dimensional counterpart. 
%However, thanks to the product nature of the Gaussian kernel, 
%one promising way to overcome the curse of dimensionality is to exploit better spectral properties via Hadamard product and dimension splitting. 
%The figure shows that better spectral property arises, 
%i.e., faster singular value decay, after the dimension splitting. 
%Future work includes incorporating the above mentioned development, i.e., anchor net-based sampling method and dimension splitting techniques, 
%in SMASH \cite{smash} to attack high-dimensional problems.
%
%\begin{figure}[htbp] 
%    \centering 
%    \includegraphics[scale=.5]{./src/productS2.eps}
%    \caption{Singular values for three kernel matrices associated with $d$ dimensional Abalone data ($d=8,d=4,d=4$)}
%    \label{fig:productSVD}
%\end{figure}
%
\bibliography{nystrom_arxiv}
\bibliographystyle{siamplain}
\end{document}